\documentclass{article}

\usepackage[twoside,
paperwidth=210mm,
paperheight=297mm,
textheight=622pt,
textwidth=468pt,
centering,
headheight=50pt,
headsep=12pt,
footskip=18pt,
footnotesep=24pt plus 2pt minus 12pt,
columnsep=2pc]{geometry}

\usepackage[auth-sc]{authblk}

\usepackage{mathtools}

\allowdisplaybreaks

\usepackage{amssymb}
\usepackage[mathscr]{eucal}

\usepackage{mismath}

\usepackage{hyperref}

\usepackage{bm} % Bold math font.

\usepackage{graphicx}
\graphicspath{{figs/}}

\usepackage{tikz}
\usetikzlibrary{math}
\usetikzlibrary{calc}

\usepackage{subcaption}

\usepackage{multirow}

\usepackage{makecell}
\setcellgapes{2pt}

\usepackage{booktabs}

\usepackage{siunitx}
\usepackage{algorithm}
\usepackage{algpseudocode}

\usepackage[skins]{tcolorbox}
\newtcolorbox{stepbox}[2][]{%
  enhanced,
  attach boxed title to top center={yshift=-3mm,yshifttext=-1mm},
  colframe=blue!75!black,
  colbacktitle=red!80!black,
  fonttitle=\bfseries,
  title=#2,#1
}

\usepackage{amsthm}

\usepackage[capitalise]{cleveref}
\newtheorem{theorem}{Theorem}[section]
\newtheorem{lemma}[theorem]{Lemma}
\newtheorem{remark}[theorem]{Remark}

\theoremstyle{definition}

\newtheorem{assumption}{Assumption}

\theoremstyle{remark}

\providecommand{\keywords}[1]{\textbf{\textit{Keywords---}} #1} 

\crefname{equation}{}{}

\usepackage{lineno}
\usepackage[numbers,sort&compress]{natbib}

\numberwithin{figure}{section}
\numberwithin{table}{section}

\title{A locking free multiscale method for linear elasticity in stress-displacement formulation with high contrast coefficients}
\author[1]{Eric T. Chung}
\author[1]{Changqing Ye}
\author[1]{Xiang Zhong\thanks{Corresponding author.
(Email address: \href{mailto:xzhong@math.cuhk.edu.hk}{xzhong@math.cuhk.edu.hk})}}
\affil[1]{Department of Mathematics, The Chinese University of Hong Kong, Shatin, Hong~Kong~SAR, China.}
\date{}
\begin{document}
\maketitle

\begin{abstract}
  Achieving strongly symmetric stress approximations for linear elasticity problems in high-contrast media poses a significant computational challenge. Conventional methods often struggle with prohibitively high computational costs due to excessive degrees of freedom, limiting their practical applicability. To overcome this challenge, we introduce an efficient multiscale model reduction method and a computationally inexpensive coarse-grid simulation technique for linear elasticity equations in highly heterogeneous, high-contrast media. We first utilize a stable stress-displacement mixed finite element method to discretize the linear elasticity problem and then present the construction of multiscale basis functions for the displacement and the stress. The mixed formulation offers several advantages such as direct stress computation without post-processing, local momentum conservation (ensuring physical consistency), and robustness against locking effects, even for nearly incompressible materials. Theoretical analysis confirms that our method is inf-sup stable and locking-free, with first-order convergence relative to the coarse mesh size. Notably, the convergence remains independent of contrast ratios as enlarging oversampling regions. Numerical experiments validate the method’s effectiveness, demonstrating its superior performance even under extreme contrast conditions.
\end{abstract}

 % 填写你的关键词

\section{Introduction}
Multiscale modeling plays an important role in computational mechanics, with applications ranging from structural analysis to geomechanics. However, resolving material heterogeneities at fine scales often leads to prohibitively expensive computations, particularly when addressing problems with complex microstructures. To address this challenge, multiscale model reduction techniques have emerged as a powerful framework, enabling efficient simulation by systematically capturing microscale features while drastically reducing computational costs. Existing multiscale model reduction methodologies include multiscale finite element methods \cite{YTY2009,HTY1997}, localized orthogonal decomposition method \cite{AMDP2014}, variational multiscale finite element methods \cite{HFJL1998}, numerical upscaling \cite{PDR2016}, heterogeneous multiscale methods \cite{EW2003,EWBE2003,BEYH2005} and so on. Recent advances have further expanded the scope of multiscale methods to address diverse challenges in real-world applications (see for instance, \cite{Eric2018, EricCS2018, EricSM2020, EricFu2023, EricGao2015, Ericwang2024, EricYE2016, EricYETY2016, EricYEFS2014, EricYELee2015, Eric_mix_2018, YeEric2023, LZET2023, FEW2020,YeJinEric2023,XLXE2025,EricYETY2023,FEL2024, YefuEric2023}).

In context of linear elasticity, we carefully concern two crucial aspects: the accurate representation of stress-strain relationships across multiple scales, and the preservation of symmetry conditions in stress approximations. 
There are numerous studies devoted to solving the linear elasticity problem. This problem becomes harder when the Lam$\rm \acute{e}$ coefficients are heterogeneous and with high contrast. Additionally, it is crucial to address the locking effect related to Poisson's ratio. A particular interest in the multiscale finite element analysis of elasticity can be found in \cite{EricFu2023, EricGao2015, Ericwang2024, EricYE2016, EricYEFS2014, FEW2020}. More precisely, \cite{EricFu2023} considers the contrasts based on the constraint energy minimizing generalized multiscale finite element method (CEM-GMsFEM) with discontinuous Galerkin coupling; \cite{EricGao2015} develops fast yet accurate full wavefield modeling method for elastic wave propagation
in heterogeneous, anisotropic media; \cite{Ericwang2024} studies the CEM-GMsFEM to solve the complex elastic PDEs with mixed inhomogeneous boundary conditions. 
However, most existing multiscale methods rely on standard $H^1$-elliptic displacement discretization, which can lead to numerically unstable stress estimates for nearly incompressible materials due to locking effects. Mixed methods offer a robust alternative, as they inherently avoid locking and provide accurate stress approximations independent of Poisson’s ratio.
Mixed multiscale methods are developed in \cite{EricCS2018, EricSM2020, EricYELee2015, Eric_mix_2018}. The basic idea of mixed generalized multiscale ﬁnite element method (GMsFEM) \cite{EricYELee2015} is to construct multiscale basis functions following a multiscale finite element framework and couple these basis functions using a mixed finite element method. A novel mixed CEM-GMsFEM is proposed in \cite{Eric_mix_2018} based on \cite{EricYELee2015}.  Following \cite{EricYELee2015}, \cite{EricCS2018} builds a multiscale coarse space for the stress with enforcement of strong symmetry. The coarse space for the stress is the union of an edge-based space and a vertex-based space. \cite{EricSM2020} develops the mixed CEM-GMsFEM for the first-order wave equation 
based on the idea of proposed in \cite{Eric_mix_2018}. 

Building upon the aforementioned research, we propose a locking free multiscale model reduction method for linear elasticity in stress-displacement formulation with high contrasts. Our approach extends the CEM-GMsFEM \cite{Eric2018} to a mixed formulation designed on special grids (see Figure \ref{coarse_bigger} and \cite{Zhong_2023}). This mixed framework ensures strongly symmetric stress approximations and robust handling of nearly incompressible materials, effectively circumventing locking effects.
More precisely, we focus on constructing a mixed version of the CEM-GMsFEM for linear elasticity in highly heterogeneous and high-contrast media. The computation of multiscale basis functions is carried out in two stages.
In the first stage, we construct auxiliary multiscale basis functions for the displacement on each coarse element by solving local spectral problems. These basis functions form an auxiliary multiscale space, which serves as the approximation space for the displacement field.
In the second stage, we utilize this auxiliary multiscale space to construct a multiscale space for the stress field. Notably, the displacement basis functions are locally supported (i.e., their supports coincide with the coarse elements), whereas the stress basis functions are supported on oversampled regions encompassing the displacement basis supports.
By appropriately selecting the number of oversampling layers, we achieve first-order convergence with respect to the coarse mesh size, independent of the local contrast. Moreover, our method maintains stability and convergence even in scenarios involving extremely high contrasts and nearly incompressible materials. We provide detailed theoretical analyses and numerical experiments to validate our approach’s efficacy across challenging scenarios.

Our main contributions are threefold. Firstly, it is well known that achieving strongly symmetric stress approximation is not easy. While traditional methods \cite{adam2002,arnold2002,arnold2008,DFLS2008, CB1978, BJ1968} require prohibitively many degrees of freedom to achieve strongly symmetric stress approximations, our multiscale model reduction technique dramatically reduces computational complexity while maintaining solution accuracy.
Secondly, building upon previous works \cite{EricFu2023, EricGao2015, Ericwang2024, EricYE2016, EricYEFS2014, FEW2020}, we present a mixed multiscale method for solving linear elasticity problems with highly heterogeneous and high-contrast coefficients. Our mixed formulation guarantees that the numerical approximation of stress remains locking-free with respect to Poisson's ratio. Lastly, compared with the previous work of the mixed GMsFEM for the planar linear elasticity \cite{EricCS2018}, we consider the problem in both two- and three dimensions and establish coarse-mesh-dependent convergence without using many basis functions. This is achieved by following the CEM-GMsFEM framework and developing a novel mixed CEM-GMsFEM specifically designed for linear elasticity equations in highly heterogeneous, high-contrast media.

We point out that for the reference solution, we implement a mixed finite element discretization using composite triangular elements. Following established methodologies \cite{DFLS2008, CB1978, BJ1968}, each element is subdivided into three subtriangles by connecting its barycenter to the vertices (see Figure \ref{coarse_bigger}). This composite element approach ensures strongly symmetric stress approximations and serves as the theoretical foundation for all fine-scale analyses in this study.
While this conventional mixed finite element method provides mathematically rigorous solutions, it suffers from computational inefficiency due to its high degrees of freedom. To address this practical concern in numerical simulations, we employ techniques from an equivalent mixed elements \cite[Section 6]{JJ2012} that maintains exact stress symmetry while significantly reducing the system size. This mixed approach also employs the composite elements (known as Hseih-Clough-Tocher grids, see \cite[Section 6]{JJ2012}) and has been successfully implemented in \cite{Zhong_2023}.

This paper is organized as follows. In Section \ref{sec: Preliminaries}, we present the model problem and introduce some notations. The construction of multiscale basis functions in the proposed method is described in Section \ref{Construction of the approximation space}. Detailed analyses for the stability and convergence are presented in Section \ref{Analysis for stability and convergence}. To validate the performance of the proposed method, we report some numerical experiments conducted on two different test models in Section \ref{Numerical experiments}. Finally, in Section \ref{conclusions}, we conclude the paper.

\section{Preliminaries}\label{sec: Preliminaries}
In this section, we first present the linear elastic problem in heterogeneous media. Then we introduce the key definitions and notation that will be employed throughout our subsequent analysis. Finally, we describe the mixed finite element method that serves as our reference solution methodology.
\subsection{Model problem}
We consider the linear elastic problem in the domain $\Omega\subset\mathbb{R}^n$:
\begin{subequations}
	\label{our pde}
	\begin{align}		
		\mathcal{A}\underline{\bm{\sigma}}=\underline{\bm{\epsilon}}(\mathbf{u})\quad &\rm in\quad\Omega ,\label{our pde_a}\\
		\textbf{div}\underline{\bm{\sigma}}=\mathbf{f}\quad &\rm in\quad\Omega,\label{oue pde_b}\\
		\underline{\bm{\sigma}}\mathbf{n}=\mathbf{0}\quad &\rm on\quad\partial\Omega,\label{our pde_d}
	\end{align}	
\end{subequations}
where $\underline{\bm{\sigma}}:\Omega\to\mathbb{R}^{n\times n}$ is the symmetric stress tensor, $\mathbf{u}$ denotes the displacement and $\underline{\bm{\epsilon}}(\mathbf{u})=\frac{1}{2}(\nabla\mathbf{u}+(\nabla\mathbf{u})^t)$ the linearized strain tensor. $\mathbf{n}$ denotes the unit outward normal to the domain. Notice that the source term $\mathbf{f}=(f_i)_{i=1}^n$ satisfies $\int_\Omega f_idx=0$ for $i=1,...,n$.  $\Omega\subset \mathbb{R}^n$ 
($n=2,3$) is a bounded and connected Lipschitz polyhedral domain occupied by an isotropic and linearly elastic
solid. $\partial\Omega$ is the boundary of the domain $\Omega$. $\mathcal{A}$ is the inverse of the elasticity operator, which is given by
$\mathcal{A}\underline{\bm{\tau}}:=\frac{1}{2\mu}\underline{\bm{\tau}}^D+\frac{1}{n(n\lambda+2\mu)}(tr\underline{\bm{\tau}})\bm{\mathit{I}},$
where $\lambda$ and $\mu$ are the Lam$\rm \acute{e}$ coefficients. $\bm{\mathit{I}}$ is the identity matrix of $\mathbb{R}^{n\times n}$ and the deviatoric tensor $\underline{\bm{\tau}}^D:=\underline{\bm{\tau}}-\frac{1}{n}(tr\underline{\bm{\tau}})\bm{\mathit{I}}$. Let $E$ be Young's modulus and $\nu$ be Poisson's ratio, then the Lam$\rm \acute{e}$ coefficients $\lambda$ and $\mu$ are denoted as
\begin{equation}
	\label{Lame constants}
	\lambda\coloneqq\frac{E\nu}{(1+\nu)(1-2\nu)},\quad  \mu\coloneqq\frac{E}{2(1+\nu)}.
\end{equation}
For nearly incompressible materials, $\lambda$ is large in comparison with $\mu$ (More precisely, $\lambda\to\infty$ as Poisson’s ratio $\nu\to 1/2$ while $\mu$ is bounded). In this paper, $\lambda$ and $\mu$ are highly heterogeneous in space and possibly high contrast. For $n=2,3$, by using Voigt notation, the inverse of the elasticity operator can be
expressed in terms of the following coefficient matrices:
\[
\mathcal{A} = \begin{pmatrix}
\lambda+2\mu & \lambda & 0 \\
\lambda & \lambda+2\mu & 0 \\
0 & 0 & 2\mu
\end{pmatrix}^{-1}(n=2),\text{ or }
\mathcal{A} = 
\begin{pmatrix}
\lambda + 2\mu & \lambda         & \lambda         & 0    & 0    & 0    \\
\lambda         & \lambda + 2\mu & \lambda         & 0    & 0    & 0    \\
\lambda         & \lambda         & \lambda + 2\mu & 0    & 0    & 0    \\
0               & 0               & 0               & 2\mu  & 0    & 0    \\
0               & 0               & 0               & 0    & 2\mu  & 0    \\
0               & 0               & 0               & 0    & 0    & 2\mu  \\
\end{pmatrix}^{-1}(n=3).
\]
We define the following bilinear forms. For a given domain $D\subset\Omega$,
\[ 
(\mathbf{u},\mathbf{v})_D=\int_D\mathbf{u}\cdot\mathbf{v}dx,\quad (\mathcal{A}\underline{\bm{\sigma}},\underline{\bm{\tau}})_D=\int_D\mathcal{A}\underline{\bm{\sigma}}\colon\underline{\bm{\tau}}dx.
\]
Their respective norms are defined to be 
\[
\norm{\mathbf{u}}_{L^2(D)}=(\mathbf{u},\mathbf{u})_D^{\frac{1}{2}},\quad \norm{\underline{\bm{\sigma}}}_{\mathcal{A}(D)}=(\mathcal{A}\underline{\bm{\sigma}},\underline{\bm{\sigma}})_D^{\frac{1}{2}}.
\]
We will drop the subscript $D$ when $D=\Omega$. It is clear that $\norm{\underline{\bm{\sigma}}}_{\mathcal{A}(D)}\leq C\norm{\underline{\bm{\sigma}}}_{L^2(D)}$, where $C$ does not depend on $\lambda$ but only depends on the lower bound of $\mu$.

Let $\underline{\bm{\mathit{Y}}}\coloneqq\{\underline{\bm{\tau}}\in\underline{\bm{\mathit{H}}}(\textbf{div};\Omega,\mathbb{S})\colon\underline{\bm{\tau}}\mathbf{n}=\mathbf{0}\hspace{0.5em} {\rm on\hspace{0.5em}\partial\Omega}\},$ a closed subspace of $\underline{\bm{\mathit{H}}}(\textbf{div};\Omega,\mathbb{S})$. The weak form of problem (\ref{our pde}) is to find $(\underline{\bm{\sigma}},\mathbf{u})\in \bm{\underline{H}}(\textbf{div};\Omega,\mathbb{S})\times \bm{\mathit{L}}^2(\Omega)$ such that 
\begin{subequations}
	\label{weak form for our pde}
	\begin{align}		(\mathcal{A}\underline{\bm{\sigma}},\underline{\bm{\tau}})+(\textbf{div}\underline{\bm{\tau}},\mathbf{u})&=0  \quad \forall \underline{\bm{\tau}}\in \underline{\bm{\mathit{Y}}}, \label{weak form for our pde_a}\\
		(\textbf{div}\underline{\bm{\sigma}},\mathbf{v})&= (\mathbf{f},\mathbf{v})\quad \forall \mathbf{v}\in \bm{\mathit{L}}^2(\Omega).\label{weak form for our pde_d}
	\end{align}	
\end{subequations}
We remark that problem (\ref{weak form for our pde}) is solved together with the condition $\int_\Omega u_i=0$ for $i=1,2$ ($\mathbf{u}=(u_1,u_2)$) to ensure the uniqueness of the displacement in our computation.
\begin{assumption}
	\label{0128b}
	(see\cite{Dau1988,Gri1986,SDR2019,Zhong_2023})$\exists$ $s_0\in(0,1/2)$ such that for all $s\in(0,s_0]$: There exists a constant $C^{reg}>0$, independent of $\lambda$, such that for all $\bm{\mathit{f}}\in\bm{\mathit{L}}^2(\Omega)$, if $(\underline{\bm{\sigma}},\mathbf{u})\in\underline{\bm{\mathit{Y}}}\times\bm{\mathit{L}}^2(\Omega)$	is the solution to (\ref{weak form for our pde}), then
	$||\underline{\bm{\sigma}}||_s+||\mathbf{u}||_{1+s}\leq C^{reg}||\bm{\mathit{f}}||_{L^2}.$
	Consequently, $(\underline{\bm{\sigma}},\mathbf{u})\in\underline{\bm{\mathit{H}}}^s(\Omega)\times\underline{\bm{\mathit{H}}}^{1+s}(\Omega)$.
\end{assumption}

 We will discuss the construction of multiscale basis functions in Section \ref{Construction of the approximation space}. We consider the multiscale spaces $U_\mathup{aux}$ and $\Sigma_\mathup{ms}$ for the approximation of displacement and stress, respectively. The multiscale solution $(\underline{\bm{\sigma}}_\mathup{ms},\mathbf{u}_\mathup{ms})\in\Sigma_\mathup{ms}\times U_\mathup{aux}$ is obtained by solving
\begin{subequations}
	\label{multiscale solution}
	\begin{align}		
		(\mathcal{A}\underline{\bm{\sigma}}_\mathup{ms},\underline{\bm{\tau}})+(\textbf{div}\underline{\bm{\tau}},\mathbf{u}_\mathup{ms})&=0  \quad \forall \underline{\bm{\tau}}\in \Sigma_\mathup{ms}, \label{multiscale solution_a}\\
		(\textbf{div}\underline{\bm{\sigma}}_\mathup{ms},\mathbf{v})&= (\mathbf{f},\mathbf{v})\quad \forall \mathbf{v}\in U_\mathup{aux}.\label{multiscale solution_d}
	\end{align}	
\end{subequations}

\begin{remark}
\label{generalized to more bdcs}
We stress that the homogeneous Neumann boundary conditions employed in (\ref{our pde}) are chosen for clarity of exposition, not a fundamental limitation of our methodology. Indeed, all theoretical analyses and numerical techniques developed in this work can be naturally extended to more general boundary conditions, including Dirichlet and mixed boundary scenarios. To illustrate this extensibility, consider the case of mixed boundary conditions where $\partial\Omega = \Gamma_D \cup \Gamma_N$ with $\Gamma_D \cap \Gamma_N = \emptyset$, where $\Gamma_D \neq \emptyset$ denotes the Dirichlet boundary and $\Gamma_N$ the Neumann boundary. In this setting, the function space $\underline{\bm{\mathit{Y}}}$ would be modified to:
$\underline{\bm{\mathit{Y}}}\coloneqq\{\underline{\bm{\tau}}\in\underline{\bm{\mathit{H}}}(\textbf{\textup{div}};\Omega,\mathbb{S})\colon\underline{\bm{\tau}}\mathbf{n}=\mathbf{0}\hspace{0.5em} {\rm on\hspace{0.5em}} \Gamma_N\}$.
The well-posedness of the mixed linear elasticity problem under such boundary conditions, including the inf-sup condition and coercivity, follows standard arguments \cite{arnold2002, arnold2008, DFLS2008, BF1991, Zhong_2023}. Consequently, the analytical framework developed for the homogeneous Neumann case adapts straightforwardly to mixed boundary conditions.
Furthermore, Section~\ref{different bdcs} presents comprehensive numerical experiments demonstrating the accuracy of our method for both Dirichlet and mixed boundary conditions, validating its robustness across different boundary constraint scenarios.
\end{remark}

\subsection{Definitions and notation}
Let $\mathcal{T}_H\coloneqq\cup_{i=1}^N\{K_i\}$ denote a conforming quasi-uniform partition of the domain $\Omega$ into quadrilateral elements, where $H$ represents the coarse mesh size and $N$ is the total number of coarse elements. We refer to $\mathcal{T}_H$ as the coarse grid, where each coarse element $K_i$ is further subdivided into a connected union of fine-grid blocks. The corresponding fine grid, denoted by $\mathcal{T}_h\coloneqq\cup_{i=1}^{N_h}\{T_i\}$ (with $N_h$ being the number of fine elements), is constructed as a refinement of $\mathcal{T}_H$.
For an illustrative example, see Figure \ref{coarse_bigger}, which depicts the coarse mesh, fine mesh, and an oversampling domain with a single oversampling layer.

For each coarse element \( K_i \) (\( 1 \leq i \leq N \)), we define the local auxiliary multiscale space as $U_{\mathrm{aux}}(K_i)\coloneqq\text{span}\{\mathbf{p}^i_j|1\leq j\leq l_i\}$, where the specific construction of $\{\mathbf{p}_j^i\}_{j=1}^{l_i}$ will be detailed in Section~\ref{Displacement basis function}. The global auxiliary multiscale finite element space \( U_{\mathrm{aux}} \) is then defined by
\(
U_{\mathrm{aux}} = \oplus_i U_{\mathrm{aux}}(K_i).
\)
We emphasize that \( U_{\mathrm{aux}} \) serves as the approximation space for displacement fields, as employed in problem \cref{multiscale solution}.

For each coarse element $K_i$ ($1 \leq i \leq N$), we define a bilinear form
\begin{equation}
\label{def_s_i}
s_i(\mathbf{p},\mathbf{q}) = \int_{K_i} \tilde{k} \mathbf{p} \cdot \mathbf{q} \, dx, \quad \tilde{k} = k H^{-2},
\end{equation}
for all $\mathbf{p}, \mathbf{q} \in \bm{\mathit{L}}^2(K_i)$, where $k = \lambda + 2\mu$ represents the composite material parameter. We assume the normalization $s_i(\mathbf{p}^i_j,\mathbf{p}^i_j)=1$. For each local auxiliary space $U_{\mathrm{aux}}(K_i)$ ($1 \leq i \leq N$), the bilinear form $s_i$ defined in (\ref{def_s_i}) induces an inner product with the associated norm
\[
\norm{\mathbf{p}}_{s(K_i)} = s_i(\mathbf{p},\mathbf{p})^{1/2}, \quad \forall \mathbf{p} \in U_{\mathrm{aux}}(K_i).
\]
These local definitions naturally extend to the global auxiliary multiscale space $U_{\mathrm{aux}}$ through
\[
s(\mathbf{p},\mathbf{q}) = \sum_{i=1}^N s_i(\mathbf{p},\mathbf{q}) \quad \text{and} \quad \norm{\mathbf{p}}_s = s(\mathbf{p},\mathbf{p})^{1/2}, \quad \forall \mathbf{p} \in U_{\mathrm{aux}}.
\]
We remark that $s(\cdot,\cdot)$ and $\norm{\cdot}_s$ also constitute a valid inner product and norm on the space $\bm{\mathit{L}}^2(\Omega)$. Next, we define the projection operator $\pi_i \colon \bm{\mathit{L}}^2(K_i) \to U_{\mathrm{aux}}(K_i)$ 
with respect to the inner product $s_i(\cdot,\cdot)$. Specifically, for any $\mathbf{q} \in \bm{\mathit{L}}^2(K_i)$, 
the operator $\pi_i$ is given by
\[
\pi_i(\mathbf{q}) = \sum_{j=1}^{l_i} s_i(\mathbf{q}, \mathbf{p}_j^i)\mathbf{p}_j^i.
\]
Similarly, we define the global projection operator $\pi \colon \bm{\mathit{L}}^2(\Omega) \to U_{\mathrm{aux}}$ 
with respect to the inner product $s(\cdot,\cdot)$. For any $\mathbf{q} \in \bm{\mathit{L}}^2(\Omega)$,
this operator takes the form
$$
\pi(\mathbf{q}) = \sum_{i=1}^N \sum_{j=1}^{l_i} s_i(\mathbf{q}, \mathbf{p}_j^i) \mathbf{p}_j^i.
$$
It follows immediately that $\pi$ admits the decomposition $\pi = \sum_{i=1}^N \pi_i$.

To facilitate the analysis in Section~\ref{Analysis for stability and convergence}, we introduce the following bilinear forms on the restricted space $\underline{\bm{Y}}|_D$, where $D \subseteq \Omega$ denotes an arbitrary subdomain:
\[
(\underline{\bm{\sigma}},\underline{\bm{\tau}})_{\Sigma,D} = (\mathcal{A}\underline{\bm{\sigma}},\underline{\bm{\tau}})_D + \int_D \tilde{k}^{-1} \mathbf{div}\,\underline{\bm{\sigma}} \cdot \mathbf{div}\,\underline{\bm{\tau}} \, dx,
\quad
(\underline{\bm{\sigma}},\underline{\bm{\tau}})_{\mathcal{A},D,\textbf{div}} = (\mathcal{A}\underline{\bm{\sigma}},\underline{\bm{\tau}})_D + (\mathbf{div}\,\underline{\bm{\sigma}}, \mathbf{div}\,\underline{\bm{\tau}})_D
\]
and their associated norms
\[
\norm{\underline{\bm{\sigma}}}_{\Sigma(D)} = \left( (\mathcal{A}\underline{\bm{\sigma}},\underline{\bm{\sigma}})_D + \int_D \tilde{k}^{-1} \mathbf{div}\,\underline{\bm{\sigma}} \cdot \mathbf{div}\,\underline{\bm{\sigma}} \, dx \right)^{1/2},
\quad
\norm{\underline{\bm{\beta}}}_{\mathcal{A}(D),\textbf{div}}=\left((\mathcal{A}\underline{\bm{\beta}},\underline{\bm{\beta}})_D+(\textbf{div}\underline{\bm{\beta}},\textbf{div}\underline{\bm{\beta}})_D\right)^{1/2}.
\]
We shall omit the subscript $D$ when the domain under consideration is the entire $\Omega$ (i.e., when $D = \Omega$).

\subsection{Mixed finite element method}
We discretize problem (\ref{weak form for our pde}) by mixed finite element method proposed by Falk etc.\ on the fine grid $\mathcal{T}_h$ (see for instance \cite{DFLS2008, CB1978, BJ1968}). For example, following the construction in \cite[Section 2.1]{DFLS2008}, we partition each triangle $T \in \mathcal{T}_h$ into three sub-triangles $\{t_i\}_{i=1}^3$ by connecting its barycenter to the vertices, forming a composite element $T = \bigcup_{i=1}^3 t_i$. For polynomial degree $k \geq 2$, we define the finite element spaces as follows:
$\underline{\bm{\Sigma}}_h \coloneqq \{\underline{\bm{\tau}} \in \underline{\bm{\mathit{H}}}(\textbf{div};\Omega,\mathbb{S}) \colon \underline{\bm{\tau}}|_{t_i} \in \underline{\mathcal{P}}_k(t_i,\mathbb{S}), \text{ $\underline{\bm{\tau}}\mathbf{n}$ continuous across internal edges}, 
 \textbf{div}\underline{\bm{\tau}} \in \mathcal{P}_{k-1}(T,\mathbb{R}^2)\}$, 
$\bm{\mathit{U}}_h \coloneqq \{\mathbf{v} \in \bm{\mathit{L}}^2(\Omega) \colon \mathbf{v}|_{T} \in \mathcal{P}_{k-1}(T,\mathbb{R}^2)\}$.
Then the fine problem is to find $(\underline{\bm{\sigma}}_h,\mathbf{u}_h)\in \underline{\bm{\mathit{\Sigma}}}_h\times\bm{\mathit{U}}_h$ such that
\begin{subequations}
	\label{discrete formula}
	\begin{align}		
		(\mathcal{A}\underline{\bm{\sigma}}_h,\underline{\bm{\tau}}_h)+(\textbf{div}\underline{\bm{\tau}}_h,\mathbf{u}_h)&=0  \quad \forall \underline{\bm{\tau}}_h\in \underline{\bm{\mathit{\Sigma}}}_h, \label{discrete formula_a}\\
		(\textbf{div}\underline{\bm{\sigma}}_h,\mathbf{v}_h)&= s(\pi(\tilde{k}^{-1}\mathbf{f}),\mathbf{v}_h)\quad \forall \mathbf{v}_h\in \bm{\mathit{U}}_h,\label{discrete formula_d}
	\end{align}	
\end{subequations} 
The well-posedness of Eqs. (\ref{discrete formula}) is standard:

(i) the discrete inf-sup condition: $\exists\beta>0$
independent of $h$ such that for all $\mathbf{w}_h\in\bm{\mathit{U}}_h$ with $\int_\Omega\mathbf{w}_h=\mathbf{0}$, we have
\begin{equation}
\label{discrete inf-sup}
\sup_{\underline{\bm{\tau}}_h\in\underline{\bm{\mathit{\Sigma}}}_h}\frac{(\textbf{div}\underline{\bm{\tau}}_h,\mathbf{w}_h)}{\norm{\underline{\bm{\tau}}_h}_{\underline{\bm{\mathit{H}}}(\textbf{div};\Omega)}}\geq \beta\norm{\mathbf{w}_h}_{L^2}.
\end{equation}

(ii) the coercivity in the kernel condition $(\mathcal{A}\underline{\bm{\tau}},\underline{\bm{\tau}})\geq C||\underline{\bm{\tau}}||^2_{\underline{\bm{\mathit{H}}}(\textbf{div};\Omega)}$,
for all $\underline{\bm{\tau}}$ in the kernel given by $\bm{\mathit{K}}$:= $\{\underline{\bm{\tau}}\in\underline{\bm{\mathit{\Sigma}}}_h: (\bm{\omega}_h,\textbf{div}\underline{\bm{\tau}})_{\Omega}=0\hspace{0.5em} {\rm for\hspace{0.5em} all\hspace{0.5em}} \bm{\omega}_h\in\bm{\mathit{U}}^h\hspace{0.5em}\}.$ Here the constant $C$ doesn't depend on $\lambda$.

We remark that $(\underline{\bm{\sigma}}_h,\mathbf{u}_h)$ is considered as the reference solution in our analysis. Moreover, let $\bm{\mathit{P}}\colon \bm{\mathit{L}}^2(\Omega)\to\bm{\mathit{U}}_h$ be the $L^2$ projection from  $\bm{\mathit{L}}^2(\Omega)$ onto the fine grid space $\bm{\mathit{U}}_h$. That is, for any $\mathbf{w}\in \bm{\mathit{L}}^2(\Omega)$, $\bm{\mathit{P}}\mathbf{w}$ is the unique function in $\bm{\mathit{U}}_h$ such that
$(\bm{\mathit{P}}\mathbf{w},\mathbf{v}_h)=(\mathbf{w},\mathbf{v}_h)$ for all $\mathbf{v}_h\in\bm{\mathit{U}}_h$.

Given $t>0$, let $\mathbf{\Pi}_h: \underline{\bm{\mathit{H}}}^t(\Omega)\cap\bm{\underline{H}}(\textbf{div};\Omega,\mathbb{S})\to \underline{\bm{\mathit{\Sigma}}}_h$ be the interpolation operator that satisfies the commuting diagram property (see \cite{BF1991,DFLS2008}):
\begin{equation}
\label{commuting diagram}
\textbf{div}\mathbf{\Pi}_h\underline{\bm{\tau}}=\bm{\mathit{P}}\textbf{div}\underline{\bm{\tau}},\quad\forall \underline{\bm{\tau}}\in\underline{\bm{\mathit{H}}}^t(\Omega)\cap\underline{\bm{\mathit{H}}}(\textbf{div};\Omega).
\end{equation}
In addition, it's well known that there exists $C>0$, independent of $h$, such that for each $\underline{\bm{\tau}}\in\underline{\bm{\mathit{H}}}^t(\Omega)\cap\bm{\underline{H}}(\textbf{div};\Omega,\mathbb{S})$ (see \cite{BBF2013}): 
\begin{equation}
	\label{ii}
	||\underline{\bm{\tau}}-\mathbf{\Pi}_h\underline{\bm{\tau}}||_{0,\Omega}\leq Ch^{\rm min\{t, k+2\}}||\underline{\bm{\tau}}||_{t,\Omega},\quad \forall t>1/2.
\end{equation}
For less regular tensorial fields, we have the following error estimate (see \cite[Theorem 3.16]{H2002})
\begin{equation}
	\label{jj}
	||\underline{\bm{\tau}}-\mathbf{\Pi}_h\underline{\bm{\tau}}||_{0,\Omega}\leq Ch^t(||\underline{\bm{\tau}}||_{t,\Omega}+||\textbf{div}\underline{\bm{\tau}}||_{0,\Omega}),\quad \forall t\in(0,1/2].
\end{equation}
For any $t>0$, we also have
\begin{equation}
	\label{ll}
	||\mathbf{u}-\bm{\mathit{P}}\mathbf{u}||_{0,\Omega}\leq Ch^{\rm min\{t,k+1\}}||\mathbf{u}||_{t,\Omega}\quad\forall \mathbf{u}\in{\bm{\mathit{H}}}^t(\Omega).
\end{equation}

\section{The multiscale method}
\label{Construction of the approximation space}
In this section, we will present the construction of our multiscale method. The construction of the basis functions are developed on the coarse mesh illustrated in Figure~\ref{coarse_bigger} and divided into two stages.
The first stage consists of constructing the multiscale space for the displacement $\mathbf{u}$ (Section \ref{Displacement basis function}). 
In the second stage, we will use the multiscale space for displacement to construct a multiscale space for the stress $\bm{\underline{\sigma}}$ (Section \ref{Stress basis function}). 
We point out that the supports of displacement basis are the coarse elements. 
For stress basis functions, the support is an oversampled region containing the support of displacement basis functions. 
\begin{figure}[tbph] 
		\centering 
		\includegraphics[height=5cm,width=9.3cm]{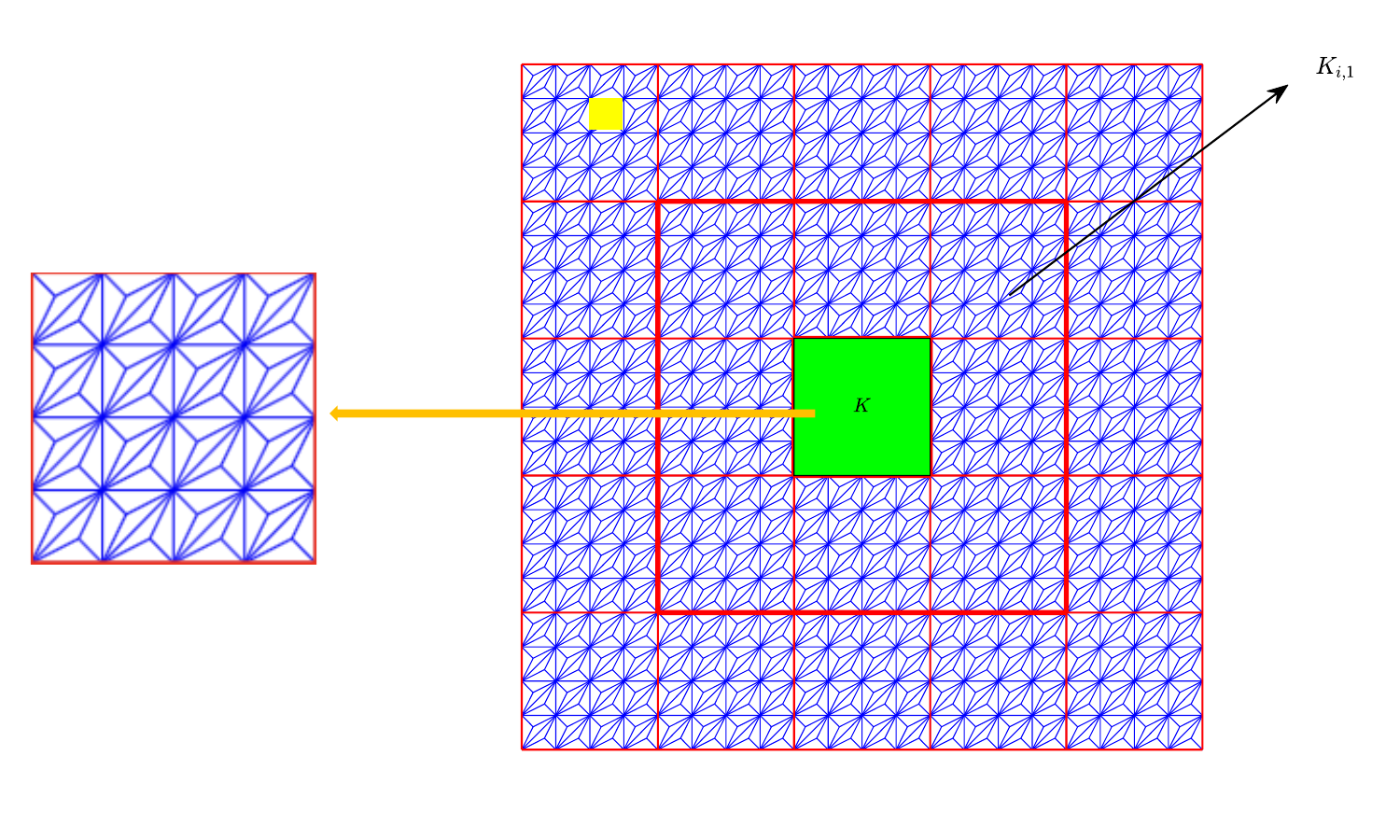} 
		\caption{An illustration of coarse element(green, and left for clearer), fine element(yellow), and an oversampled region $K_{i,1}$ by extending a coarse element by one coarse grid layer(thick red line)}
		\label{coarse_bigger}
\end{figure}
\subsection{Displacement basis function}\label{Displacement basis function}
We will construct a set of auxiliary multiscale basis functions for each coarse element $K_i$ by solving a local spectral problem. First, we define some notation. For a general set $R$, we define $\bm{\mathit{U}}_h(R)$ as the restriction of $\bm{\mathit{U}}_h$ on $R$ and $\underline{\bm{\mathit{\Sigma}}}_h(R)\coloneqq\{\underline{\bm{\tau}}_h\in\underline{\bm{\mathit{\Sigma}}}_h\colon\underline{\bm{\tau}}_h\mathbf{n}=\mathbf{0}\text{ on }\partial R\}$. Note that $\underline{\bm{\mathit{\Sigma}}}_h(R)$ is with homogeneous traction boundary condition on $R$, which confirms the conforming property of the multiscale bases in the construction process.

Next, we define the required spectral problem. For each coarse element $K_i$, we solve the eigenvalue problem: find $(\bm{\underline{\phi}}^i_j,\mathbf{p}^i_j)\in \underline{\bm{\mathit{\Sigma}}}_h(K_i)\times \bm{\mathit{U}}_h(K_i)$ and $\lambda^i_j\in\mathbb{R}$ such that
\begin{subequations}
	\label{local spectral problem}
	\begin{align}		
		(\mathcal{A}\bm{\underline{\phi}}^i_j,\underline{\bm{\tau}}_h)+(\textbf{div}\underline{\bm{\tau}}_h,\mathbf{p}^i_j)&=0  \quad \forall \underline{\bm{\tau}}_h\in \underline{\bm{\mathit{\Sigma}}}_h(K_i), \label{local spectral problem_a}\\
		-(\textbf{div}\bm{\underline{\phi}}^i_j,\mathbf{v}_h)&= \lambda^i_js_i(\mathbf{p}^i_j,\mathbf{v}_h)\quad \forall \mathbf{v}_h\in \bm{\mathit{U}}_h(K_i).\label{local spectral problem_d}
	\end{align}	
\end{subequations}
We arrange the eigenvalues of (\ref{local spectral problem}) in non-decreasing order $0=\lambda_1^i\leq\lambda_2^i\leq\cdot\cdot\cdot\leq\lambda_{L_i}^i$, where $L_i$ is the dimension of the space $\bm{\mathit{U}}_h(K_i)$. For each $i\in\{1,2,\cdot\cdot\cdot,N\}$, choose the first $l_i$ $(1\leq l_i\leq L_i)$ eigenfunctions $\{\mathbf{p}^i_j\}_{j=1}^{l_i}$ corresponding the first $l_i$ smallest eigenvalues. Then, the local auxiliary multiscale space $U_\mathup{aux}(K_i)$ for displacement is $U_\mathup{aux}(K_i)\coloneqq\text{span}\{\mathbf{p}^i_j|1\leq j\leq l_i\}.$
The global auxiliary multiscale finite element space $U_\mathup{aux}$ is defined by $U_\mathup{aux}=\oplus_iU_\mathup{aux}(K_i)$. We define $\Lambda=\min_{1\leq i\leq N}\lambda_{l_i+1}^i$. The first component of the first nine displacement multiscale basis functions are presented in Figure \ref{multiscale displacement basis}.
\begin{remark}
We point out that by spectral theorem (see for instance \cite{PG_2013}), compactness of the corresponding solution operator (in continuous level) and the estimate for discrete eigenproblem(see \cite{Zhong_2023,SDR2013}), the eigenvalues obtained in (\ref{local spectral problem}) are real. Moreover, let the test functions $\underline{\bm{\tau}}_h=\bm{\underline{\phi}}^i_j,\mathbf{v}_h=\tilde{k}^{-1}\textbf{\textup{div}}\bm{\underline{\phi}}^i_j$ in (\ref{local spectral problem}), then by a simple computation, we have $(\tilde{k}^{-1}\textbf{\textup{div}}\bm{\underline{\phi}}^i_j,\textbf{\textup{div}}\bm{\underline{\phi}}^i_j)=\lambda^i_j(\mathcal{A}\bm{\underline{\phi}}^i_j,\bm{\underline{\phi}}^i_j)$, which indicates $\lambda^i_j\geq0$.
\end{remark}
\begin{figure}[tbph] 
		\centering 
		\includegraphics[height=10cm,width=14cm]{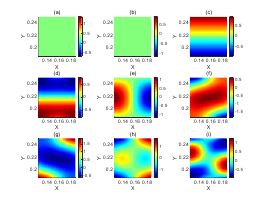} 
		\caption{First components of displacement bases corresponding to the first nine smallest eigenvalues}
		\label{multiscale displacement basis}
	\end{figure}
\subsection{Stress basis function}
\label{Stress basis function}
In this section, we will present the construction of the stress basis functions. We directly present the relaxed version of stress multiscale basis functions (see, for instance, \cite{Eric_mix_2018}). Let $\mathbf{p}^i_j\in U_\mathup{aux}$ be a given displacement basis function supported in $K_i$. $K_i^+$ is the corresponding oversampled region. We will define a stress basis function $\bm{\underline{\psi}}^i_{j,\mathup{ms}}\in\underline{\bm{\mathit{\Sigma}}}_h(K_i^+)$ by solving (\ref{multiscale basis}). Some multiscale stress bases are presented in Figure \ref{multiscale stress basis}. The multiscale space is defined as $\Sigma_\mathup{ms}\coloneqq\text{span}\{\bm{\underline{\psi}}^i_{j,\mathup{ms}}\}$. Note that the basis function is supported in $K_i^+$, which is a union of connected coarse elements and contains $K_i$. We define $J_i$ as the set of indices such that if $k\in J_i$, then $K_k\in K_i^+$. We also define $U_\mathup{aux}(K_i^+)=\text{span}\{\mathbf{p}^k_j|1\leq j\leq l_k,k\in J_i\}$.

\begin{figure}[tbph] 
		\centering 
		\includegraphics[height=10cm,width=14cm]{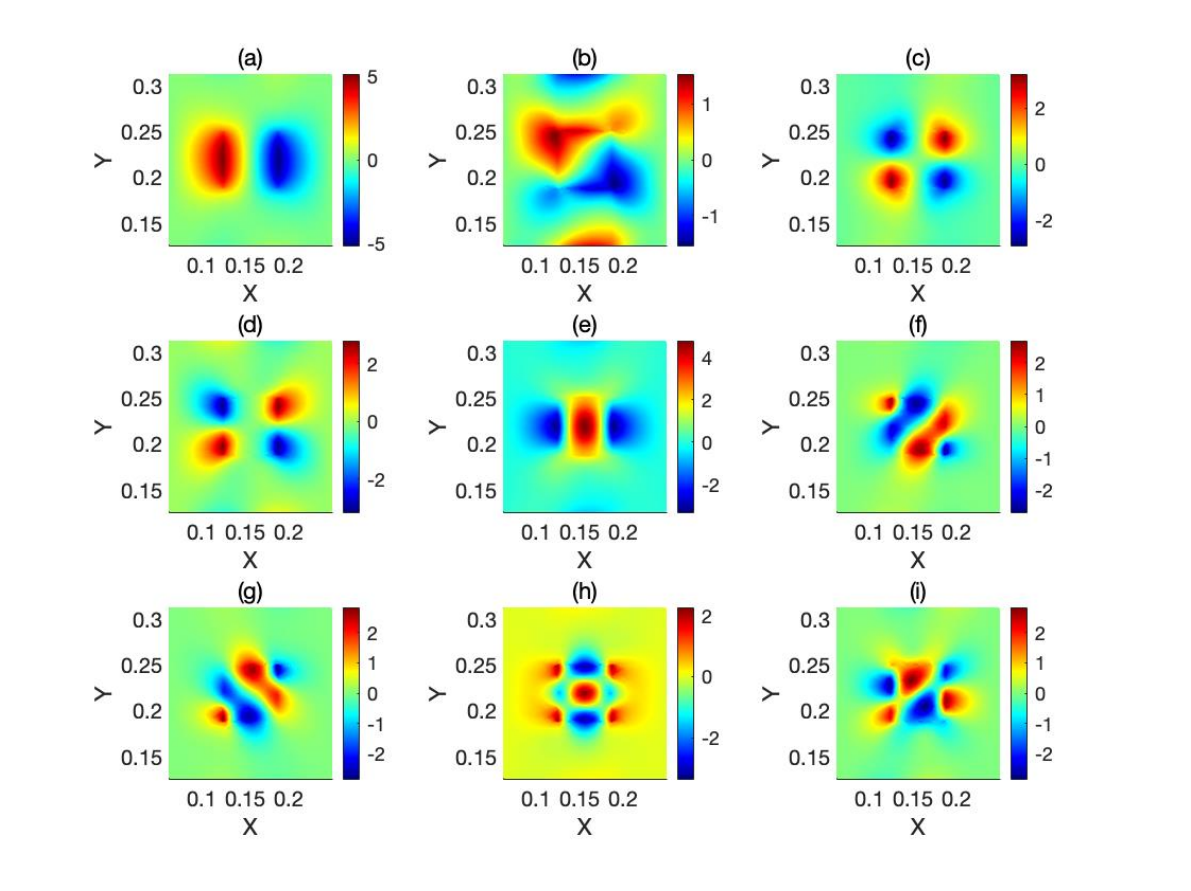} 
		\caption{First components of multiscale stress bases corresponding to displacement bases in Figure \ref{multiscale displacement basis}}
		\label{multiscale stress basis}
\end{figure}

We find $\bm{\underline{\psi}}^i_{j,\mathup{ms}}\in\underline{\bm{\mathit{\Sigma}}}_h(K_i^+)$ and $\mathbf{q}^i_{j,\mathup{ms}}\in \bm{\mathit{U}}_h(K_i^+)$ such that
\begin{subequations}
	\label{multiscale basis}
	\begin{align}		
		(\mathcal{A}\bm{\underline{\psi}}^i_{j,\mathup{ms}},\underline{\bm{\tau}}_h)+(\textbf{div}\underline{\bm{\tau}}_h,\mathbf{q}^i_{j,\mathup{ms}})&=0  \quad \forall \underline{\bm{\tau}}_h\in \underline{\bm{\mathit{\Sigma}}}_h(K_i^+), \label{multiscale basis_a}\\
s(\pi\mathbf{q}^i_{j,\mathup{ms}},\pi\mathbf{v}_h)-(\textbf{div}\bm{\underline{\psi}}^i_{j,\mathup{ms}},\mathbf{v}_h)&= s(\mathbf{p}^i_j,\mathbf{v}_h)\quad \forall \mathbf{v}_h\in \bm{\mathit{U}}_h(K_i^+).\label{multiscale basis_d}
	\end{align}	
\end{subequations}
To enhance computational efficiency for (\ref{multiscale basis}), we will combine the techniques in \cite{HWW1989,RKS1992}. The global basis function $\bm{\underline{\psi}}^i_j\in\underline{\bm{\mathit{\Sigma}}}_h$ is constructed by solving the following problem: find $\bm{\underline{\psi}}^i_j\in\underline{\bm{\mathit{\Sigma}}}_h$ and $\mathbf{q}^i_j\in\bm{\mathit{U}}_h$ such that 
\begin{subequations}
	\label{global multiscale basis}
	\begin{align}		
		(\mathcal{A}\bm{\underline{\psi}}^i_j,\underline{\bm{\tau}}_h)+(\textbf{div}\underline{\bm{\tau}}_h,\mathbf{q}^i_j)&=0  \quad \forall \underline{\bm{\tau}}_h\in \underline{\bm{\mathit{\Sigma}}}_h, \label{global multiscale basis_a}\\
s(\pi\mathbf{q}^i_j,\pi\mathbf{v}_h)-(\textbf{div}\bm{\underline{\psi}}^i_j,\mathbf{v}_h)&= s(\mathbf{p}^i_j,\mathbf{v}_h)\quad \forall \mathbf{v}_h\in \bm{\mathit{U}}_h\label{global multiscale basis_d}
	\end{align}	
\end{subequations}
The global
multiscale space is defined as $\Sigma_\mathup{glo}=\text{span}\{\bm{\underline{\psi}}^i_j\}$. Similar to the observation in \cite{Eric_mix_2018}, we find (\ref{global multiscale basis}) defines a mapping $G$ from $U_\mathup{aux}$ to $\Sigma_\mathup{glo}\times \bm{\mathit{U}}_h$. In particular, given $\mathbf{p}_\mathup{aux}\in U_\mathup{aux}$, the image $G(\mathbf{p}_\mathup{aux})=(G_1(\mathbf{p}_\mathup{aux}),G_2(\mathbf{p}_\mathup{aux}))=(\underline{\bm{\psi}},\mathbf{r})\in\Sigma_\mathup{glo}\times \bm{\mathit{U}}_h$ is defined by
\begin{subequations}
	\label{global multiscale basis_another version}
	\begin{align}		
		(\mathcal{A}\bm{\underline{\psi}},\underline{\bm{\tau}}_h)+(\textbf{div}\underline{\bm{\tau}}_h,\mathbf{r})&=0  \quad \forall \underline{\bm{\tau}}_h\in \underline{\bm{\mathit{\Sigma}}}_h, \label{global multiscale basis_another version_a}\\
s(\pi\mathbf{r},\pi\mathbf{v}_h)-(\textbf{div}\bm{\underline{\psi}},\mathbf{v}_h)&= s(\mathbf{p}_\mathup{aux},\mathbf{v}_h)\quad \forall \mathbf{v}_h\in \bm{\mathit{U}}_h.\label{global multiscale basis_another version_d}
	\end{align}	
\end{subequations}
Next, we give a characterization of the space $\Sigma_\mathup{glo}$ in the following lemma.
\begin{lemma}
\label{characterization of global basis functions}
Let $\Sigma_\mathup{glo}$ be the global multiscale space for the stress. For each $\mathbf{p}_\mathup{aux}\in U_\mathup{aux}$ with $s(\mathbf{p}_\mathup{aux},\mathbf{x})=0$ (where $\mathbf{x}\in\text{span}\{(1,0)^t,(0,1)^t\}$ in two dimension while in three dimension, $\mathbf{x}\in\text{span}\{(1,0,0)^t,(0,1,0)^t,(0,0,1)^t\}$), there is a unique $\bm{\underline{\sigma}}^*\in\Sigma_\mathup{glo}$ such that $(\bm{\underline{\sigma}}^*,\mathbf{u}^*)\in\bm{\underline{\Sigma}}_h\times\bm{\mathit{U}}_h$ with $\int_\Omega\mathbf{u}^*=\mathbf{0}$ is the solution of 
\begin{subequations}
	\label{global multiscale basis_characterization}
	\begin{align}		
		(\mathcal{A}\bm{\underline{\sigma}}^*,\underline{\bm{\tau}}_h)+(\textbf{\textup{div}}\underline{\bm{\tau}}_h,\mathbf{u}^*)&=0  \quad \forall \underline{\bm{\tau}}_h\in \underline{\bm{\mathit{\Sigma}}}_h, \label{global multiscale basis_characterization_a}\\
-(\textbf{\textup{div}}\bm{\underline{\sigma}}^*,\mathbf{v}_h)&= s(\mathbf{p}_\mathup{aux},\mathbf{v}_h)\quad \forall \mathbf{v}_h\in \bm{\mathit{U}}_h.\label{global multiscale basis_characterization_d}
	\end{align}	
\end{subequations}
\end{lemma}
\begin{proof}
The proof is similar to \cite[Lemma 1]{Eric_mix_2018} so we omit the details.
\end{proof}

\section{Analysis}
\label{Analysis for stability and convergence}
In this section, we provide comprehensive analyses for the proposed multiscale method. We first prove the global stability and convergence in Section \ref{Stability and convergence of using global basis functions}. Then in Section \ref{Well-posedness of global and multiscale basis functions}, we consider the well-posedness of finding the multiscale basis functions defined in (\ref{multiscale basis}). Next in Section \ref{Decay property of global basis functions}, we prove a decay property of the global basis functions, which is crucial for the stability and convergence of the multiscale method (\ref{multiscale solution}) in Section \ref{final stability and convergence}.

Let $(\underline{\bm{\sigma}}_h,\mathbf{u}_h) \in \underline{\bm{\Sigma}}_h \times \bm{U}_h$ denote the reference solution of (\ref{discrete formula}). Our analysis focuses on estimating the errors $\underline{\bm{\sigma}}_h - \underline{\bm{\sigma}}_\mathrm{ms}$ and $\mathbf{u}_h - \mathbf{u}_\mathrm{ms}$, beginning with bounds for $\mathbf{u} - \mathbf{u}_h$ and $\underline{\bm{\sigma}} - \underline{\bm{\sigma}}_h$ in the following lemma.
\begin{lemma}
\label{convergence form the fine solution}
Let $(\underline{\bm{\sigma}}_h,\mathbf{u}_h)$ be the fine solution defined in problem (\ref{discrete formula}) and let $(\underline{\bm{\sigma}},\mathbf{u})$ be the exact solution defined in (\ref{weak form for our pde}). Then, we have
\begin{equation}
\label{fine error}
\norm{\underline{\bm{\sigma}}-\underline{\bm{\sigma}}_h}_\mathcal{A}^2+\norm{\mathbf{u}-\mathbf{u}_h}_{L^2}^2\leq C_Sh^{2s}\norm{\mathbf{f}}_{L^2}^2+C\frac{1}{\Lambda}\norm{(I-\pi)(\tilde{k}^{-1}\mathbf{f})}_s^2,
\end{equation}
where $C,C_S>0$ are independent of mesh sizes $h,H$ and the Lam$\rm \acute{e}$ coefficient $\lambda$, but $C_S$ depends on the regularity of the PDE problem (\ref{our pde}) (see Assumption \ref{0128b}). $\Lambda=\min_{1\leq i\leq N}\lambda_{l_i+1}^i$ maintains independent of $h,H$, and the Lam$\acute{\rm e}$ coefficient $\lambda$.
\end{lemma}
\begin{proof}
Combining (\ref{weak form for our pde}), (\ref{discrete formula}) and (\ref{commuting diagram}), we have
\begin{subequations}
	\label{weak form-discrete formula_2}
	\begin{align}		
		(\mathcal{A}(\mathbf{\Pi}_h\underline{\bm{\sigma}}-\underline{\bm{\sigma}}_h),\underline{\bm{\tau}}_h)+(\textbf{div}\underline{\bm{\tau}}_h,\bm{\mathit{P}}\mathbf{u}-\mathbf{u}_h)&=(\mathcal{A}(\mathbf{\Pi}_h\underline{\bm{\sigma}}-\underline{\bm{\sigma}}),\underline{\bm{\tau}}_h)  \quad \forall \underline{\bm{\tau}}_h\in \underline{\bm{\mathit{\Sigma}}}_h, \label{weak form-discrete formula_2_a}\\
		(\textbf{div}(\mathbf{\Pi}_h\underline{\bm{\sigma}}-\underline{\bm{\sigma}}_h),\mathbf{v}_h)&= s((I-\pi)(\tilde{k}^{-1}\mathbf{f}),\mathbf{v}_h),\quad \forall \mathbf{v}_h\in \bm{\mathit{U}}_h.\label{weak form-discrete formula_2_d}
	\end{align}	
\end{subequations}
Letting $\underline{\bm{\tau}}_h=\mathbf{\Pi}_h\underline{\bm{\sigma}}-\underline{\bm{\sigma}}_h$, $\mathbf{v}_h=\bm{\mathit{P}}\mathbf{u}-\mathbf{u}_h$ and combining Assumption \ref{0128b}, (\ref{ii}), (\ref{jj}), we easily obtain that
\[
(\mathcal{A}(\mathbf{\Pi}_h\underline{\bm{\sigma}}-\underline{\bm{\sigma}}_h),\mathbf{\Pi}_h\underline{\bm{\sigma}}-\underline{\bm{\sigma}}_h)\leq C_Sh^s\norm{\mathbf{f}}_{L^2}\norm{\mathbf{\Pi}_h\underline{\bm{\sigma}}-\underline{\bm{\sigma}}_h}_\mathcal{A}+\norm{(I-\pi)(\tilde{k}^{-1}\mathbf{f})}_s\cdot\norm{(I-\pi)(\bm{\mathit{P}}\mathbf{u}-\mathbf{u}_h)}_s,
\]
where $C_S>0$ is independent of $h$ and $\lambda$, but depends on the regularity  Assumption \ref{0128b}.

For each $K_i$, we define $\mathbf{u}_i$ to be the restriction of $(I-\pi)(\bm{\mathit{P}}\mathbf{u}-\mathbf{u}_h)$ on $K_i$. Then, by the definition of $\pi$, we can write $\mathbf{u}_i=\sum_{j>l_i}c^i_j\mathbf{p}^i_j$. Denote $\underline{\mathbf{x}}_i\in\underline{\bm{\Sigma}}_h(K_i)$ as $\underline{\mathbf{x}}_i=\sum_{j>l_i}(\lambda^i_j)^{-1}c^i_j\underline{\bm{\phi}}^i_j$.
Let $\underline{\bm{\tau}}_h=\underline{\mathbf{x}}_i$ in (\ref{weak form-discrete formula_2}), we obtain
$
(\mathcal{A}(\mathbf{\Pi}_h\underline{\bm{\sigma}}-\underline{\bm{\sigma}}_h),\underline{\mathbf{x}}_i)+(\textbf{div}\underline{\mathbf{x}}_i,\bm{\mathit{P}}\mathbf{u}-\mathbf{u}_h)=(\mathcal{A}(\mathbf{\Pi}_h\underline{\bm{\sigma}}-\underline{\bm{\sigma}}),\underline{\mathbf{x}}_i).
$
Using the spectral problem (\ref{local spectral problem}), we have
$
-(\textbf{div}\underline{\mathbf{x}}_i,\bm{\mathit{P}}\mathbf{u}-\mathbf{u}_h)=-(\textbf{div}\underline{\mathbf{x}}_i,(I-\pi)(\bm{\mathit{P}}\mathbf{u}-\mathbf{u}_h))=\sum_{j>l_i}(c^i_j)^2=\norm{(I-\pi)(\bm{\mathit{P}}\mathbf{u}-\mathbf{u}_h)}_s^2.
$
Then by a simple computation, 
$(\mathcal{A}(\mathbf{\Pi}_h\underline{\bm{\sigma}}-\underline{\bm{\sigma}}_h),\underline{\mathbf{x}}_i)-(\mathcal{A}(\mathbf{\Pi}_h\underline{\bm{\sigma}}-\underline{\bm{\sigma}}),\underline{\mathbf{x}}_i)\leq \frac{1}{\sqrt{\Lambda}}(C\norm{\mathbf{\Pi}_h\underline{\bm{\sigma}}-\underline{\bm{\sigma}}_h}_\mathcal{A}+C_Sh^s\norm{\mathbf{f}}_{L^2})\cdot\norm{(I-\pi)(\bm{\mathit{P}}\mathbf{u}-\mathbf{u}_h)}_s$.
Therefore, combining Assumption \ref{0128b}, we have
\[
\norm{\mathbf{\Pi}_h\underline{\bm{\sigma}}-\underline{\bm{\sigma}}_h}_\mathcal{A}^2\leq C\frac{1}{\Lambda}\norm{(I-\pi)(\tilde{k}^{-1}\mathbf{f})}_s^2+C_Sh^{2s}\norm{\mathbf{f}}_{L^2}^2.
\]
By utilizing the inf-sup condition (\ref{discrete inf-sup}), we have
$
\norm{\bm{\mathit{P}}\mathbf{u}-\mathbf{u}_h}_{L^2}\leq C\sup_{\underline{\bm{\tau}}\in\underline{\bm{\Sigma}}_h}\frac{(\textbf{div}\underline{\bm{\tau}},\bm{\mathit{P}}\mathbf{u}-\mathbf{u}_h)}{\norm{\underline{\bm{\tau}}}_{H(\textbf{div})}}\leq C\norm{\mathbf{\Pi}_h\underline{\bm{\sigma}}-\underline{\bm{\sigma}}_h}_\mathcal{A}+C_Sh^s\norm{\mathbf{f}}_{L^2}.
$
Then combining Assumption \ref{0128b} and (\ref{ii})-(\ref{ll}), we obtain the desired results.
\end{proof}
\subsection{Stability and convergence of using global basis functions}\label{Stability and convergence of using global basis functions}
We approximate the problem (\ref{our pde}) using the space $U_\mathup{aux}$ for displacement and $\Sigma_\mathup{glo}$ for stress. So, we define the solution $(\underline{\bm{\sigma}}_\mathup{glo},\mathbf{u}_\mathup{glo})\in \Sigma_\mathup{glo}\times U_\mathup{aux}$ using global basis function by the following:
\begin{subequations}
	\label{global variational form}
	\begin{align}		
		(\mathcal{A}\underline{\bm{\sigma}}_\mathup{glo},\underline{\bm{\tau}})+(\textbf{div}\underline{\bm{\tau}},\mathbf{u}_\mathup{glo})&=0  \quad \forall \underline{\bm{\tau}}\in \Sigma_\mathup{glo}, \label{global variational form_a}\\
		(\textbf{div}\underline{\bm{\sigma}}_\mathup{glo},\mathbf{v})&= (\mathbf{f},\mathbf{v})\quad \forall \mathbf{v}\in U_\mathup{aux}.\label{global variational form_d}
	\end{align}	
\end{subequations}
We will prove the stability and convergence of (\ref{global variational form}). First, we prove the following inf-sup condition.
\begin{lemma}
\label{inf-sup for global variational form}
For every $\mathbf{p}_\mathup{aux}\in U_\mathup{aux}$ with $s(\mathbf{p}_\mathup{aux},\mathbf{x})=0$ (where $\mathbf{x}\in\text{span}\{(1,0)^t,(0,1)^t\}$ in two dimension while in three dimension, $\mathbf{x}\in\text{span}\{(1,0,0)^t,(0,1,0)^t,(0,0,1)^t\}$), there is $\underline{\bm{\beta}}\in \Sigma_\mathup{glo}$ such that
\[
\norm{\mathbf{p}_\mathup{aux}}_s\leq C_g\frac{(\textbf{\textup{div}}\underline{\bm{\beta}},\mathbf{p}_\mathup{aux})}{\norm{\underline{\bm{\beta}}}_{\mathcal{A},\textbf{\textup{div}}}},
\]
where $\norm{\underline{\bm{\beta}}}_{\mathcal{A},\textbf{\textup{div}}}^2=(\mathcal{A}\underline{\bm{\beta}},\underline{\bm{\beta}})+(\textbf{\textup{div}}\underline{\bm{\beta}},\textbf{\textup{div}}\underline{\bm{\beta}})$ by the definition of $\norm{\cdot}_{\mathcal{A},\textbf{\textup{div}}}$.
\end{lemma}
\begin{proof}
Let $\mathbf{p}_\mathup{aux}\in U_\mathup{aux}$ with $s(\mathbf{p}_\mathup{aux},\mathbf{x})=0$ (where $\mathbf{x}\in\text{span}\{(1,0)^t,(0,1)^t\}$ in two dimension while in three dimension, $\mathbf{x}\in\text{span}\{(1,0,0)^t,(0,1,0)^t,(0,0,1)^t\}$). Similar to Lemma \ref{characterization of global basis functions}, we consider the following problem:
\begin{subequations}
	\label{variational form for inf-sup}
	\begin{align}		
		(\mathcal{A}\bm{\underline{\alpha}},\underline{\bm{\tau}}_h)+(\textbf{div}\underline{\bm{\tau}}_h,\mathbf{p})&=0  \quad \forall \underline{\bm{\tau}}_h\in \underline{\bm{\mathit{\Sigma}}}_h, \label{variational form for inf-sup_a}\\
-(\textbf{div}\bm{\underline{\alpha}},\mathbf{v}_h)&= s(\mathbf{p}_\mathup{aux},\mathbf{v}_h)\quad \forall \mathbf{v}_h\in \bm{\mathit{U}}_h, \label{variational form for inf-sup_d}
	\end{align}	
\end{subequations}
where the solution $\underline{\bm{\alpha}}\in \Sigma_\mathup{glo}$. Taking $\mathbf{v}_h=\mathbf{p}_\mathup{aux}$ in (\ref{variational form for inf-sup_d}), we have $-(\textbf{div}\bm{\underline{\alpha}},\mathbf{p}_\mathup{aux})= s(\mathbf{p}_\mathup{aux},\mathbf{p}_\mathup{aux})=\norm{\mathbf{p}_\mathup{aux}}_s^2$. Next, taking $\underline{\bm{\tau}}_h=\underline{\bm{\alpha}}$ and $\mathbf{v}_h=\mathbf{p}_\mathup{aux}$ in (\ref{variational form for inf-sup}), we have
\[
\norm{\underline{\bm{\alpha}}}_\mathcal{A}^2=-(\textbf{div}\underline{\bm{\alpha}},\mathbf{p})=s(\mathbf{p}_\mathup{aux},\mathbf{p})\leq\tilde{C}^{\frac{1}{2}}\norm{\mathbf{p}_\mathup{aux}}_s\cdot\norm{\mathbf{p}}_{L^2},
\]
where $\tilde{C}$ is the maximum value of $\tilde{k}$. Using the inf-sup condition (\ref{discrete inf-sup}), we have
\[
\norm{\mathbf{p}}_{L^2}\leq C\sup_{\underline{\bm{v}}\in\underline{\bm{\Sigma}}_h}\frac{(\textbf{div}\underline{\bm{v}},\mathbf{p})}{\norm{\underline{\bm{v}}}_{H(\textbf{div})}}=C\sup_{\underline{\bm{v}}\in\underline{\bm{\Sigma}}_h}\frac{-(\mathcal{A}\bm{\underline{\alpha}},\bm{\underline{v}})}{\norm{\underline{\bm{v}}}_{H(\textbf{div})}}\leq C\norm{\bm{\underline{\alpha}}}_\mathcal{A},
\]
which implies $\norm{\bm{\underline{\alpha}}}_\mathcal{A}\leq C\tilde{C}^{\frac{1}{2}}\norm{\mathbf{p}_\mathup{aux}}_s$. Taking $\mathbf{v}_h=\nabla\cdot\underline{\bm{\alpha}}\in \bm{\mathit{U}}_h$, then we have
\[
\norm{\nabla\cdot\bm{\underline{\alpha}}}_{L^2}^2=-s(\mathbf{p}_\mathup{aux},\nabla\cdot\bm{\underline{\alpha}})=-\int_\Omega\tilde{k}\mathbf{p}_\mathup{aux}(\nabla\cdot\bm{\underline{\alpha}})\leq \tilde{C}^{\frac{1}{2}}\norm{\mathbf{p}_\mathup{aux}}_s\norm{\nabla\cdot\bm{\underline{\alpha}}}_{L^2}.
\]
Thus we obtain $\norm{\nabla\cdot\bm{\underline{\alpha}}}_{L^2}\leq\tilde{C}^{\frac{1}{2}}\norm{\mathbf{p}_\mathup{aux}}_s$. Then 
$
\norm{\underline{\bm{\alpha}}}_{\mathcal{A},\textbf{div}}^2=\norm{\underline{\bm{\alpha}}}_{\mathcal{A}}^2+\norm{\nabla\cdot\bm{\underline{\alpha}}}_{L^2}^2\leq \tilde{C}(C^2+1)\norm{\mathbf{p}_\mathup{aux}}_s^2.
$Finally, we obtain that 
\[
\norm{\mathbf{p}_\mathup{aux}}_s\cdot\norm{\underline{\bm{\alpha}}}_{\mathcal{A},\textbf{div}}\leq (1+C^2)^{\frac{1}{2}}\tilde{C}^{\frac{1}{2}}\norm{\mathbf{p}_\mathup{aux}}_s^2.
\]
Taking $\underline{\bm{\beta}}=-\underline{\bm{\alpha}}$ and letting $C_g=(1+C^2)^{\frac{1}{2}}\tilde{C}^{\frac{1}{2}}$, we obtain the desired result.
\end{proof}
\begin{lemma}
\label{glo-fine solution}
Let $(\underline{\bm{\sigma}}_\mathup{glo},\mathbf{u}_\mathup{glo})$ be the solution of (\ref{global variational form}) and let $(\underline{\bm{\sigma}}_h,\mathbf{u}_h)$ be the fine solution of (\ref{discrete formula}). Then we have $\underline{\bm{\sigma}}_\mathup{glo}=\underline{\bm{\sigma}}_h$, $\mathbf{u}_\mathup{glo}=\pi\mathbf{u}_h$ and
\[
\norm{\mathbf{u}_h-\mathbf{u}_\mathup{glo}}_s\leq \frac{1}{\sqrt{\Lambda}}(C_Sh^s\norm{\mathbf{f}}_{L^2}+C\frac{1}{\sqrt{\Lambda}}\norm{(I-\pi)(\tilde{k}^{-1}\mathbf{f})}_s+C\norm{\mathbf{f}}_{L^2}),
\]
where $C,C_S>0$ are independent of mesh sizes $h,H$ and the Lam$\rm \acute{e}$ coefficient $\lambda$, but $C_S$ depends on the regularity of the PDE problem (\ref{our pde}) (see Assumption \ref{0128b}).
\end{lemma}
\begin{proof}
By Lemma \ref{characterization of global basis functions}, we see that $\underline{\bm{\sigma}}_h\in\Sigma_\mathup{glo}\subset\Sigma_h$. In terms of Lemma \ref{inf-sup for global variational form}, we know the existence and uniqueness of solution of the system (\ref{global variational form}). Then we have $\underline{\bm{\sigma}}_h=\underline{\bm{\sigma}}_\mathup{glo}$. We note that (\ref{global variational form}) is a conforming approximation to (\ref{discrete formula}), and that $\textbf{div}\Sigma_\mathup{glo}\subset\bm{\mathit{P}}(\tilde{k}U_\mathup{aux})$ by (\ref{global multiscale basis_d}). Then, using  (\ref{global variational form_a}) and (\ref{discrete formula_a}), we have
\[
(\textbf{div}\underline{\bm{\tau}},\mathbf{u}_h-\mathbf{u}_\mathup{glo})=0,\quad \forall \underline{\bm{\tau}}\in \Sigma_\mathup{glo}.
\]
Since $\textbf{div}\Sigma_\mathup{glo}\subset\bm{\mathit{P}}(\tilde{k}U_\mathup{aux})$, we have that $\mathbf{u}_\mathup{glo}=\pi\mathbf{u}_h$. Then $\norm{\mathbf{u}_\mathup{glo}-\mathbf{u}_h}_s=\norm{(I-\pi)\mathbf{u}_h}_s$. We define $\bm{\mu}_i$ to be the restriction of $(I-\pi)\mathbf{u}_h$ on $K_i$. Then by the definition of $\pi$, we can write $\bm{\mu}_i=\sum_{j=l_i+1}^{L_i}c^i_j\mathbf{p}^i_j$. We define $\underline{\bm{w}}_i\in\underline{\bm{\Sigma}}_h(K_i)$ as $\underline{\bm{w}}_i=\sum_{j=l_i+1}^{L_i}(\lambda^i_j)^{-1}c^i_j\underline{\bm{\phi}}^i_j$. Letting $\underline{\bm{\tau}}_h=\underline{\bm{w}}_i$ in (\ref{discrete formula_a}), we have
$
(\mathcal{A}\underline{\bm{\sigma}}_h,\underline{\bm{w}}_i)+(\textbf{div}\underline{\bm{w}}_i,\mathbf{u}_h)=0. 
$
Using the local spectral problem (\ref{local spectral problem}), we have 
$
-(\textbf{div}\underline{\bm{w}}_i,\mathbf{u}_h)=-(\textbf{div}\underline{\bm{w}}_i,(I-\pi)\mathbf{u}_h)=\sum_{j=l_i+1}^{L_i}(c^i_j)^2=\norm{(I-\pi)\mathbf{u}_h}_s^2
$
and
$
\norm{\underline{\bm{w}}_i}_\mathcal{A}^2\leq \frac{1}{\Lambda}\sum_{j=l_i+1}^{L_i}(c^i_j)^2=\frac{1}{\Lambda}\norm{(I-\pi)\mathbf{u}_h}_s^2.
$
Combining above three relationships, we have
$
\norm{(I-\pi)\mathbf{u}_h}_s^2\leq \norm{\underline{\bm{\sigma}}_h}_\mathcal{A}\cdot\frac{1}{\sqrt{\Lambda}}\norm{(I-\pi)\mathbf{u}_h}_s
$
Therefore, we have
\[
\norm{(I-\pi)\mathbf{u}_h}_s\leq \frac{1}{\sqrt{\Lambda}}\norm{\underline{\bm{\sigma}}_h}_\mathcal{A}.
\]
Finally, by Lemma \ref{convergence form the fine solution}, we have
\begin{align*}
\norm{\mathbf{u}_\mathup{glo}-\mathbf{u}_h}_s&\leq \frac{1}{\sqrt{\Lambda}}(\norm{\underline{\bm{\sigma}}_h-\underline{\bm{\sigma}}}_\mathcal{A}+\norm{\underline{\bm{\sigma}}}_\mathcal{A})\leq C_S\frac{1}{\sqrt{\Lambda}}h^s\norm{\mathbf{f}}_{L^2}+C\frac{1}{\Lambda}\norm{(I-\pi)(\tilde{k}^{-1}\mathbf{f})}_s+C\frac{1}{\sqrt{\Lambda}}\norm{\underline{\bm{\sigma}}}_\mathcal{A}.
\end{align*}
By (\ref{weak form for our pde}) and the Korn's inequality (see, for instance, \cite{PG_2013}), we can obtain
\begin{align*}
(\mathcal{A}\underline{\bm{\sigma}},\underline{\bm{\sigma}})&=-(\mathbf{f},\mathbf{u})\leq \norm{\mathbf{f}}_{L^2}\cdot\norm{\mathbf{u}}_{L^2}\leq \norm{\mathbf{f}}_{L^2}\cdot\norm{\underline{\bm{\epsilon}}(\mathbf{u})}_{L^2}\leq C\norm{\mathbf{f}}_{L^2}\cdot\norm{\underline{\bm{\sigma}}}_{\mathcal{A}}.
\end{align*}
Thus, we have $\norm{\underline{\bm{\sigma}}}_\mathcal{A}\leq C\norm{\mathbf{f}}_{L^2}$. So we get there exists $C,C_S>0$ independent of $h$ and the Lame constant $\lambda$ ($C_S$ depends on the regularity of the PDE problem (\ref{our pde})) such that
\[
\norm{\mathbf{u}_\mathup{glo}-\mathbf{u}_h}_s\leq \frac{1}{\sqrt{\Lambda}}(C_Sh^s\norm{\mathbf{f}}_{L^2}+C\frac{1}{\sqrt{\Lambda}}\norm{(I-\pi)(\tilde{k}^{-1}\mathbf{f})}_s+C\norm{\mathbf{f}}_{L^2}).
\]
\end{proof}
\subsection{Well-posedness of global and multiscale basis functions}
\label{Well-posedness of global and multiscale basis functions}
We discuss the well-posedness of the multiscale basis functions defined in (\ref{multiscale basis}) in the following. For simplicity, let $k=\lambda+2\mu$ be piecewise constant over the fine mesh with highly heterogeneous and high contrast. % Recall the norm $\norm{\underline{\bm{v}}}_\Sigma^2=\norm{\underline{\bm{v}}}_\mathcal{A}^2+\norm{\tilde{k}^{-\frac{1}{2}}\nabla\cdot\underline{\bm{v}}}_{L^2}^2$.
\begin{lemma}
There exists a constant $C_\star$ such that $\forall$ $(\underline{\bm{\sigma}}_h,\mathbf{u}_h)\in\underline{\bm{\Sigma}}_h(K_i^+)\times \bm{\mathit{U}}_h(K_i^+)$, $\exists$ a pair $(\underline{\bm{w}}_h,\mathbf{p}_h)\in\underline{\bm{\Sigma}}_h(K_i^+)\times \bm{\mathit{U}}_h(K_i^+)$ such that
\[
\norm{(\underline{\bm{\sigma}}_h,\mathbf{u}_h)}\leq C_\star\frac{A((\underline{\bm{\sigma}}_h,\mathbf{u}_h),(\underline{\bm{w}}_h,\mathbf{p}_h))}{\norm{(\underline{\bm{w}}_h,\mathbf{p}_h)}},
\]
where $A((\underline{\bm{\sigma}}_h,\mathbf{u}_h),(\underline{\bm{w}}_h,\mathbf{p}_h))=(\mathcal{A}\underline{\bm{\sigma}}_h,\underline{\bm{w}}_h)+(\textbf{\textup{div}}\underline{\bm{w}}_h,\mathbf{u}_h)+(\textbf{\textup{div}}\underline{\bm{\sigma}}_h,\mathbf{p}_h)-s(\pi\mathbf{u}_h,\pi\mathbf{p}_h)$ and $\norm{(\underline{\bm{\sigma}}_h,\mathbf{u}_h)}^2=\norm{\underline{\bm{\sigma}}_h}_\mathcal{A}^2+\norm{\mathbf{u}_h}_s^2$.
\end{lemma}
\begin{proof}
Let $\mathbf{q}_h\in\bm{\mathit{U}}_h(K_i^+)$. Consider a coarse element $K_i\subset K_i^+$ and $\mathbf{q}^i_h=\mathbf{q}_h|_{K_i}$. Note that, we can write $\mathbf{q}^i_h=\sum_{j=1}^{L_i}c^i_j\mathbf{p}^i_j$ since the set of eigenfunctions $\{\mathbf{p}^i_j\}_{j=1}^{L_i}$ forms a basis for the space $\bm{\mathit{U}}_h(K_i)$. And we define $\underline{\bm{v}}_h^i=\sum_{j=l_i+1}^{L_i}c^i_j/\lambda^i_j\underline{\bm{\phi}}^i_j$ and $\underline{\bm{v}}_h=\sum_i\underline{\bm{v}}_h^i$, where the sum is taken over all $K_i\subset K_i^+$. Note that, by the orthogonality of eigenfunctions,
\begin{align*}
\norm{\underline{\bm{v}}_h}_\mathcal{A}^2&=\sum_{K_i\in K_i^+}\sum_{j=l_i+1}^{L_i}(c^i_j/\lambda^i_j)^2\norm{\underline{\bm{\phi}}^i_j}_\mathcal{A}^2=\sum_{K_i\in K_i^+}\sum_{j=l_i+1}^{L_i}(c^i_j/\lambda^i_j)^2\lambda^i_j\norm{\mathbf{p}^i_j}_s^2\leq \frac{1}{\Lambda}\norm{(I-\pi)\mathbf{q}_h}_s^2.
\end{align*}
On the other hand, by the local spectral problem (\ref{local spectral problem}), we have $\nabla\cdot\underline{\bm{v}}_h^i=-\sum_{j=l_i+1}^{L_i}c^i_j\tilde{k}\mathbf{p}^i_j$. So we obtain that
$
\norm{\tilde{k}^{-\frac{1}{2}}\nabla\cdot\underline{\bm{v}}_h}_{L^2}^2=\sum_{K_i\subset K_i^+}\norm{\tilde{k}^{-\frac{1}{2}}\nabla\cdot\underline{\bm{v}}_h^i}_{L^2}^2=\norm{(I-\pi)\mathbf{q}_h}_s^2.
$
Thus, by above estimates and the definition of $\norm{\cdot}_\Sigma$ we have
\begin{align}
\label{v_h Sigma norm}
\norm{\underline{\bm{v}}_h}_\Sigma^2\leq(1+\frac{1}{\Lambda})\norm{(I-\pi)\mathbf{q}_h}_s^2.
\end{align}
Next, $(\textbf{div}\underline{\bm{v}}_h,\mathbf{q}_h)=\sum_{K_i\subset K_i^+}\int_{K_i}\sum_{j=l_i+1}^{L_i}c^i_j\tilde{k}\mathbf{p}^i_j\mathbf{q}_h^i=\norm{(I-\pi)\mathbf{q}_h}_s^2$, which implies
\[
\frac{(\textbf{div}\underline{\bm{v}}_h,\mathbf{q}_h)^2}{\norm{\underline{\bm{v}}_h}_\Sigma^2}\geq \frac{\norm{(I-\pi)\mathbf{q}_h}_s^4}{(1+\frac{1}{\Lambda})\norm{(I-\pi)\mathbf{q}_h}_s^2}=(1+\frac{1}{\Lambda})^{-1}\norm{(I-\pi)\mathbf{q}_h}_s^2.
\]
Therefore, for any $\mathbf{q}_h\in\bm{\mathit{U}}_h(K_i^+)$, there is $\underline{\bm{v}}_h\in\underline{\bm{\Sigma}}_h(K_i^+)$ such that
\begin{align}
\label{inf-sup for global basis}
\norm{\mathbf{q}_h}_s^2=\norm{(I-\pi)\mathbf{q}_h}_s^2+\norm{\pi\mathbf{q}_h}_s^2\leq (1+\frac{1}{\Lambda})\frac{(\textbf{div}\underline{\bm{v}}_h,\mathbf{q}_h)^2}{\norm{\underline{\bm{v}}_h}_\Sigma^2}+\norm{\pi\mathbf{q}_h}_s^2.
\end{align}
By the definition of $A((\underline{\bm{\sigma}}_h,\mathbf{u}_h),(\underline{\bm{w}}_h,\mathbf{p}_h))$, 
 we have $\norm{\underline{\bm{\sigma}}_h}_\mathcal{A}^2+\norm{\pi(\mathbf{p}_h)}_s^2=A((\underline{\bm{\sigma}}_h,\mathbf{u}_h),(\underline{\bm{\sigma}}_h,-\mathbf{u}_h))$. In terms of the proof of (\ref{v_h Sigma norm}), we know $\exists\underline{\bm{y}}_h$ such that $\norm{\underline{\bm{y}}_h}_\Sigma=\norm{(I-\pi)\mathbf{u}_h}_s$ and
$
\norm{\mathbf{u}_h}_s^2\leq (1+\frac{1}{\Lambda})\frac{(\textbf{div}\underline{\bm{y}}_h,\mathbf{u}_h)^2}{\norm{\underline{\bm{y}}_h}_\Sigma^2}+\norm{\pi\mathbf{u}_h}_s^2.
$
Then we can obtain
\begin{align*}
A((\underline{\bm{\sigma}}_h,\mathbf{u}_h),(\underline{\bm{y}}_h,\mathbf{0}))
%&=(\mathcal{A}\underline{\bm{\sigma}}_h,\underline{\bm{y}}_h)+(\textbf{div}\underline{\bm{y}}_h,\mathbf{u}_h)\\
%&\geq (\mathcal{A}\underline{\bm{\sigma}}_h,\underline{\bm{y}}_h)+(1+\frac{1}{\Lambda})^{-1}\norm{(I-\pi)\mathbf{u}_h}_s\cdot\norm{\underline{\bm{y}}_h}_\Sigma\\
\geq ((1+\frac{1}{\Lambda})^{-1}\norm{(I-\pi)\mathbf{u}_h}_s-\norm{\underline{\bm{\sigma}}_h}_\mathcal{A})\norm{(I-\pi)\mathbf{u}_h}_s
\end{align*}
and
\begin{align*}
A((\underline{\bm{\sigma}}_h,\mathbf{u}_h),(\mathbf{0},\tilde{k}^{-1}\nabla\cdot\underline{\bm{\sigma}}_h))
%&=\norm{\tilde{k}^{-\frac{1}{2}}}_{L^2}^2-s(\pi(\mathbf{u}_h),\tilde{k}^{-1}\nabla\cdot\underline{\bm{\sigma}}_h)\\
\geq \norm{\tilde{k}^{-\frac{1}{2}}}_{L^2}(\norm{\tilde{k}^{-\frac{1}{2}}}_{L^2}-\norm{\pi(\mathbf{u}_h)}_s).
\end{align*}
Next, we take $(\underline{\bm{w}}_h,\mathbf{p}_h)=(\underline{\bm{\sigma}}_h+\gamma\underline{\bm{y}}_h,-\mathbf{u}_h+\tilde{k}^{-1}\nabla\cdot\underline{\bm{\sigma}}_h)$, where $\gamma\in\mathbb{R}$ is a constant. Then by a simple computation, we have
\begin{align*}
A((\underline{\bm{\sigma}}_h,\mathbf{u}_h),(\underline{\bm{w}}_h,\mathbf{p}_h))
\geq \norm{\underline{\bm{\sigma}}_h}_\mathcal{A}^2+\norm{\pi(\mathbf{u}_h)}_s^2+\norm{\tilde{k}^{-\frac{1}{2}}\nabla\cdot\underline{\bm{\sigma}}_h}_{L^2}^2+\gamma(2(1+\frac{1}{\Lambda})^{-1}-\gamma)\norm{(I-\pi)\mathbf{u}_h}_s^2.
\end{align*}
Let $\gamma=(1+\frac{1}{\Lambda})^{-1}$, then we obtain
$
A((\underline{\bm{\sigma}}_h,\mathbf{u}_h),(\underline{\bm{w}}_h,\mathbf{p}_h))\geq \norm{\underline{\bm{\sigma}}_h}_\mathcal{A}^2+\norm{\pi(\mathbf{u}_h)}_s^2+\norm{\tilde{k}^{-\frac{1}{2}}\nabla\cdot\underline{\bm{\sigma}}_h}_{L^2}^2+1+\frac{1}{\Lambda})^{-2}\norm{(I-\pi)\mathbf{u}_h}_s^2
$.
Since $\norm{(\underline{\bm{w}}_h,\mathbf{p}_h)}\leq 2\norm{(\underline{\bm{\sigma}}_h,\mathbf{u}_h)}+\norm{\tilde{k}^{-\frac{1}{2}}\nabla\cdot\underline{\bm{\sigma}}_h}_{L^2}+C\norm{(I-\pi)\mathbf{u}_h}_s$, we obtain that there exist a constant $C_\star$ such that
$
\norm{(\underline{\bm{w}}_h,\mathbf{p}_h)}\cdot\norm{(\underline{\bm{\sigma}}_h,\mathbf{u}_h)}\leq C_\star A((\underline{\bm{\sigma}}_h,\mathbf{u}_h),(\underline{\bm{w}}_h,\mathbf{p}_h)).
$
\end{proof}
\begin{lemma}
\label{decay_property_1}
Let $(\bm{\underline{\psi}}^i_j,\mathbf{q}^i_j)\in\underline{\bm{\mathit{\Sigma}}}_h\times\bm{\mathit{U}}_h$ be the solution of (\ref{global multiscale basis}) and let $(\bm{\underline{\psi}}^i_{j,\mathup{ms}},\mathbf{q}^i_{j,ms})\in\underline{\bm{\mathit{\Sigma}}}_h(K_i)\times\bm{\mathit{U}}_h(K_i)$ be the solution of (\ref{multiscale basis}). Then we have the following approximation property
\[
\norm{\bm{\underline{\psi}}^i_j-\bm{\underline{\psi}}^i_{j,\mathup{ms}}}_\Sigma^2+\norm{\mathbf{q}^i_j-\mathbf{q}^i_{j,ms}}_s^2\leq C(1+\frac{1}{\Lambda})(\norm{\bm{\underline{\psi}}^i_j-\underline{\bm{\tau}}_h}_\Sigma^2+\norm{\mathbf{q}^i_j-\mathbf{v}_h}_s^2),
\]
for all $(\underline{\bm{\tau}}_h,\mathbf{v}_h)\in\underline{\bm{\mathit{\Sigma}}}_h(K_i^+)\times\bm{\mathit{U}}_h(K_i^+)$.
\end{lemma}
\begin{proof}
Combining (\ref{global multiscale basis}) and (\ref{multiscale basis}), we have
\begin{subequations}
	\label{global-local}
	\begin{align}		
		(\mathcal{A}(\bm{\underline{\psi}}^i_j-\bm{\underline{\psi}}^i_{j,\mathup{ms}}),\underline{\bm{\tau}}_h)+(\textbf{div}\underline{\bm{\tau}}_h,\mathbf{q}^i_j-\mathbf{q}^i_{j,\mathup{ms}})&=0  \quad \forall \underline{\bm{\tau}}_h\in \underline{\bm{\mathit{\Sigma}}}_h(K_i^+), \label{global-local_a}\\
s(\pi(\mathbf{q}^i_j-\mathbf{q}^i_{j,\mathup{ms}}),\pi\mathbf{v}_h)-(\textbf{div}(\bm{\underline{\psi}}^i_j-\bm{\underline{\psi}}^i_{j,\mathup{ms}}),\mathbf{v}_h)&= 0\quad \forall \mathbf{v}_h\in \bm{\mathit{U}}_h(K_i^+).\label{global-local_d}
	\end{align}	
\end{subequations}
Then, for all $(\underline{\bm{\tau}}_h,\mathbf{v}_h)\in\underline{\bm{\mathit{\Sigma}}}_h(K_i^+)\times\bm{\mathit{U}}_h(K_i^+)$, using (\ref{global-local}), we have
\begin{align*}
&\quad(\mathcal{A}(\bm{\underline{\psi}}^i_j-\bm{\underline{\psi}}^i_{j,\mathup{ms}}),\bm{\underline{\psi}}^i_j-\bm{\underline{\psi}}^i_{j,\mathup{ms}})+s(\pi(\mathbf{q}^i_j-\mathbf{q}^i_{j,\mathup{ms}}),\pi(\mathbf{q}^i_j-\mathbf{q}^i_{j,\mathup{ms}}))\\
&=(\mathcal{A}(\bm{\underline{\psi}}^i_j-\bm{\underline{\psi}}^i_{j,\mathup{ms}}),\bm{\underline{\psi}}^i_j-\underline{\bm{\tau}}_h)+(\mathcal{A}(\bm{\underline{\psi}}^i_j-\bm{\underline{\psi}}^i_{j,\mathup{ms}}),\underline{\bm{\tau}}_h-\bm{\underline{\psi}}^i_{j,\mathup{ms}})\\
&\quad+s(\pi(\mathbf{q}^i_j-\mathbf{q}^i_{j,\mathup{ms}}),\pi(\mathbf{q}^i_j-\mathbf{v}_h))+s(\pi(\mathbf{q}^i_j-\mathbf{q}^i_{j,\mathup{ms}}),\pi(\mathbf{v}_h-\mathbf{q}^i_{j,\mathup{ms}}))\\
&=(\mathcal{A}(\bm{\underline{\psi}}^i_j-\bm{\underline{\psi}}^i_{j,\mathup{ms}}),\bm{\underline{\psi}}^i_j-\underline{\bm{\tau}}_h)+s(\pi(\mathbf{q}^i_j-\mathbf{q}^i_{j,\mathup{ms}}),\pi(\mathbf{q}^i_j-\mathbf{v}_h))\\
&\quad-(\textbf{div}(\underline{\bm{\tau}}_h-\bm{\underline{\psi}}^i_{j,\mathup{ms}}),\mathbf{q}^i_j-\mathbf{q}^i_{j,\mathup{ms}})+(\textbf{div}(\bm{\underline{\psi}}^i_j-\bm{\underline{\psi}}^i_{j,\mathup{ms}}),\mathbf{v}_h-\mathbf{q}^i_{j,\mathup{ms}}).
\end{align*}
Note that the first two terms on the right hand side of above equality can be estimated easily. For the other two terms, we observe that
\begin{equation}
\begin{aligned}
\label{two terms}
&\quad-(\textbf{div}(\underline{\bm{\tau}}_h-\bm{\underline{\psi}}^i_{j,\mathup{ms}}),\mathbf{q}^i_j-\mathbf{q}^i_{j,\mathup{ms}})+(\textbf{div}(\bm{\underline{\psi}}^i_j-\bm{\underline{\psi}}^i_{j,\mathup{ms}}),\mathbf{v}_h-\mathbf{q}^i_{j,\mathup{ms}})\\
&=-(\textbf{div}(\underline{\bm{\tau}}_h-\bm{\underline{\psi}}^i_j),\mathbf{q}^i_j-\mathbf{q}^i_{j,\mathup{ms}})-(\textbf{div}(\bm{\underline{\psi}}^i_j-\bm{\underline{\psi}}^i_{j,\mathup{ms}}),\mathbf{q}^i_j-\mathbf{q}^i_{j,\mathup{ms}})\\
&\quad +(\textbf{div}(\bm{\underline{\psi}}^i_j-\bm{\underline{\psi}}^i_{j,\mathup{ms}}),\mathbf{v}_h-\mathbf{q}^i_j)+(\textbf{div}(\bm{\underline{\psi}}^i_j-\bm{\underline{\psi}}^i_{j,\mathup{ms}}),\mathbf{q}^i_j-\mathbf{q}^i_{j,\mathup{ms}})\\
&=-(\textbf{div}(\underline{\bm{\tau}}_h-\bm{\underline{\psi}}^i_j),\mathbf{q}^i_j-\mathbf{q}^i_{j,\mathup{ms}})+(\textbf{div}(\bm{\underline{\psi}}^i_j-\bm{\underline{\psi}}^i_{j,\mathup{ms}}),\mathbf{v}_h-\mathbf{q}^i_j).
\end{aligned}
\end{equation}
We will estimate the two terms on the right hand side of (\ref{two terms}). For the first term on the right hand side of (\ref{two terms}), by using the Cauchy-Schwarz inequality and the triangle inequality,
\[
(\textbf{div}(\underline{\bm{\tau}}_h-\bm{\underline{\psi}}^i_j),\mathbf{q}^i_j-\mathbf{q}^i_{j,\mathup{ms}})\leq \norm{\underline{\bm{\tau}}_h-\bm{\underline{\psi}}^i_j}_\Sigma\cdot(\norm{\mathbf{q}^i_j-\mathbf{v}_h}_s+\norm{\mathbf{v}_h-\mathbf{q}^i_{j,\mathup{ms}}}_s).
\]
By (\ref{inf-sup for global basis}), there is $\underline{\bm{w}}_h\in\underline{\bm{\Sigma}}_h(K_i^+)$ such that
$
\norm{\mathbf{v}_h-\mathbf{q}^i_{j,\mathup{ms}}}_s^2\leq (1+\frac{1}{\Lambda})\frac{(\textbf{div}\underline{\bm{w}}_h,\mathbf{v}_h-\mathbf{q}^i_{j,\mathup{ms}})^2}{\norm{\underline{\bm{w}}_h}_\Sigma^2}+\norm{\pi(\mathbf{v}_h-\mathbf{q}^i_{j,\mathup{ms}})}_s^2.
$
Note that $\norm{\pi(\mathbf{v}_h-\mathbf{q}^i_{j,\mathup{ms}})}_s\leq\norm{\mathbf{v}_h-\mathbf{q}^i_j}_s+\norm{\pi(\mathbf{q}^i_j-\mathbf{q}^i_{j,\mathup{ms}})}_s$.
In addition, by (\ref{global-local_a}), we obtain
\begin{align*}
(\textbf{div}\underline{\bm{w}}_h,\mathbf{v}_h-\mathbf{q}^i_{j,\mathup{ms}})
=(\textbf{div}\underline{\bm{w}}_h,\mathbf{v}_h-\mathbf{q}^i_j)-(\mathcal{A}(\bm{\underline{\psi}}^i_j-\bm{\underline{\psi}}^i_{j,\mathup{ms}}),\underline{\bm{w}}_h).
\end{align*}
Thus we have
\begin{equation}
\begin{aligned}
\label{estimate for the first term}
&\quad(\textbf{div}(\underline{\bm{\tau}}_h-\bm{\underline{\psi}}^i_j),\mathbf{q}^i_j-\mathbf{q}^i_{j,\mathup{ms}})\leq \norm{\underline{\bm{\tau}}_h-\bm{\underline{\psi}}^i_j}_\Sigma\cdot\big(\norm{\mathbf{q}^i_j-\mathbf{v}_h}_s+(1+\frac{1}{\Lambda})^{\frac{1}{2}}(\norm{\mathbf{v}_h-\mathbf{q}^i_j}_s+\norm{\bm{\underline{\psi}}^i_j-\bm{\underline{\psi}}^i_{j,\mathup{ms}}}_\mathcal{A})\\
&+\norm{\pi(\mathbf{v}_h-\mathbf{q}^i_{j,\mathup{ms}})}_s\big)\leq C(1+\frac{1}{\Lambda})^{\frac{1}{2}}\norm{\underline{\bm{\tau}}_h-\bm{\underline{\psi}}^i_j}_\Sigma\cdot(\norm{\mathbf{q}^i_j-\mathbf{v}_h}_s+\norm{\bm{\underline{\psi}}^i_j-\bm{\underline{\psi}}^i_{j,\mathup{ms}}}_\Sigma+\norm{\pi(\mathbf{v}_h-\mathbf{q}^i_{j,\mathup{ms}})}_s)
\end{aligned}
\end{equation}
This gives an estimate for the first term on the right hand side of (\ref{two terms}). 

For the second term on the right hand side of  (\ref{two terms}), we have
\[
(\textbf{div}(\bm{\underline{\psi}}^i_j-\bm{\underline{\psi}}^i_{j,\mathup{ms}}),\mathbf{v}_h-\mathbf{q}^i_j)\leq \norm{\tilde{k}^{-\frac{1}{2}}\nabla\cdot(\bm{\underline{\psi}}^i_j-\bm{\underline{\psi}}^i_{j,\mathup{ms}})}_{L^2}\cdot\norm{\mathbf{v}_h-\mathbf{q}^i_j}_s.
\]
By (\ref{global-local_d}), let $\mathbf{v}_h=\tilde{k}^{-1}\nabla\cdot(\bm{\underline{\psi}}^i_j-\bm{\underline{\psi}}^i_{j,\mathup{ms}})$ (Note that (\ref{global-local_d}) also holds for all $\mathbf{v}_h\in\bm{\mathit{U}}_h$)
\begin{align*}
\norm{\tilde{k}^{-\frac{1}{2}}\nabla\cdot(\bm{\underline{\psi}}^i_j-\bm{\underline{\psi}}^i_{j,\mathup{ms}})}_{L^2}^2&=s(\pi(\mathbf{q}^i_j-\mathbf{q}^i_{j,\mathup{ms}}),\tilde{k}^{-1}\nabla\cdot(\bm{\underline{\psi}}^i_j-\bm{\underline{\psi}}^i_{j,\mathup{ms}}))\\
&\leq \norm{\pi(\mathbf{q}^i_j-\mathbf{q}^i_{j,\mathup{ms}})}_s\cdot\norm{\tilde{k}^{-\frac{1}{2}}\nabla\cdot(\bm{\underline{\psi}}^i_j-\bm{\underline{\psi}}^i_{j,\mathup{ms}})}_{L^2}.
\end{align*}
So we obtain that
$
(\textbf{div}(\bm{\underline{\psi}}^i_j-\bm{\underline{\psi}}^i_{j,\mathup{ms}}),\mathbf{v}_h-\mathbf{q}^i_j)\leq\norm{\pi(\mathbf{q}^i_j-\mathbf{q}^i_{j,\mathup{ms}})}_s\cdot\norm{\mathbf{v}_h-\mathbf{q}^i_j}_s.
$
And we have
\begin{align*}
&\quad \norm{\bm{\underline{\psi}}^i_j-\quad\bm{\underline{\psi}}^i_{j,\mathup{ms}}}_\Sigma^2=\norm{\bm{\underline{\psi}}^i_j-\bm{\underline{\psi}}^i_{j,\mathup{ms}}}_\mathcal{A}^2+\norm{\tilde{k}^{-\frac{1}{2}}\nabla\cdot(\bm{\underline{\psi}}^i_j-\bm{\underline{\psi}}^i_{j,\mathup{ms}})}_{L^2}^2\\
&\leq \norm{\bm{\underline{\psi}}^i_j-\bm{\underline{\psi}}^i_{j,\mathup{ms}}}_\mathcal{A}^2+\norm{\pi(\mathbf{q}^i_j-\mathbf{q}^i_{j,\mathup{ms}})}_s^2\leq \norm{\bm{\underline{\psi}}^i_j-\bm{\underline{\psi}}^i_{j,\mathup{ms}}}_\mathcal{A}^2+\norm{\mathbf{q}^i_j-\mathbf{q}^i_{j,\mathup{ms}}}_s^2.
\end{align*}
Combining above estimates, we get the desired results by applying the Cauchy-Schwarz inequality.
\end{proof}
\subsection{Decay property of global basis functions}\label{Decay property of global basis functions}
In this section, we will prove a decay property of the global basis functions $\underline{\bm{\psi}}^i_j$ and $\mathbf{q}^i_j$. Let $K_i$ be a given coarse element. We define $K_{i,m}$ as the oversampling coarse neighborhood of enlarging $K_i$ by $m$ coarse grid layer. For $M>m$, we define $\chi^{M,m}_i\in\mathup{span}\{\chi_i\}$ (where $\{\chi_i\}$ is the set of standard multiscale basis functions, which satisfy the partition of unity property) such that
\begin{align*}
\chi^{M,m}_i&=1\quad\mathup{in}\quad K_{i,m},\\
\chi^{M,m}_i&=0\quad\mathup{in}\quad \Omega\backslash K_{i,M}.
\end{align*}
We notice that $\abs{\nabla\chi^{M,m}_i}^2\leq CH^{-2}$. In the following lemma, we will prove the decay property using $K_i^+=K_{i,l}$, that is, the basis functions are constructed in a region which is a $l$ coarse-grid layer extension of $K_i$, with $l\geq2$. For simplicity, we assume $\mathcal{A}$ is a piecewise constant. In addition, we define a $L^2$ projection operator  $\underline{\bm{P}}$ to $\underline{\bm{\mathit{\Sigma}}}_h$ as follows:
\[
\underline{\bm{P}}\colon\underline{\bm{L}}^2(\Omega)\to
% \Sigma_\mathup{ms}\subset
\underline{\bm{\mathit{\Sigma}}}_h.
\]
Then we clearly have
\begin{equation}
\label{L2 projection to discrete stress space}
\norm{\underline{\bm{P}}\underline{v}-\underline{v}}_{L^2}\leq C\norm{\underline{v}}_{L^2},\quad
\norm{\underline{\bm{P}}\underline{v}}_{L^2}\leq C\norm{\underline{v}}_{L^2},
\end{equation}
where $C$ is independent of $h$.
\begin{lemma}
\label{decay property of basis function_2}
Let $(\bm{\underline{\psi}}^i_j,\mathbf{q}^i_j)\in\underline{\bm{\mathit{\Sigma}}}_h\times\bm{\mathit{U}}_h$ be the solution of (\ref{global multiscale basis}) and let $(\bm{\underline{\psi}}^i_{j,\mathup{ms}},\mathbf{q}^i_{j,ms})\in\underline{\bm{\mathit{\Sigma}}}_h(K_i)\times\bm{\mathit{U}}_h(K_i)$ be the solution of (\ref{multiscale basis}). For $K_{i,l}$ with $l\geq2$, we have
\[
\norm{\bm{\underline{\psi}}^i_j-\bm{\underline{\psi}}^i_{j,\mathup{ms}}}_\Sigma^2+\norm{\mathbf{q}^i_j-\mathbf{q}^i_{j,ms}}_s^2\leq C_\#\norm{\mathbf{p}^i_j}_s^2,
\]
where $C_\#=C(1+\frac{1}{\Lambda})(1+C^{-1}(1+\frac{1}{\Lambda})^{-\frac{1}{2}})^{1-l}$.
\end{lemma}
\begin{proof}
By lemma \ref{decay_property_1}, 
\begin{equation}
\label{decay inequality}
\norm{\bm{\underline{\psi}}^i_j-\bm{\underline{\psi}}^i_{j,\mathup{ms}}}_\Sigma^2+\norm{\mathbf{q}^i_j-\mathbf{q}^i_{j,ms}}_s^2\leq C(1+\frac{1}{\Lambda})(\norm{\bm{\underline{\psi}}^i_j-\underline{\bm{\tau}}_h}_\Sigma^2+\norm{\mathbf{q}^i_j-\mathbf{v}_h}_s^2)
\end{equation}
for all $(\underline{\bm{\tau}}_h,\mathbf{v}_h)\in\underline{\bm{\mathit{\Sigma}}}_h(K_i^+)\times\bm{\mathit{U}}_h(K_i^+)$. Next, we will choose $\underline{\bm{\tau}}_h$ and $\mathbf{v}_h$ in order to obtain the required result. We define
\[
\underline{\bm{\tau}}_h=\underline{\bm{P}}(\chi^{l,l-1}_i\underline{\bm{\psi}}^i_j),\quad \mathbf{v}_h=\bm{\mathit{P}}(\chi^{l,l-1}_i\mathbf{q}^i_j).
\]
Then we see that $\underline{\bm{\tau}}_h\in\underline{\bm{\Sigma}}_h(K_{i,l})$ and $\mathbf{v}_h\in \bm{\mathit{U}}_h(K_{i,l})$. 

\paragraph{\textbf{Step 1:}} 
In the first step, we will show that 
\[
\norm{\bm{\underline{\psi}}^i_j-\bm{\underline{\psi}}^i_{j,\mathup{ms}}}_\Sigma^2+\norm{\mathbf{q}^i_j-\mathbf{q}^i_{j,ms}}_s^2\leq C(\norm{\bm{\underline{\psi}}^i_j}_{\mathcal{A}(\Omega\backslash K_{i,l-1})}^2+\norm{\mathbf{q}^i_j}_{s(\Omega\backslash K_{i,l-1})}^2),
\]
where $C$ is independent of mesh sizes and the Lame constant $\lambda$. By (\ref{decay inequality}), we need to estimate $\norm{\bm{\underline{\psi}}^i_j-\underline{\bm{\tau}}_h}_\Sigma^2$ and $\norm{\mathbf{q}^i_j-\mathbf{v}_h}_s^2$. By the properties of $\chi^{l,l-1}_i$, the interpolation operators $\bm{\mathit{P}}$ and $\underline{\bm{P}}$ , we see that
\begin{align*}
&\quad\norm{\mathbf{q}^i_j-\mathbf{v}_h}_s=\norm{\mathbf{q}^i_j-\bm{\mathit{P}}(\chi^{l,l-1}_i\mathbf{q}^i_j)}_s=\norm{\bm{\mathit{P}}\big((1-\chi^{l,l-1}_i)\mathbf{q}^i_j\big)}_{s(\Omega\backslash K_{i,l-1})}\\
&=\norm{\tilde{k}^{\frac{1}{2}}\bm{\mathit{P}}\big((1-\chi^{l,l-1}_i)\mathbf{q}^i_j\big)}_{L^2(\Omega\backslash K_{i,l-1})}=\norm{\bm{\mathit{P}}\big(\tilde{k}^{\frac{1}{2}}(1-\chi^{l,l-1}_i)\mathbf{q}^i_j\big)}_{L^2(\Omega\backslash K_{i,l-1})}\\
&\leq \norm{\tilde{k}^{\frac{1}{2}}(1-\chi^{l,l-1}_i)\mathbf{q}^i_j}_{L^2(\Omega\backslash K_{i,l-1})}= \norm{(1-\chi^{l,l-1}_i)\mathbf{q}^i_j}_{s(\Omega\backslash K_{i,l-1})}\leq \norm{\mathbf{q}^i_j}_{s(\Omega\backslash K_{i,l-1})}
\end{align*}
in terms of (\ref{ll}). Similarly, by (\ref{L2 projection to discrete stress space}), we have
\begin{align*}
\norm{\bm{\underline{\psi}}^i_j-\underline{\bm{\tau}}_h}_\mathcal{A}\leq \norm{\underline{\bm{\psi}}^i_j-\chi^{l,l-1}_i\underline{\bm{\psi}}^i_j}_\mathcal{A}+\norm{\chi^{l,l-1}_i\underline{\bm{\psi}}^i_j-\underline{\bm{P}}(\chi^{l,l-1}_i\underline{\bm{\psi}}^i_j)}_{\mathcal{A}(\Omega\backslash K_{i,l-1})}\leq C\norm{\underline{\bm{\psi}}^i_j}_{\mathcal{A}(\Omega\backslash K_{i,l-1})}.
\end{align*}
By the inverse inequality, we obtain
\begin{align*}
&\quad\norm{\tilde{k}^{-\frac{1}{2}}\nabla\cdot(\bm{\underline{\psi}}^i_j-\bm{\underline{\tau}}_h)}_{L^2}=\norm{\tilde{k}^{-\frac{1}{2}}\nabla\cdot\big(\bm{\underline{\psi}}^i_j-\underline{\bm{P}}(\chi^{l,l-1}_i\underline{\bm{\psi}}^i_j)\big)}_{L^2}=\norm{\tilde{k}^{-\frac{1}{2}}\nabla\cdot\big(\underline{\bm{P}}(1-\chi^{l,l-1}_i)\underline{\bm{\psi}}^i_j\big)}_{L^2}\\
&\leq CH^{-1}\norm{\tilde{k}^{-\frac{1}{2}}\underline{\bm{P}}(1-\chi^{l,l-1}_i)\underline{\bm{\psi}}^i_j}_{L^2}\leq CH^{-1}\norm{\tilde{k}^{-\frac{1}{2}}(1-\chi^{l,l-1}_i)\underline{\bm{\psi}}^i_j}_{L^2}\leq C\norm{\underline{\bm{\psi}}^i_j}_{\mathcal{A}(\Omega\backslash K_{i,l-1})}
\end{align*}
since $\tilde{k}^{-\frac{1}{2}}=k^{-\frac{1}{2}}H=1/\sqrt{\lambda+2\mu}H$, and $\norm{\frac{1}{\sqrt{\lambda+2\mu}}\underline{\bm{\psi}}^i_j}_{L^2(\Omega\backslash K_{i,l-1)}}^2\leq\norm{\underline{\bm{\psi}}^i_j}_{\mathcal{A}(\Omega\backslash K_{i,l-1)}}^2$. Here $C$ is independent of mesh sizes and the Lame constant $\lambda$. This completes the proof of Step 1.

\paragraph{\textbf{Step 2:}} In the second step, we will show that
\[
\norm{\bm{\underline{\psi}}^i_j}_{\mathcal{A}(\Omega\backslash K_{i,l-1})}^2+\norm{\mathbf{q}^i_j}_{s(\Omega\backslash K_{i,l-1})}^2\leq C(1+\frac{1}{\Lambda})(\norm{\underline{\bm{\psi}}^i_j}_{\mathcal{A}(\Omega\backslash K_{i,l-1})}^2+\norm{\pi\mathbf{q}^i_j}_{s(\Omega\backslash K_{i,l-1})}^2).
\]
We note that $\norm{\mathbf{q}^i_j}_{s(\Omega\backslash K_{i,l-1})}^2=\norm{(I-\pi)\mathbf{q}^i_j}_{s(\Omega\backslash K_{i,l-1})}^2+\norm{\pi\mathbf{q}^i_j}_{s(\Omega\backslash K_{i,l-1})}^2$. For each coarse element $K_m\subset\Omega\backslash K_{i,l-1}$, we let $\mathbf{r}_m$ be the restriction of $(I-\pi)\mathbf{q}^i_j$ on $K_m$. Then we can write $\mathbf{r}_m=\sum_{n=l_m+1}^{L_m}d^m_n\mathbf{p}^m_n$. we define $\underline{\bm{w}}_m=\sum_{n=l_m+1}^{L_m}\lambda^m_nd^m_n\underline{\bm{\phi}}^m_n$. Note that by the local spectral problem (\ref{local spectral problem}), we have $\nabla\cdot\underline{\bm{\phi}}^m_n=-\tilde{k}\lambda^m_n\mathbf{p}^m_n$. Thus, by the orthogonality of eigenfunctions, we obtain that
\begin{align*}
\norm{(I-\pi)\mathbf{q}^i_j}_{s_m}^2=\sum_{n=l_m+1}^{L_m}d^m_ns_m(\mathbf{p}^m_n,\mathbf{z}_m)=-\sum_{n=l_m+1}^{L_m}d^m_n\frac{1}{\lambda^m_n}(\textbf{div}\underline{\bm{\phi}}^m_n,\mathbf{z}_m)=-\int_{K_m}(\textbf{div}\underline{\bm{w}}_m)\mathbf{z}_mdx,
\end{align*}
where $\mathbf{z}_m$ is the restriction of $\mathbf{q}^i_j$ on $K_m$. Summing the above over all $K_m\subset\Omega\backslash K_{i,l-1}$, we have (let $\underline{\bm{w}}=\sum_{K_m\subset(\Omega\backslash K_{i,l-1})}\underline{\bm{w}}_m\in\underline{\bm{\Sigma}}_h(\Omega\backslash K_{i,l-1})$.)
\begin{align*}
\norm{(I-\pi)\mathbf{q}^i_j}_{s(\Omega\backslash K_{i,l-1})}^2&=\int_{\Omega\backslash K_{i,l-1}}\mathbf{q}^i_j\nabla\cdot\underline{\bm{w}}=-(\mathcal{A}\underline{\bm{\psi}}^i_j,\underline{\bm{w}})\leq \norm{\underline{\bm{\psi}}^i_j}_{\mathcal{A}(\Omega\backslash K_{i,l-1})}\cdot\norm{\underline{\bm{w}}}_{\mathcal{A}(\Omega\backslash K_{i,l-1})}.
\end{align*}
For any $K_m\subset\Omega\backslash K_{i,l-1}$, using (\ref{local spectral problem}), we obtain
$
\int_{K_m}(\mathcal{A}\underline{\bm{w}})\underline{\bm{w}}dx=\sum_{n=l_m+1}^{L_m}(\lambda^m_n)^{-1}(d^m_n)^2\leq \frac{1}{\Lambda}\norm{\mathbf{r}_m}_{s_m}^2.
$
That is, we have
$
\norm{(I-\pi)\mathbf{q}^i_j}_{s_m}\leq \frac{1}{\sqrt{\Lambda}}\norm{\underline{\bm{\psi}}^i_j}_{\mathcal{A}(\Omega\backslash K_{i,l-1})}.
$
This completes the proof of Step 2.

\paragraph{\textbf{Step 3:}} In terms of Step 2, we know for all $l\geq2$
\begin{equation}
\label{l l-1 estimate}
\norm{\bm{\underline{\psi}}^i_j}_{\mathcal{A}(K_{i,l-1}\backslash K_{i,l-2})}^2+\norm{\mathbf{q}^i_j}_{s(K_{i,l-1}\backslash K_{i,l-2})}^2\leq C(1+\frac{1}{\Lambda})(\norm{\underline{\bm{\psi}}^i_j}_{\mathcal{A}(K_{i,l-1}\backslash K_{i,l-2})}^2+\norm{\pi\mathbf{q}^i_j}_{s(K_{i,l-1}\backslash K_{i,l-2})}^2).
\end{equation}
In this step, we will prove that 
\begin{align*}
\norm{\underline{\bm{\psi}}^i_j}_{\mathcal{A}(\Omega\backslash K_{i,l-1})}^2+\norm{\pi\mathbf{q}^i_j}_{s(\Omega\backslash K_{i,l-1})}^2\leq C(1+\frac{1}{\Lambda})^{\frac{1}{2}}(\norm{\underline{\bm{\psi}}^i_j}_{\mathcal{A}(K_{i,l-1}\backslash K_{i,l-2})}^2+\norm{\pi\mathbf{q}^i_j}_{s(K_{i,l-1}\backslash K_{i,l-2})}^2).
\end{align*}
Let $\xi=1-\chi^{l-1,l-2}_i$, then utilizing (\ref{global multiscale basis}) and combining (\ref{commuting diagram})-(\ref{jj}),
\begin{align*}
&\quad\norm{\underline{\bm{\psi}}^i_j}_{\mathcal{A}(\Omega\backslash K_{i,l-1})}^2=(\mathcal{A}\underline{\bm{\psi}}^i_j,\underline{\bm{\psi}}^i_j)_{\Omega\backslash K_{i,l-1}}=(\mathcal{A}\underline{\bm{\psi}}^i_j,\underline{\bm{P}}(\xi\underline{\bm{\psi}}^i_j))-(\mathcal{A}\underline{\bm{\psi}}^i_j,\underline{\bm{P}}(\xi\underline{\bm{\psi}}^i_j))_{K_{i,l-1}\backslash K_{i,l-2}}\\
&=-(\textbf{div}\underline{\bm{P}}(\xi\underline{\bm{\psi}}^i_j),\mathbf{q}^i_j)-(\mathcal{A}\underline{\bm{\psi}}^i_j,\underline{\bm{P}}(\xi\underline{\bm{\psi}}^i_j)-\xi\underline{\bm{\psi}}^i_j)_{K_{i,l-1}\backslash K_{i,l-2}}-(\mathcal{A}\underline{\bm{\psi}}^i_j,\xi\underline{\bm{\psi}}^i_j)_{K_{i,l-1}\backslash K_{i,l-2}}\\
&\leq -(\textbf{div}\underline{\bm{\psi}}^i_j,\mathbf{q}^i_j)_{\Omega\backslash K_{i,l-1}}-(\textbf{div}\underline{\bm{P}}(\xi\underline{\bm{\psi}}^i_j),\mathbf{q}^i_j)_{K_{i,l-1}\backslash K_{i,l-2}}+C\norm{\underline{\bm{\psi}}^i_j}_{\mathcal{A}(K_{i,l-1}\backslash K_{i,l-2})}^2.
\end{align*}
Besides, by using (\ref{global multiscale basis_d}), 
\begin{align*}
&\quad\norm{\pi(\mathbf{q}^i_j)}_{s(\Omega\backslash K_{i,l-1})}^2=s(\pi(\mathbf{q}^i_j),\pi(\xi\mathbf{q}^i_j))-\int_{K_{i,l-1\backslash K_{i,l-2}}}\tilde{k}\pi(\mathbf{q}^i_j)\cdot\pi(\xi\mathbf{q}^i_j)dx\\
&=(\textbf{div}\underline{\bm{\psi}}^i_j,\xi\mathbf{q}^i_j)-s(\mathbf{p}^i_j,\xi\mathbf{q}^i_j)-\int_{K_{i,l-1\backslash K_{i,l-2}}}\tilde{k}\pi(\mathbf{q}^i_j)\cdot\pi(\xi\mathbf{q}^i_j)dx\\
&=(\textbf{div}\underline{\bm{\psi}}^i_j,\mathbf{q}^i_j)_{\Omega\backslash K_{i,l-1}}+(\textbf{div}\underline{\bm{\psi}}^i_j,\xi\mathbf{q}^i_j)_{K_{i,l-1\backslash K_{i,l-2}}}-\int_{K_{i,l-1\backslash K_{i,l-2}}}\tilde{k}\pi(\mathbf{q}^i_j)\cdot\pi(\xi\mathbf{q}^i_j)dx
\end{align*}
since $\mathup{supp}(\mathbf{p}^i_j)\subset K_i$ and $\mathup{supp}(\xi\mathbf{q}^i_j)\subset \Omega\backslash K_{i,l-2}$. Adding above two equations and combining (\ref{l l-1 estimate}), also the inverse inequality, we have
\begin{align*}
&\quad\norm{\underline{\bm{\psi}}^i_j}_{\mathcal{A}(\Omega\backslash K_{i,l-1})}^2+\norm{\pi(\mathbf{q}^i_j)}_{s(\Omega\backslash K_{i,l-1})}^2\\
&\leq -(\textbf{div}\underline{\bm{P}}(\xi\underline{\bm{\psi}}^i_j),\mathbf{q}^i_j)_{K_{i,l-1}\backslash K_{i,l-2}}+(\textbf{div}\underline{\bm{\psi}}^i_j,\xi\mathbf{q}^i_j)_{K_{i,l-1\backslash K_{i,l-2}}}+C\norm{\underline{\bm{\psi}}^i_j}_{\mathcal{A}(K_{i,l-1}\backslash K_{i,l-2})}^2\\
&\qquad-\int_{K_{i,l-1\backslash K_{i,l-2}}}\tilde{k}\pi(\mathbf{q}^i_j)\cdot\pi(\xi\mathbf{q}^i_j)dx\\
&\leq C(\norm{\underline{\bm{\psi}}^i_j}_{\mathcal{A}(K_{i,l-1}\backslash K_{i,l-2})}+\norm{\pi(\mathbf{q}^i_j)}_{s(K_{i,l-1}\backslash K_{i,l-2})})\norm{\mathbf{q}^i_j}_{s(K_{i,l-1}\backslash K_{i,l-2})}+C\norm{\underline{\bm{\psi}}^i_j}_{\mathcal{A}(K_{i,l-1}\backslash K_{i,l-2})}^2\\
&\leq C(1+\frac{1}{\Lambda})^{\frac{1}{2}}(\norm{\underline{\bm{\psi}}^i_j}_{\mathcal{A}(K_{i,l-1}\backslash K_{i,l-2})}^2+\norm{\pi(\mathbf{q}^i_j)}_{s(K_{i,l-1}\backslash K_{i,l-2})}^2).
\end{align*}

\paragraph{\textbf{Step 4:}} In the final step, we will prove the desired estimate. Note that
\begin{align*}
&\quad\norm{\underline{\bm{\psi}}^i_j}_{\mathcal{A}(\Omega\backslash K_{i,l-2})}^2+\norm{\pi(\mathbf{q}^i_j)}_{s(\Omega\backslash K_{i,l-2})}\\
&=\norm{\underline{\bm{\psi}}^i_j}_{\mathcal{A}(\Omega\backslash K_{i,l-1})}^2+\norm{\pi(\mathbf{q}^i_j)}_{s(\Omega\backslash K_{i,l-1})}+\norm{\underline{\bm{\psi}}^i_j}_{\mathcal{A}( K_{i,l-1} \backslash K_{i,l-2})}^2+\norm{\pi(\mathbf{q}^i_j)}_{s( K_{i,l-1}\backslash K_{i,l-2})}\\
&\geq (1+C^{-1}(1+\frac{1}{\Lambda})^{-\frac{1}{2}})(\norm{\underline{\bm{\psi}}^i_j}_{\mathcal{A}(\Omega\backslash K_{i,l-1})}^2+\norm{\pi(\mathbf{q}^i_j)}_{s(\Omega\backslash K_{i,l-1})}).
\end{align*}
Using the above inequality, recursion and the construction of global basis functions (\ref{global multiscale basis}), we obtain
\begin{align*}
\norm{\underline{\bm{\psi}}^i_j}_{\mathcal{A}(\Omega\backslash K_{i,l-1})}^2+\norm{\pi(\mathbf{q}^i_j)}_{s(\Omega\backslash K_{i,l-1})}&\leq (1+C^{-1}(1+\frac{1}{\Lambda})^{-\frac{1}{2}})^{1-l}(\norm{\underline{\bm{\psi}}^i_j}_{\mathcal{A}(\Omega\backslash K_{i,l-1})}^2+\norm{\pi(\mathbf{q}^i_j)}_{s(\Omega\backslash K_{i,l-1})})\\
&\leq C(1+C^{-1}(1+\frac{1}{\Lambda})^{-\frac{1}{2}})^{1-l}\norm{\mathbf{p}^i_j}_s^2.
\end{align*}
Combining Step 1, we obtain the desired estimate.
\end{proof}
\subsection{Stability and convergence of using multiscale basis functions} \label{final stability and convergence}
In this section, we prove the stability and convergence for the multiscale method (\ref{multiscale solution}). We begin with the following inf-sup stability.
\begin{lemma}
\label{Stability of using multiscale basis functions}
For any $\mathbf{q}_0\in U_\mathup{aux}$ with $s(\mathbf{q}_0,\mathbf{x})=0$ (where $\mathbf{x}\in\text{span}\{(1,0)^t,(0,1)^t\}$ in two dimension while in three dimension, $\mathbf{x}\in\text{span}\{(1,0,0)^t,(0,1,0)^t,(0,0,1)^t\}$), there is $\underline{\bm{\sigma}}_0\in \Sigma_\mathup{ms}$ such that
\[
\norm{\mathbf{q}_0}_s\leq C_\mathup{ms}\frac{(\textbf{\textup{div}}\underline{\bm{\sigma}}_0,\mathbf{q}_0)}{\norm{\underline{\bm{\sigma}}_0}_\Sigma},
\]
where $C_\mathup{ms}>0$ is a constant independent of mesh sizes $h,H$ and the Lam$\rm \acute{e}$ constant $\lambda$.
\end{lemma}
\begin{proof}
By using the proof of Lemma \ref{inf-sup for global variational form}, there is $\underline{\bm{w}}\in\Sigma_\mathup{glo}$ and $\mathbf{p}\in\mathup{span}\{\mathbf{q}^i_j\}$ such that 
\begin{subequations}
\label{exist w}
\begin{align}
(\mathcal{A}\underline{\bm{w}},\underline{\bm{\tau}}_h)+(\textbf{div}\underline{\bm{\tau}}_h,\mathbf{p})&=0\quad \forall \underline{\bm{\tau}}_h\in\underline{\bm{\Sigma}}_h,\\
-(\textbf{div}\underline{\bm{w}},\mathbf{v}_h)&=s(-\mathbf{q}_0,\mathbf{v}_h)\quad \forall \mathbf{v}_h\in\bm{\mathit{U}}_h,
\end{align}
\end{subequations}
and $s(\mathbf{p},\mathbf{x})=0$ with $s(\mathbf{q}_0,\mathbf{x})=0$ (where $\mathbf{x}\in\text{span}\{(1,0)^t,(0,1)^t\}$ in two dimension while in three dimension, $\mathbf{x}\in\text{span}\{(1,0,0)^t,(0,1,0)^t,(0,0,1)^t\}$), $\norm{\mathbf{q}_0}_s=\norm{-\mathbf{q}_0}_s\leq C_\mathup{g}\frac{(\textbf{div}\underline{\bm{w}},\mathbf{q}_0)}{\norm{\underline{\bm{w}}}_\Sigma}$. Then we have
\begin{equation}
\label{inf sup w}
\norm{\underline{\bm{w}}}_\Sigma\leq C_\mathup{g}\frac{(\textbf{div}\underline{\bm{w}},\mathbf{q}_0)}{\norm{\mathbf{q}_0}_s}=C_\mathup{g}\frac{s(\mathbf{q}_0,\mathbf{q}_0)}{\norm{\mathbf{q}_0}_s}=C_\mathup{g}\norm{\mathbf{q}_0}_s.
\end{equation}
Notice that $\mathbf{p}\in\mathup{span}\{\mathbf{q}^i_j\}$, we can write
$
\mathbf{p}=\sum_{i=1}^N\sum_{j=1}^{l_i}d^i_j\mathbf{q}^i_j.
$
Utilizing the construction of global multiscale basis functions (\ref{global multiscale basis}), we can see that
$
\underline{\bm{w}}=\sum_{i=1}^N\sum_{j=1}^{l_i}d^i_j\underline{\bm{\psi}}^i_j.
$
The above motivates the definition of $\underline{\bm{\sigma}}_0\in\Sigma_\mathup{ms}$, which is given by
$
\underline{\bm{\sigma}}_0=\sum_{i=1}^N\sum_{j=1}^{l_i}d^i_j\underline{\bm{\psi}}^i_{j,\mathup{ms}}.
$
We call $\underline{\bm{\sigma}}_0$ the projection of $\underline{\bm{w}}\in\Sigma_\mathup{ms}$ into the space $\Sigma_\mathup{ms}$. Next, we note that 
\begin{equation}
\label{u q_0}
(\textbf{div}\underline{\bm{\sigma}}_0,\mathbf{q}_0)=(\textbf{div}\underline{\bm{w}},\mathbf{q}_0)+(\textbf{div}(\underline{\bm{\sigma}}_0-\underline{\bm{w}}),\mathbf{q}_0)=\norm{\mathbf{q}_0}_s^2+(\textbf{div}(\underline{\bm{\sigma}}_0-\underline{\bm{w}}),\mathbf{q}_0).
\end{equation}
In terms of the Cauchy-Schwarz inequality, we have
$
(\textbf{div}(\underline{\bm{\sigma}}_0-\underline{\bm{w}}),\mathbf{q}_0)\leq\norm{\tilde{k}^{-\frac{1}{2}}\nabla\cdot(\underline{\bm{\sigma}}_0-\underline{\bm{w}})}_{L^2}\cdot\norm{\mathbf{q}_0}_s.
$
By the definition of $\underline{\bm{u}}$ and $\underline{\bm{w}}$, we have
$
\norm{\tilde{k}^{-\frac{1}{2}}\nabla\cdot(\underline{\bm{\sigma}}_0-\underline{\bm{w}})}_{L^2}^2\leq C_\mathup{ol}(l+1)^n\sum_{i=1}^N\sum_{j=1}^{l_i}(d^i_j)^2\norm{\tilde{k}^{-\frac{1}{2}}\nabla\cdot(\underline{\bm{\psi}}^i_j-\underline{\bm{\psi}}^i_{j,\mathup{ms}})}_{L^2}^2,
$
where $C_\mathup{ol}$ is the maximum number of overlapping regions. By Lemma \ref{decay property of basis function_2}, we know
\begin{equation}
\label{u-w decay}
\norm{\underline{\bm{\sigma}}_0-\underline{\bm{w}}}_\Sigma^2\leq C_\mathup{ol}(l+1)^n\sum_{i=1}^N\sum_{j=1}^{l_i}(d^i_j)^2 C_\#.
\end{equation}
Using (\ref{global multiscale basis}), the definition of $\mathbf{p}$ and $\underline{\bm{w}}$, we see that   $\mathbf{p}$ and $\underline{\bm{w}}$ satisfy:
\begin{subequations}
	\label{w p variational form}
	\begin{align}		
		(\mathcal{A}\underline{\bm{w}},\underline{\bm{\tau}}_h)+(\textbf{div}\underline{\bm{\tau}}_h,\mathbf{p})&=0  \quad \forall \underline{\bm{\tau}}_h\in \underline{\bm{\mathit{\Sigma}}}_h, \label{w p variational form_a}\\
s(\pi\mathbf{p},\pi\mathbf{v}_h)-(\textbf{div}\underline{\bm{w}},\mathbf{v}_h)&= s(\tilde{\mathbf{p}},\mathbf{v}_h)\quad \forall \mathbf{v}_h\in \bm{\mathit{U}}_h,\label{w p variational form_d}
	\end{align}	
\end{subequations}
where $\tilde{\mathbf{p}}=\sum_{i=1}^N\sum_{j=1}^{l_i}d^i_j\mathbf{p}^i_j$.
Combining (\ref{w p variational form}) and (\ref{exist w}), we have
$
s(\mathbf{q}_0,\mathbf{v}_h)=s(\tilde{\mathbf{p}}-\pi\mathbf{p},\mathbf{v}_h)$  for all $\mathbf{v}_h\in \bm{\mathit{U}}_h.
$
Then we have $\mathbf{q}_0=\tilde{\mathbf{p}}-\pi\mathbf{p}$. So $\norm{\tilde{\mathbf{p}}}_s\leq\norm{\mathbf{q}_0}_s+\norm{\pi\mathbf{p}}_s.$
Note that by Lemma \ref{inf-sup for global variational form}, there is $\underline{\mathbf{x}}\in\Sigma_\mathup{glo}$ such that
\begin{align*}
\norm{\pi\mathbf{p}}_s\leq C_\mathup{g}\frac{(\textbf{div}\underline{\mathbf{x}},\pi\mathbf{p})}{\norm{\underline{\mathbf{x}}}_\Sigma}&=C_\mathup{g}\frac{(\textbf{div}\underline{\mathbf{x}},\pi\mathbf{p}-\mathbf{p})+(\textbf{div}\underline{\mathbf{x}},\mathbf{p})}{\norm{\underline{\mathbf{x}}}_\Sigma}=C_\mathup{g}\frac{-(\mathcal{A}\underline{\mathbf{w}},\underline{\mathbf{x}})}{\norm{\underline{\mathbf{x}}}_\Sigma}\leq C_\mathup{g}\norm{\underline{\bm{w}}}_\Sigma\leq C_\mathup{g}^2\norm{\mathbf{q}_0}_s,
\end{align*}
where (\ref{inf sup w}) has been used in the last line.
Combining above two estimates, we obtain
\begin{equation}
\label{d p s}
\sum_{i=1}^N\sum_{j=1}^{l_i}(d^i_j)^2=\norm{\tilde{\mathbf{p}}}_s^2\leq(1+C_\mathup{g}^2)^2\norm{\mathbf{q}_0}_s^2.
\end{equation}
Then by (\ref{d p s}) and (\ref{u-w decay}), we have
$
\norm{\underline{\bm{\sigma}}_0-\underline{\bm{w}}}_\Sigma^2\leq C_\mathup{ol}(l+1)^n C_\#(1+C_\mathup{g}^2)^2\norm{\mathbf{q}_0}_s^2.
$
Combining above inequality and (\ref{u q_0}), we obtain that
$
(\textbf{div}\underline{\bm{\sigma}}_0,\mathbf{q}_0)\geq \big(1-C_\mathup{ol}^{\frac{1}{2}}(l+1)^{\frac{n}{2}}C_\#^{\frac{1}{2}}(1+C_\mathup{g}^2)\big)\norm{\mathbf{q}_0}_s^2.
$
Besides, by (\ref{u-w decay}), we have
\[
\norm{\underline{\bm{\sigma}}_0}_\Sigma^2\leq2\norm{\underline{\bm{\sigma}}_0-\underline{\bm{w}}}_\Sigma^2+2\norm{\underline{\bm{w}}}_\Sigma^2\leq2C_\mathup{ol}(l+1)^n C_\#(1+C_\mathup{g}^2)^2\norm{\mathbf{q}_0}_s^2+2C_\mathbf{g}^2\norm{\mathbf{q}_0}_s^2.
\]
Then by above estimates, we obtain
\[
\frac{(\textbf{div}\underline{\bm{\sigma}}_0,\mathbf{q}_0)}{\norm{\underline{\bm{\sigma}}_0}_\Sigma}\geq\frac{\big(1-C_\mathup{ol}^{\frac{1}{2}}(l+1)^{\frac{n}{2}}C_\#^{\frac{1}{2}}(1+C_\mathup{g}^2)\big)\norm{\mathbf{q}_0}_s}{\sqrt{2(C_\mathup{ol}(l+1)^n C_\#(1+C_\mathup{g}^2)^2+C_\mathbf{g}^2)}}.
\]
That is, we have
$
\norm{\mathbf{q}_0}_s\leq C_\mathup{ms}\frac{(\textbf{div}\underline{\bm{\sigma}}_0,\mathbf{q}_0)}{\norm{\underline{\bm{\sigma}}_0}_\Sigma},
$
where $C_\mathup{ms}=\sqrt{2(C_\mathup{ol}(l+1)^n C_\#(1+C_\mathup{g}^2)^2+C_\mathbf{g}^2)}/(1-C_\mathup{ol}^{\frac{1}{2}}(l+1)^{\frac{n}{2}}C_\#^{\frac{1}{2}}(1+C_\mathup{g}^2))$.
\end{proof}
\begin{theorem}
\label{final convergence}
Let $(\underline{\bm{\sigma}}_\mathup{ms},\mathbf{u}_\mathup{ms})\in\Sigma_\mathup{ms}\times U_\mathup{aux}$ the multiscale solution of (\ref{multiscale solution}) and let $(\underline{\bm{\sigma}}_h,\mathbf{u}_h)\in \underline{\bm{\mathit{\Sigma}}}_h\times\bm{\mathit{U}}_h$ be the fine solution of (\ref{discrete formula}). Then we have
\begin{align*}
\norm{\underline{\bm{\sigma}}_\mathup{ms}-\underline{\bm{\sigma}}_h}_\mathcal{A}^2+\norm{\mathbf{u}_\mathup{ms}-\pi(\mathbf{u}_h)}_{L^2}^2&\leq (1+(\max\{\tilde{k}^{-\frac{1}{2}}\})^2)(1+C_\text{ms})^4C_\text{ol}(l+1)^nC_\#(1+C_g^2)^2\norm{\tilde{k}^{-1}\mathbf{f}}_s^2\\
&\quad+2(\max\{\tilde{k}^{-\frac{1}{2}}\})^2\frac{1}{\Lambda}(C_Sh^s\norm{\mathbf{f}}_{L^2}+C\frac{1}{\sqrt{\Lambda}}\norm{(I-\pi)(\tilde{k}^{-1}\mathbf{f})}_s+C\norm{\mathbf{f}}_{L^2})^2,
\end{align*}
where $C,C_S>0$ are independent of mesh sizes $h,H$ and the Lam$\rm \acute{e}$ coefficient $\lambda$, but $C_S$ depends on the regularity of the PDE problem (\ref{our pde}) (see Assumption \ref{0128b}).
\end{theorem}
\begin{proof}
By (\ref{multiscale solution}) and (\ref{discrete formula}), we have
\begin{subequations}
	\label{multiscale solution-discrete formula}
	\begin{align}		
		(\mathcal{A}(\underline{\bm{\sigma}}_h-\underline{\bm{\sigma}}_\mathup{ms}),\underline{\bm{\tau}}_h)+(\textbf{div}\underline{\bm{\tau}}_h,\mathbf{u}_h-\mathbf{u}_\mathup{ms})&=0  \quad \forall \underline{\bm{\tau}}_h\in \Sigma_\mathup{ms}\subset\underline{\bm{\mathit{\Sigma}}}_h, \label{multiscale solution-discrete formula_a}\\
(\textbf{div}(\underline{\bm{\sigma}}_h-\underline{\bm{\sigma}}_\mathup{ms}),\mathbf{v}_h)&= 0\quad \forall \mathbf{v}_h\in U_\mathup{aux}\subset\bm{\mathit{U}}_h.\label{multiscale solution-discrete formula_d}
	\end{align}	
\end{subequations}
By the construction of multiscale basis functions (\ref{multiscale basis}), we have $\nabla\cdot\Sigma_\mathup{ms}\subset\tilde{k} U_\mathup{aux}$. Therefore, we have
\begin{equation}
\label{div Sigma in s}
(\textbf{div}\underline{\bm{v}},\pi(\mathbf{u}_h)-\mathbf{u}_h)=0\quad\forall \underline{\bm{v}}\in\Sigma_\mathup{ms}.
\end{equation}
Then, for any $\underline{\bm{v}}\in\Sigma_\mathup{ms}$, we have
\begin{align*}
(\mathcal{A}(\underline{\bm{\sigma}}_h-\underline{\bm{\sigma}}_\mathup{ms}),\underline{\bm{\sigma}}_h-\underline{\bm{\sigma}}_\mathup{ms})&=(\mathcal{A}(\underline{\bm{\sigma}}_h-\underline{\bm{\sigma}}_\mathup{ms}),\underline{\bm{\sigma}}_h-\underline{\bm{v}})+(\mathcal{A}(\underline{\bm{\sigma}}_h-\underline{\bm{\sigma}}_\mathup{ms}),\underline{\bm{v}}-\underline{\bm{\sigma}}_\mathup{ms})\\
&=(\mathcal{A}(\underline{\bm{\sigma}}_h-\underline{\bm{\sigma}}_\mathup{ms}),\underline{\bm{\sigma}}_h-\underline{\bm{v}})-(\textbf{div}(\underline{\bm{v}}-\underline{\bm{\sigma}}_\mathup{ms}),\pi(\mathbf{u}_h)-\mathbf{u}_\mathup{ms})\\
&=(\mathcal{A}(\underline{\bm{\sigma}}_h-\underline{\bm{\sigma}}_\mathup{ms}),\underline{\bm{\sigma}}_h-\underline{\bm{v}})-(\textbf{div}(\underline{\bm{v}}-\underline{\bm{\sigma}}_h),\pi(\mathbf{u}_h)-\mathbf{u}_\mathup{ms}),
\end{align*}
where (\ref{multiscale solution-discrete formula_d}) has been used in the last line. Next we obtain that
\begin{align*}
\norm{\underline{\bm{\sigma}}_h-\underline{\bm{\sigma}}_\mathup{ms}}_\mathcal{A}^2&\leq\norm{\underline{\bm{\sigma}}_h-\underline{\bm{\sigma}}_\mathup{ms}}_\mathcal{A}\cdot\norm{\underline{\bm{\sigma}}_h-\underline{\bm{v}}}_\mathcal{A}+\norm{\tilde{k}^{-\frac{1}{2}}\textbf{div}(\underline{\bm{v}}-\underline{\bm{\sigma}}_h)}_{L^2}\cdot\norm{\pi(\mathbf{u}_h)-\mathbf{u}_\mathup{ms}}_s\\
&\leq \norm{\underline{\bm{\sigma}}_h-\underline{\bm{v}}}_\Sigma\cdot(\norm{\underline{\bm{\sigma}}_h-\underline{\bm{\sigma}}_\mathup{ms}}_\mathcal{A}+\norm{\pi(\mathbf{u}_h)-\mathbf{u}_\mathup{ms}}_s).
\end{align*}
In terms of Lemma \ref{Stability of using multiscale basis functions}, there is $\underline{\bm{w}}\in\Sigma_\mathup{ms}$ such that
\begin{equation}
\label{estimate for u_h and u_ms}
\begin{aligned}
\norm{\pi(\mathbf{u}_h)-\mathbf{u}_\mathup{ms}}_s&\leq C_\mathup{ms}\frac{(\textbf{div}\underline{\bm{w}},\pi(\mathbf{u}_h)-\mathbf{u}_\mathup{ms})}{\norm{\underline{\bm{w}}}_\Sigma}=C_\mathup{ms}\frac{(\textbf{div}\underline{\bm{w}},\mathbf{u}_h-\mathbf{u}_\mathup{ms})}{\norm{\underline{\bm{w}}}_\Sigma}=C_\mathup{ms}\frac{-(\mathcal{A}(\underline{\bm{\sigma}}_h-\underline{\bm{\sigma}}_\mathup{ms}),\underline{\bm{w}}}{\norm{\underline{\bm{w}}}_\Sigma}\\
&\leq C_\mathup{ms}\norm{\underline{\bm{\sigma}}_h-\underline{\bm{\sigma}}_\mathup{ms}}_\mathcal{A}.
\end{aligned}
\end{equation}
By above two inequalities, we have
$
\norm{\underline{\bm{\sigma}}_h-\underline{\bm{\sigma}}_\mathup{ms}}_\mathcal{A}\leq(1+C_\mathup{ms})\norm{\underline{\bm{\sigma}}_h-\underline{\bm{v}}}_\Sigma.
$
Since $\underline{\bm{\sigma}}_h=\underline{\bm{\sigma}}_\mathup{glo}$ (see Lemma \ref{glo-fine solution}), we have the following estimate
\[
\norm{\underline{\bm{\sigma}}_h-\underline{\bm{\sigma}}_\mathup{ms}}_\mathcal{A}\leq(1+C_\mathup{ms})\norm{\underline{\bm{\sigma}}_\mathup{glo}-\underline{\bm{v}}}_\Sigma,\quad \forall \underline{\bm{v}}\in\Sigma_\mathup{ms}.
\]
Suppose $\underline{\bm{\sigma}}_\mathup{glo}=\sum_{i=1}^N\sum_{j=1}^{l_i}c^i_j\underline{\bm{\psi}}^i_j$, and let $\underline{\bm{v}}=\sum_{i=1}^N\sum_{j=1}^{l_i}c^i_j\underline{\bm{\psi}}^i_{j,\mathup{ms}}$, then combining the proof of Lemma \ref{Stability of using multiscale basis functions}, problem (\ref{discrete formula}) and the fact that $\underline{\bm{\sigma}}_h=\underline{\bm{\sigma}}_\mathup{glo}$, we get
\[
\norm{\underline{\bm{\sigma}}_\mathup{glo}-\underline{\bm{v}}}_\Sigma=\norm{\sum_{i=1}^N\sum_{j=1}^{l_i}c^i_j(\underline{\bm{\psi}}^i_j-\underline{\bm{\psi}}^i_{j,\mathup{ms}})}_\Sigma\leq C_\mathup{ol}^{\frac{1}{2}}(l+1)^{\frac{n}{2}}C_\#^{\frac{1}{2}}(1+C_\mathup{g}^2)\norm{\tilde{k}^{-1}\mathbf{f}}_s.
\]
Therefore, we obtain that
\begin{equation}
\label{estimate for sigma_h and sigma_ms}
\norm{\underline{\bm{\sigma}}_h-\underline{\bm{\sigma}}_\mathup{ms}}_\mathcal{A}\leq(1+C_\mathup{ms})C_\mathup{ol}^{\frac{1}{2}}(l+1)^{\frac{n}{2}}C_\#^{\frac{1}{2}}(1+C_\mathup{g}^2)\norm{\tilde{k}^{-1}\mathbf{f}}_s.
\end{equation}
Besides, in terms of Lemma \ref{glo-fine solution}, we know $\pi(\mathbf{u}_h)=\mathbf{u}_\mathup{glo}$, and
\begin{equation}
\label{uh - ums}
\begin{aligned}
&\quad\norm{\mathbf{u}_h-\mathbf{u}_\mathup{ms}}_{L^2}=\norm{\tilde{k}^{-\frac{1}{2}}\mathbf{u}_h-\mathbf{u}_\mathup{ms}}_s\leq\max\{\tilde{k}^{-\frac{1}{2}}\}\norm{\mathbf{u}_h-\mathbf{u}_\mathup{ms}}_s\\
&\leq\max\{\tilde{k}^{-\frac{1}{2}}\}(\norm{\mathbf{u}_h-\pi\mathbf{u}_h}_s+\norm{\pi\mathbf{u}_h-\mathbf{u}_\mathup{ms}}_s)=\max\{\tilde{k}^{-\frac{1}{2}}\}(\norm{\mathbf{u}_h-\mathbf{u}_\mathup{glo}}_s+\norm{\pi\mathbf{u}_h-\mathbf{u}_\mathup{ms}}_s)\\
&\leq \max\{\tilde{k}^{-\frac{1}{2}}\}\frac{1}{\sqrt{\Lambda}}(C_Sh^s\norm{\mathbf{f}}_{L^2}+C\frac{1}{\sqrt{\Lambda}}\norm{(I-\pi)(\tilde{k}^{-1}\mathbf{f})}_s+C\norm{\mathbf{f}}_{L^2})\\
&\quad+\max\{\tilde{k}^{-\frac{1}{2}}\}C_\mathup{ms}(1+C_\mathup{ms})C_\mathup{ol}^{\frac{1}{2}}(l+1)^{\frac{n}{2}}C_\#^{\frac{1}{2}}(1+C_\mathup{g}^2)\norm{\tilde{k}^{-1}\mathbf{f}}_s,
\end{aligned}
\end{equation}
where (\ref{estimate for u_h and u_ms}) and (\ref{estimate for sigma_h and sigma_ms}) are used in the last line. $C,C_S>0$ are independent of mesh sizes $h,H$ and the Lam$\rm \acute{e}$ coefficient $\lambda$, but $C_S$ depends on the regularity of the PDE problem (\ref{our pde}).
Finally, combining (\ref{estimate for sigma_h and sigma_ms}) and (\ref{uh - ums}), we obtain the desired estimate. This completes the convergence proof.
\end{proof}
\begin{remark}
\label{specific convergence order}
It's clear that $\max\{\tilde{k}^{-\frac{1}{2}}\}=O(H)$. To obtain $O(H)$ convergence, we need to choose the size of the oversampling domain such that
\[
(1+(\max\{\tilde{k}^{-\frac{1}{2}}\})^2)(1+C_\text{ms})^4C_\text{ol}(l+1)^nC_\#(1+C_g^2)^2=O(1).
\]
Therefore, we see that the size of the oversampling domain $l=O(\text{log}(\text{max}\{k\}/H^2))$.
\end{remark}

\section{Numerical experiments}
\label{Numerical experiments}
In this section, we perform some numerical simulations to show the performance of the proposed multiscale method. We performed the computations in MATLAB 2021a on a Lenovo ThinkCentre M80q Gen 4 desktop with an Intel Core i9-13900T processor and 32GB RAM. In our experiments, we take the computational domain $\Omega$ as $\Omega=[0,1]^2$. We give two test models as shown in Figure \ref{Heterogeneity pattern of coefficient}, where the Young's modulus $E$ is taken as a piecewise constant with heterogeneity pattern. For clearer, we take the yellow part as $E_1$ and the blue part as $E_2$. The Lam$\rm \acute{e}$ coefficients $\lambda,\mu$ are denoted in (\ref{Lame constants}). In Sections~\ref{test model 1}-\ref{test model 2}, we present stress field errors and corresponding plots for the Neumann boundary condition (Eq.~\eqref{our pde_d}). In Section~\ref{incompressible}, we consider nearly incompressible materials (under Neumann boundary condition). The performance of our multiscale method for both Dirichlet and mixed boundary conditions is analyzed in Section~\ref{different bdcs}. Computational costs are reported in Section~\ref{times}. For test cases in Sections~\ref{test model 1}--\ref{incompressible}, we use the source term $\mathbf{f}=(\cos\pi x\sin\pi y,0)^t$ (satisfying $\int_\Omega\mathbf{f}=\mathbf{0}$), while setting $\mathbf{f}=(1,1)^t$ in Section~\ref{different bdcs}.
\begin{figure}[tbph] 
		\centering 
		\includegraphics[height=6cm,width=14cm]{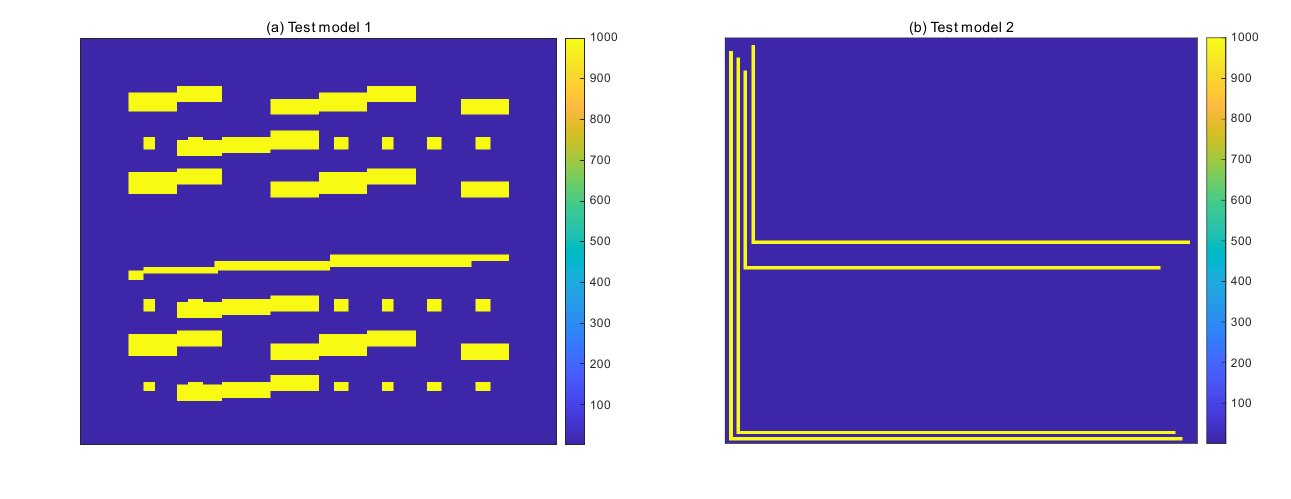} 
		\caption{Young’s Modulus of the test models}
		\label{Heterogeneity pattern of coefficient}
\end{figure}

We divide the domain $\Omega$ into $128\times128$ fine grid blocks. The coarse mesh size $H$ will be chosen from $1/8$, $1/16$ and $1/32$. The degrees of freedom (DOFs) used to obtain the reference solution is 2065920 so $128\times128$ fine grid blocks are enough to obtain the reference solution (In most existing studies, fine-scale discretizations for reference solutions involve only $\sim 10^5$ DOFs, even with smaller $h$ (e.g., 160,801 DOFs at $h = 1/400$ \cite{YeEric2023, YeJinEric2023, EricSM2020}; 115,200 DOFs at $h = 1/225$ \cite{EricFu2023})). In order to evaluate the accuracy of the multiscale solution, we use the following relative errors: 
 \begin{equation*}
e_{\sigma}=\frac{\norm{\underline{\bm{\sigma}}_h-\underline{\bm{\sigma}}_{\text{ms}}}_\mathcal{A}}{\norm{\underline{\bm{\sigma}}_h}_{\mathcal{A}}},\quad e_u = \frac{\norm{\mathbf{u}_h-\mathbf{u}_{\text{ms}}}_{L^2}}{\norm{\mathbf{u}_h}_{L^2}}.
 \end{equation*}
Note that under homogeneous Neumann boundary conditions, the multiscale solution differs from the reference solution by a rigid body motion. Thus, we only report stress field errors for the pure Neumann case, while Section~\ref{different bdcs} (with Dirichlet/mixed boundary conditions) presents both stress and displacement errors. For clarity, simplified notations are provided in Table~\ref{simplified notations}.
\begin{table}
	\footnotesize
	\centering
	\caption{Simplified notations}
	\label{simplified notations}
	\begin{tabular}{cc}
		\hline
		Parameters &Simplified symbols  \\
		\hline
		Contrast value (yellow in Figure \ref{Heterogeneity pattern of coefficient}) & $E_1$\\
        Contrast value (blue in Figure \ref{Heterogeneity pattern of coefficient}) &  $E_2$\\
        Contrast ratio & $E_1/E_2$\\
        Poisson's ratio (yellow in Figure \ref{Heterogeneity pattern of coefficient}) & $\nu_1$\\
        Poisson's ratio (blue in Figure \ref{Heterogeneity pattern of coefficient}) &  $\nu_2$\\
		Number of multiscale bases in each coarse element & $Nbf$\\
        Number of oversampling layers for each coarse element & $osly$\\
		\hline
	\end{tabular}
\end{table}

\subsection{Test model 1}
\label{test model 1}
In the first test model, we use Test model 1 of Figure \ref{Heterogeneity pattern of coefficient}. We take Poisson's ratio $\nu_1=\nu_2=0.35$. Let $E_2=1$ and $E_1=10^3,10^4,10^5,10^6$ respectively. 

To study the effects of different number of basis functions on the relative errors for the stress, we report
the errors with respect to the number of basis functions in Table \ref{errors changing by nbf}. Note that in Table \ref{errors changing by nbf}, we take
$osly=4$, $H=1/32$, $E_1/E_2=10^4$. It is clear that more number of basis functions yields more accurate CEM-GMsFEM
solutions.

Numerical errors for the stress in Test model 1 with different $E_1/E_2$, $osly$ and $H$ are shown in Tables \ref{errors_1}, \ref{errors_2} and Figure \ref{Relative errors stress of test model 1} (Let $Nbf=6$). 
We observe that as $osly$ reaches to $4$, the relative errors for stress with $H=1/32$ decrease to $0.01$ while the relative errors for stress with $H=1/16$ decays to $0.001$. We also find the independent relationship between errors and contrast ratios. In particular, the errors almost remain the same when $osly$ reaches to $4$ as the contrast ratios $E_1/E_2$ increase from $10^3$ to $10^6$.

We demonstrate components of the reference solution for the stress (i.e. $(\underline{\bm{\sigma}}_h)_{11},(\underline{\bm{\sigma}}_h)_{12},(\underline{\bm{\sigma}}_h)_{22}$) in (a)-(c) of Figure \ref{128_stress_test_model_1} and components of the multiscale solution for the stress with $H=1/16$, $osly=4$, $Nbf=6$ (i.e. $(\underline{\bm{\sigma}}_{\text{ms}})_{11},(\underline{\bm{\sigma}}_{\text{ms}})_{12},(\underline{\bm{\sigma}}_{\text{ms}})_{22}$) in (d)-(f) of Figure \ref{128_stress_test_model_1}. Observing (a)-(c) and (d)-(f) of Figure \ref{128_stress_test_model_1}, it is almost the same   
between the multiscale solution and the reference solution.

Furthermore, we investigate the convergence behavior of the mixed CEM-GMsFEM solution with respect to the coarse mesh size $H$. The number of oversampling layers are set according to Theorem \ref{final convergence} and $Nbf=6$. The numerical results, presented in Table \ref{osly and H}, demonstrate that as $H$ becomes smaller, more oversampling regions are needed to ensure the convergence rate.

\begin{table}
	\footnotesize
	\centering
	\caption{Relative errors for the stress in Test model 1 with $H=1/32,osly=4,E_1/E_2=10^4$}
	\label{errors changing by nbf}
	\begin{tabular}{ccccccc}
		\hline
		$Nbf$ &1 &2  &3  &4  &5 &6  \\
		\hline
		$e_{\sigma}$ &0.744810001 &0.043834367	&0.026334123 &0.022771214 &0.022595145 &0.022514092\\
		\hline
	\end{tabular}
\end{table}

\begin{table}
	\footnotesize
	\centering
	\caption{Relative errors for the stress in Test model 1 with different contrast ratios and $osly$ under Neumann boundary condition ($H=1/32,Nbf=6$).}
	\label{errors_1}
	\begin{tabular}{ccccc}
		\hline
		$osly\backslash E_1/E_2$ & $10^3$ & $10^4$ & $10^5$ &$10^6$ \\
		\hline
		 1 &0.619165237	&0.438103348	&0.572718412	&0.483989972\\
2 &0.187297244	&0.112316887	&0.147330352	&0.118915816\\
3 &0.064275411	&0.027982006	&0.034052014	&0.025944477\\
4 &0.030563604	&0.022514092	&0.024104299	&0.019580411\\
5 &0.013770593	&0.010902861	&0.015349838	&0.010475281
		\\
		\hline
	\end{tabular}
	%\label{tbl:table1}
\end{table}
\begin{table}
	\footnotesize
	\centering
	\caption{Relative errors for the stress in Test model 1 with different contrast ratios and $osly$ under Neumann boundary condition ($H=1/16,Nbf=6$).}
	\label{errors_2}
	\begin{tabular}{ccccc}
		\hline
		$osly\backslash E_1/E_2$ & $10^3$ &$10^4$ &$10^5$ &$10^6$\\
		\hline
		1 &0.201881723	&0.180639832	&0.257984778	&0.233849945\\
2 &0.138836717	&0.077568322	&0.198167688	&0.132833547\\
3 &0.038690094	&0.033262756	&0.036356063	&0.039588161\\
4 &0.004367468	&0.004370127	&0.004370397	&0.004370424\\
		\hline	
	\end{tabular}
	%\label{tbl:table1}
\end{table}
\begin{figure}[tbph] 
		\centering 
		\includegraphics[height=5.5cm,width=10cm]{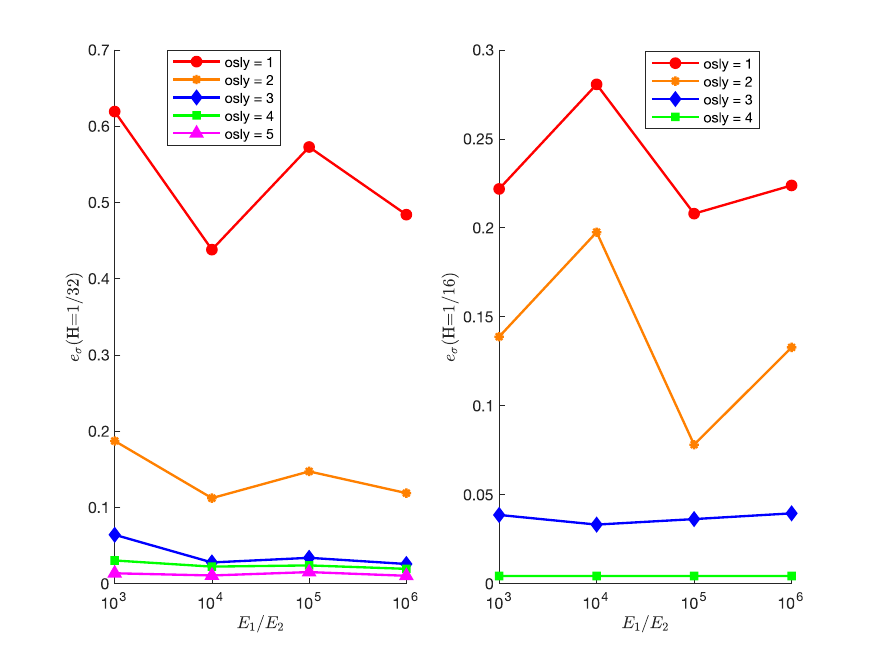} 
		\caption{Relative errors for the stress in Test model 1 with different $osly$ and $E_1/E_2$ under Neumann boundary condition ($Nbf=6$, $H=1/32$(left), $H=1/16$(right))}
		\label{Relative errors stress of test model 1}
\end{figure}

\begin{figure}[tbph] 
		\centering 
		\includegraphics[height=8cm,width=15cm]{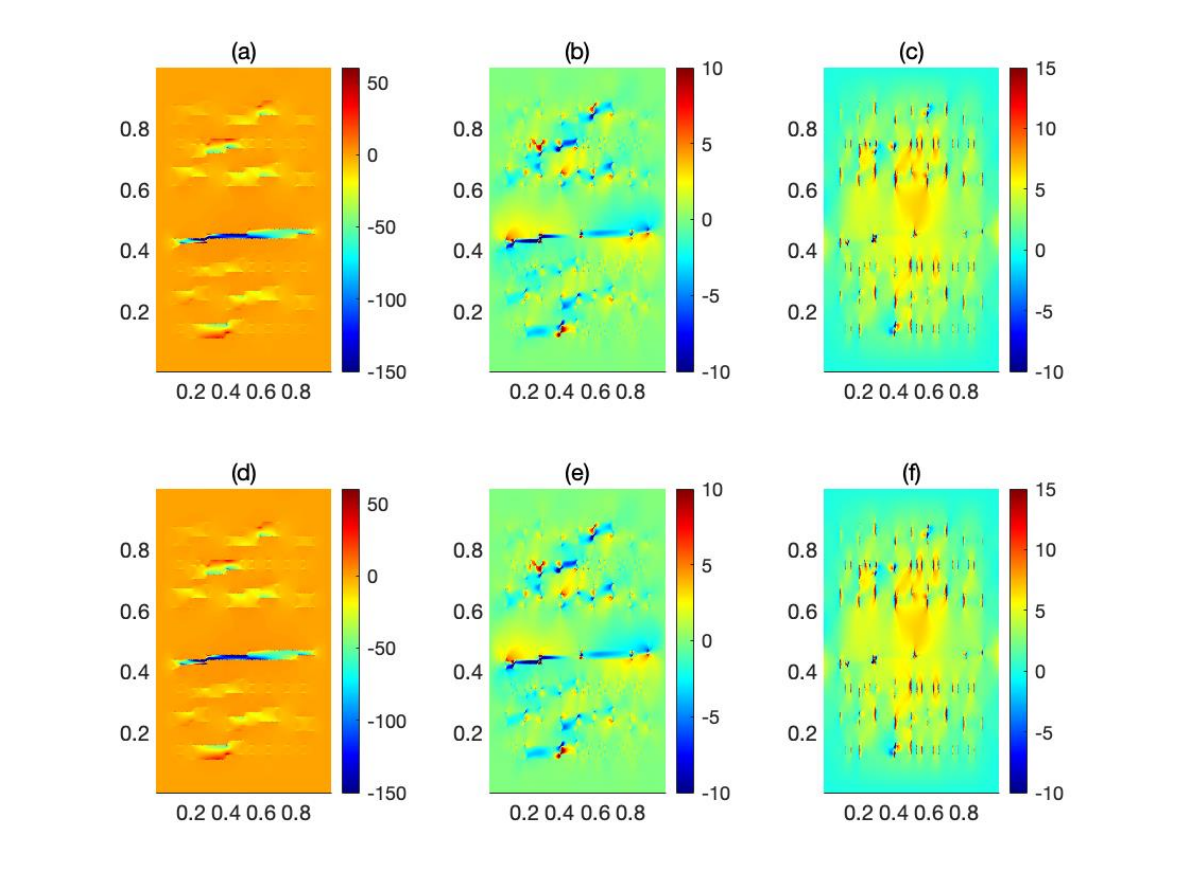} 
		\caption{(a)-(c): components of reference solution for the stress in Test model 1 (i.e. $(\underline{\bm{\sigma}}_h)_{11},(\underline{\bm{\sigma}}_h)_{12},(\underline{\bm{\sigma}}_h)_{22}$); (d)-(f): components of multiscale solution for the stress in Test model 1 (i.e. $(\underline{\bm{\sigma}}_{\text{ms}})_{11},(\underline{\bm{\sigma}}_{\text{ms}})_{12},(\underline{\bm{\sigma}}_{\text{ms}})_{22}$) with $H=1/16$, $osly=4$, $Nbf=6$ (under Neumann boundary condition)} 
		\label{128_stress_test_model_1}
\end{figure}

\begin{table}
	\footnotesize
	\centering
	\caption{Relative errors for the stress in Test model 1 with different $osly$ and $H$ under Neumann boundary condition ($E_1/E_2=10^4$, $Nbf=6$).}
	\label{osly and H}
	\begin{tabular}{ccccc}
		\hline
		$osly$ & coarse mesh size ($H$)  & Relative errors for the stress ($e_\sigma$)\\
		\hline
        2 & $1/8$ & 0.236982729\\
		3 & $1/16$ &0.033262756 \\
4 &$1/32$ &0.022514092 \\
		\hline
	\end{tabular}
	%\label{tbl:table1}
\end{table}
\subsection{Test model 2}
\label{test model 2}
In the second numerical experiment, we consider Test Model~2 illustrated in Figure~\ref{Heterogeneity pattern of coefficient}. Poisson's ratios are set equal with $\nu_1 = \nu_2 = 0.35$, while Young's moduli follow a contrast ratio with $E_2 = 1$ fixed and $E_1$ varying across multiple orders of magnitude ($10^3$, $10^4$, $10^5$, $10^6$).

The relative errors in stress for Test Model~2 are presented in Tables~\ref{Relative errors for the stress of Test model 2 with H=1/32}--\ref{Relative errors for the stress of Test model 2 with H=1/16} and Figure~\ref{Relative errors stress of test model 2}, examining various $E_1/E_2$, oversampling layer numbers ($osly$), and coarse mesh sizes ($H$). Our key observations reveal: (i) When $osly \geq 4$, the relative errors decrease to approximately 0.01 for $H=1/32$ and 0.001 for $H=1/16$. (ii) The errors are essentially independent of the contrast ratios. Specifically, for $H=1/16$ and $osly=4$, the relative errors for the stress field exhibit minimal variation ($\leq 0.0001$) across four orders of magnitude in material contrast ($E_1/E_2 \in [10^3, 10^6]$).

The stress field components are compared in Figure~\ref{128_stress_test_model_2}, where panels (a)-(c) display the reference solution and panels (d)-(f) show the multiscale solution obtained with $H=1/16$, $osly=4$, $Nbf=6$. The remarkable visual correspondence between these solutions validates the accuracy of our multiscale method.

\subsection{Nearly incompressible material}
\label{incompressible}
In this section, we investigate the performance of the mixed CEM-GMsFEM for nearly incompressible materials using Test Model~1 through five representative cases:
(i) Following \cite{EricCS2018}, we first examine $E_1=10^9$, $E_2=1$ with $\nu_1=0.49$, $\nu_2=0.35$, where only the yellow region is nearly incompressible. 
(ii) We consider $E_1=10^9$, $E_2=10^6$ with $\nu_1=0.49$, $\nu_2=0.45$, where both materials are nearly incompressible. (iii) We test $E_1=10^9$, $E_2=10^6$ with $\nu_1=0.499$, $\nu_2=0.35$, where the yellow region represents nearly incompressible material. (iv) We inverted the parameters in (iii) (i.e. let $E_1=10^6$, $E_2=10^9$ with $\nu_1=0.35$, $\nu_2=0.499$) to make the blue region nearly incompressible. (v) We test a uniformly nearly incompressible case with $E_1=E_2=10^9$ and $\nu_1=\nu_2=0.499$.

Table~\ref{nearly incompressible material error} shows that for $H=1/16$ and $Nbf=6$, the stress field relative errors consistently decrease to approximately 0.01 ($osly=3$) and 0.001 ($osly=4$) across all cases, demonstrating the method's robustness against locking for nearly incompressible materials.

\begin{table}
	\footnotesize
	\centering
	\caption{Relative errors for the stress in Test model 2 with different contrast ratios, $osly$ under Neumann boundary condition ($H=1/32,Nbf=6$).}
	\label{Relative errors for the stress of Test model 2 with H=1/32}
	\begin{tabular}{ccccc}
		\hline
		$osly$ $\backslash$ $E_1/E_2$ & $10^3$ & $10^4$ & $10^5$ & $10^6$\\
		\hline
		1 & 0.873831143	&0.482291397	&0.794938264	&0.778907196\\
2 &0.591931422	&0.254178986	&0.544829757	&0.520740495\\
3 &0.237786694	&0.153638656	&0.222141042	&0.297681871\\
4 &0.082300483	&0.091634274	&0.082806962	&0.077018736\\
5 &0.030234297	&0.064231915	&0.057788169	&0.058073445
		\\
		\hline
	\end{tabular}
	%\label{tbl:table1}
\end{table}

\begin{table}
	\footnotesize
	\centering
	\caption{Relative errors for the stress in Test model 2 with different contrast ratios, $osly$ under Neumann boundary condition ($H=1/16,Nbf=6$).}
	\label{Relative errors for the stress of Test model 2 with H=1/16}
	\begin{tabular}{ccccc}
		\hline
		$osly$ $\backslash$ $E_1/E_2$ & $10^3$ & $10^4$ & $10^5$ & $10^6$\\
		\hline
        1 & 0.22188172	&0.28063983	&0.207984778	&0.223849945\\
		2 &0.138836717	&0.197568323	&0.078167688	&0.132833547\\
3 &0.018643521	&0.037715201	&0.020196882	&0.021841288\\
4 &0.003449089	&0.00344533	 &0.003444889	&0.003444841
		\\
		\hline
	\end{tabular}
	%\label{tbl:table1}
\end{table}
\begin{figure}[tbph] 
		\centering 
		\includegraphics[height=5.5cm,width=10cm]{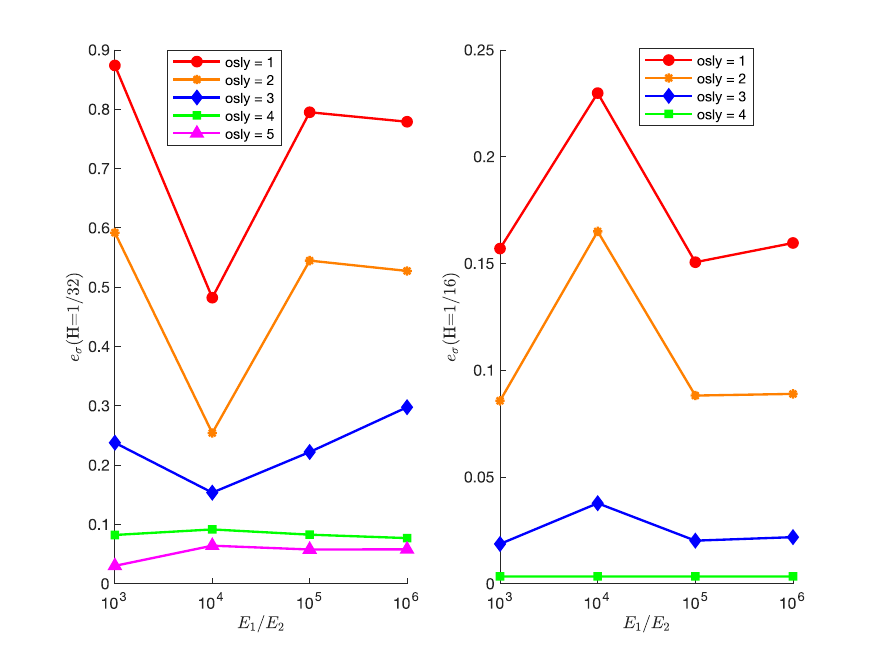} 
		\caption{Relative errors for the stress in Test model 2 with different $osly$, $E_1/E_2$ under Neumann boundary condition ($Nbf=6$, $H=1/32$(left), $H=1/16$(right))}
		\label{Relative errors stress of test model 2}
\end{figure}
\begin{figure}[tbph] 
		\centering 
		\includegraphics[height=8cm,width=15cm]{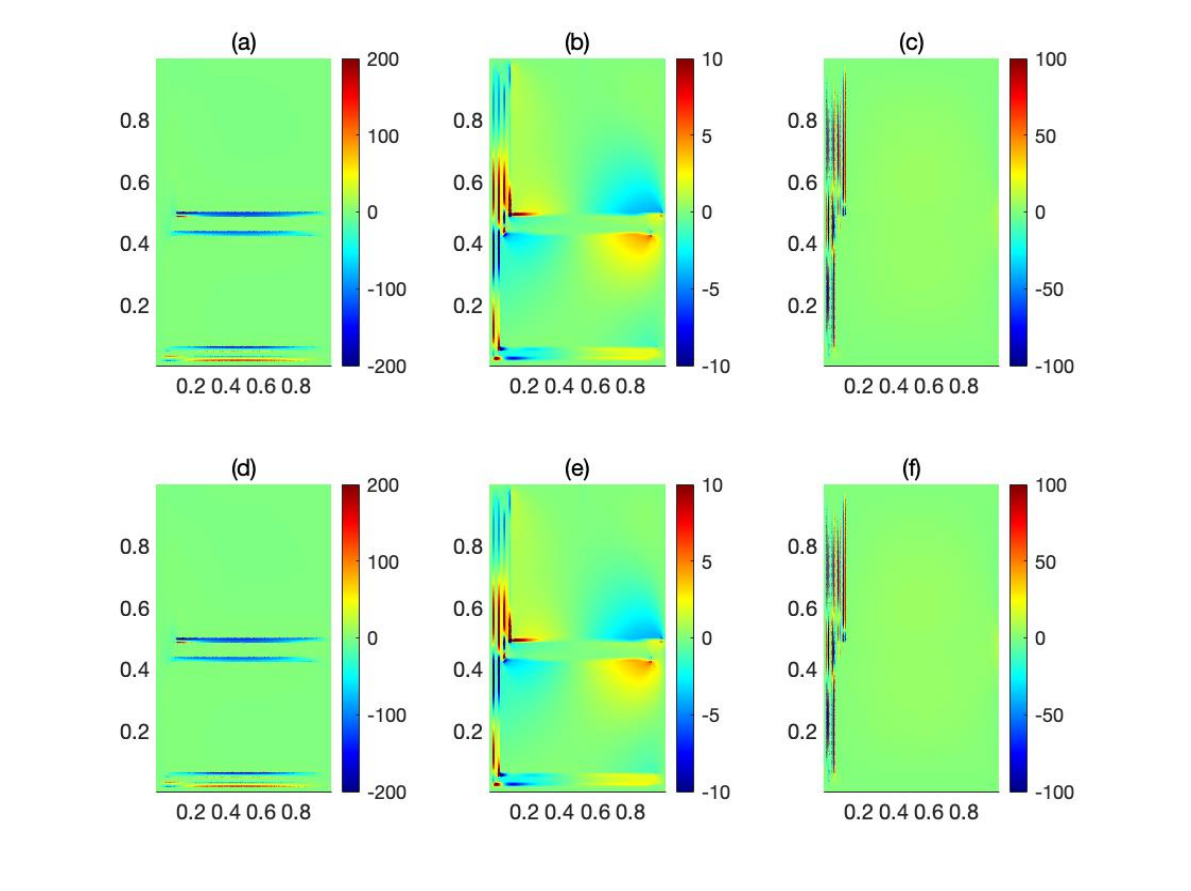} 
		\caption{(a)-(c): components of reference solution for the stress in Test model 2 (i.e. $(\underline{\bm{\sigma}}_h)_{11},(\underline{\bm{\sigma}}_h)_{12},(\underline{\bm{\sigma}}_h)_{22}$); (d)-(f): components of multiscale solution for the stress in Test model 2 (i.e. $(\underline{\bm{\sigma}}_{\text{ms}})_{11},(\underline{\bm{\sigma}}_{\text{ms}})_{12},(\underline{\bm{\sigma}}_{\text{ms}})_{22}$) with $H=1/16$, $osly=4$, $Nbf=6$ (under Neumann boundary condition)}
		\label{128_stress_test_model_2}
	\end{figure}

\begin{table}
	\footnotesize
	\centering
	\caption{Relative errors for the stress in Test model 1 with different $osly$, $E_1/E_2$, $\nu_1,\nu_2$ under Neumann boundary condition ($H=1/16,Nbf=6$).}
	\label{nearly incompressible material error}
	\begin{tabular}{cccccc}
		\hline
		$osly$ & \thead{$E_1/E_2=10^9/1$\\($\nu_1=0.49$,\\$\nu_2=0.35$)} & \thead{$E_1/E_2=10^9/10^6$\\($\nu_1=0.49$,\\$\nu_2=0.45$)}& \thead{$E_1/E_2=10^9/10^6$\\($\nu_1=0.499$,\\$\nu_2=0.35$)}& \thead{$E_1/E_2=10^6/10^9$\\($\nu_1=0.35$,\\$\nu_2=0.499$)}& \thead{$E_1=E_2=10^9$\\($\nu_1=\nu_2=0.499$)} \\
		\hline
        1 & 0.417208656	&0.455281971 &0.25215308 &0.15415484 &0.09268605\\
		2 &0.20188172	&0.228563268 &0.08899044 &0.06140178 &0.06503662\\
3 &0.056838802	&0.014550162&0.05118221 &0.04708935 &0.04963241\\
4 &0.009879735	&0.009859791&0.00491394 &0.00490894 &0.00544204
		\\
		\hline
	\end{tabular}
	%\label{tbl:table1}
\end{table}

\subsection{Numerical experiments with different boundary conditions}
\label{different bdcs}
In this section, we perform numerical experiments to evaluate the performance of our proposed multiscale method for both Dirichlet and mixed boundary conditions. For clarity, we focus on Test Model 1 with source term $\mathbf{f} = (1,1)^t$. Similar to the experiments in Sections \ref{test model 1}-\ref{test model 2}, we investigate contrast ratios $E_1/E_2 \in \{10^3, 10^4, 10^5, 10^6\}$.
We present detailed numerical results including the relative errors and convergence rates for both stress and displacement fields. These comprehensive comparisons provide rigorous validation of the method's accuracy and demonstrate its effectiveness across different boundary scenarios.

\subsubsection{Numerical verification of Dirichlet boundary condition}
\label{Dirichlet case}
In this section, we present numerical validation of the proposed multiscale method under homogeneous Dirichlet boundary conditions with $\bm{u} = \bm{0}$ on $\partial\Omega$.

The relative errors for both stress and displacement fields, presented in Table~\ref{errors_dirichlet_1}, reveal several important results. For the stress field, the relative errors decrease to 0.01 when $osly$ reaches 3, while the displacement field achieves similar accuracy at $osly=2$ and maintains stable performance for $osly=3,4,5$. Notably, the relative errors show remarkable independence from the contrast ratios, with stress errors remaining stable at $osly=4$ and displacement errors at $osly=3$ across all tested $E_1/E_2$ values.

The method's performance was further examined for nearly incompressible materials through three representative test cases. First, we considered $E_1=10^9$, $E_2=10^6$ with $\nu_1=0.499$, $\nu_2=0.35$, where the yellow region represents nearly incompressible material. Second, we inverted these parameters (i.e. let $E_1=10^6$, $E_2=10^9$ with $\nu_1=0.35$, $\nu_2=0.499$) to make the blue region nearly incompressible. Third, we tested a uniformly nearly incompressible case with $E_1=E_2=10^9$ and $\nu_1=\nu_2=0.499$. As shown in Table~\ref{nearly incompressible material error dirichlet}, the method demonstrates consistent robustness against locking in all scenarios.

Next, we analyze the convergence behavior with respect to different mesh parameters. First, we maintain a fixed fine mesh size $h=1/64$ while appropriately  increasing the oversampling size as $H$ decreases. By examining the convergence rate as a function of the coarse mesh size $H$ (shown in the first subfigures of Figures.~\ref{dirichlet_stress_rate_test_model_1}--\ref{dirichlet_displacement_rate_test_model_1}), we observe nearly first-order convergence ($O(H)$). This result aligns perfectly with Theorem~\ref{final convergence} and Remark~\ref{specific convergence order}, which establish that when $osly = O(\log(\max\{k\}/H^2))$ (meaning the oversampling size increases appropriately as $H$ decreases), the multiscale solution $(\underline{\bm{\sigma}}_{\text{ms}}, \mathbf{u}_{\text{ms}})$ achieves $O(H)$ convergence to the reference solution $(\underline{\bm{\sigma}}_h, \mathbf{u}_h)$.
Second, maintaining a fixed ratio $h/H$ while proportionally increasing the oversampling size as $H$ decreases, we examine the numerical convergence for different fine mesh size (for instance, taking $h=1/8,1/16,1/32,1/64$).
Observing the second subfigures of Figures.~\ref{dirichlet_stress_rate_test_model_1}--\ref{dirichlet_displacement_rate_test_model_1}, we find first-order convergence of $h$. This behavior again confirms the theory in Theorem~\ref{final convergence}.  
Third, when simultaneously refining both $h$ and $H$ (shown in the third subfigures of Figures.~\ref{dirichlet_stress_rate_test_model_1}--\ref{dirichlet_displacement_rate_test_model_1}), we discover more interesting convergence characteristics. With proper oversampling selection, the method demonstrates superlinear convergence rates, outperforming the standard linear case.

To further validate our method, we compare the solution components in detail. Figure~\ref{dirichlet_stress_test_model_1} (a)-(c) display the reference stress components $(\underline{\bm{\sigma}}_h)_{11}$, $(\underline{\bm{\sigma}}_h)_{12}$, and $(\underline{\bm{\sigma}}_h)_{22}$, while (d)-(f) show the multiscale solution computed with $H=1/16$, $osly=4$, and $Nbf=6$. The visual agreement is remarkable, with the multiscale solution capturing all essential features of the reference solution. Similar excellent agreement is observed for the displacement field in Figure~\ref{dirichlet_u_test_model_1}, demonstrating the method's accuracy.

\subsubsection{Numerical verification of mixed boundary conditions}
\label{mixed case}
We now present numerical results under mixed boundary conditions. The computational domain features homogeneous Dirichlet conditions $\mathbf{u}=\mathbf{0}$ on $\Gamma_D\coloneqq\{(x,y)\colon x=0, 0\leq y\leq 1\} \cup \{(x,y)\colon y=0, 0\leq x\leq 1\}$ and zero-traction Neumann conditions $\underline{\bm{\sigma}}\cdot\mathbf{n}=\mathbf{0}$ on $\Gamma_N\coloneqq \{(x,y)\colon x=1, 0\leq y\leq 1\} \cup \{(x,y)\colon y=1, 0\leq x\leq 1\}$, where $\mathbf{n}$ denotes the outward unit normal vector.

Numerical experiments reveal important convergence characteristics, as documented in Table~\ref{errors_mixed_1}. For the stress field, relative errors decrease to 0.01 when $osly$ reaches 4, while the displacement field achieves comparable accuracy at $osly=3$ and maintains this level for $osly=4,5$. Notably, these errors remain essentially independent of the contrast ratio $E_1/E_2$ for both fields, demonstrating the method's robustness across a wide range of material properties.

The method maintains robust performance for nearly incompressible materials, as demonstrated by three representative test cases following Section~\ref{Dirichlet case}. As shown in Table~\ref{nearly incompressible material error mixed bdc}, the stress field achieves a relative error of 0.01 at $osly=3$, while the displacement field reaches the same accuracy at $osly=2$ across all test cases. These results confirm the method's effectiveness in preventing locking phenomena.

Then we examine three distinct refinement strategies to observe the convergence behavior under mixed boundary conditions. First, fixing $h$ and refining $H$ (while properly increasing oversampling size), the initial subfigures of Figures~\ref{mixed_stress_rate_test_model_1}--\ref{mixed_displacement_rate_test_model_1} demonstrate nearly first-order ($O(H)$) convergence for both stress and displacement fields. This strongly corroborates the theory established in Theorem~\ref{final convergence}.
Second, fixing $h/H$ and refining $h$, we observe consistent first-order convergence with respect to $h$ in both fields, as shown in the second subfigures of Figures~\ref{mixed_stress_rate_test_model_1}--\ref{mixed_displacement_rate_test_model_1}. 
Third, simultaneous refinement of both $h$ and $H$ reveals superlinear convergence rates in the third subfigures of Figures~\ref{mixed_stress_rate_test_model_1}--\ref{mixed_displacement_rate_test_model_1}.

Visual comparison of solution components provides additional validation of the method's accuracy. Figure~\ref{mixed_stress_test_model_1} (a)-(c) display the reference stress solution components, while (d)-(f) present the multiscale approximation. The remarkable agreement between these solutions is immediately apparent. Similar agreement is observed for the displacement field in Figure~\ref{mixed_displacement_test_model_1}, further confirming the method's effectiveness under mixed boundary conditions.
\subsection{Comparison of computational costs}\label{times}
This section presents the computational times for the fine-scale solution ($h=1/64$) and our multiscale solutions with varying coarse mesh size $H$, conducted for $E_1=10^3,E_2=1$, $Nbf=3$. Recall that we performed the computations in MATLAB 2021a on a Lenovo ThinkCentre M80q Gen 4 desktop with an Intel Core i9-13900T processor and 32GB RAM. Using sparse matrices and backslash operators in MATLAB, we compute the solutions $\mathbf{A}_\mathrm{h}\backslash\mathbf{b}$ (fine grid) and $\mathbf{A}_\mathrm{ms}\backslash\mathbf{b}$ (multiscale), respectively. As evidenced by the timing data in Table~\ref{times for two methods}, the multiscale method achieves remarkable computational savings, reducing the solution time from $2.9741\,\mathrm{s}$ to $0.8766\,\mathrm{s}$ if $H=1/32$ (even to $0.0036\,\mathrm{s}$ when $H=1/8$) while maintaining accuracy.
\begin{table}
	\footnotesize
	\centering
	\caption{Relative errors for the stress and displacement in Test model 1 with different contrast ratios and $osly$  under Dirichlet boundary condition ($H=1/32$, $Nbf=6$).}
	\label{errors_dirichlet_1}
	\begin{tabular}{cccccc}
		\hline
		&$osly$ & $E_1/E_2=10^3$ & $E_1/E_2=10^4$ & $E_1/E_2=10^5$ & $E_1/E_2=10^6$ \\
		\hline
		\multirow{5}*{$e_\sigma$} & 1 &0.604044591	&0.551467072	&0.417198378	&0.63535944
		\\
		%\hline
		&2 &0.250463654	&0.196275653	&0.124823274	&0.270098938
		\\
		%\hline
		&3 &0.092000516	&0.075985727	&0.047348559	&0.099633791
		\\
		%\hline
		&4 & 0.03896118	&0.038138652	&0.022410658	&0.037956072
		\\
		%\hline
		&5 &0.022716932	&0.023074105	&0.018575064	&0.021942052
		\\
		\hline
		\multirow{5}*{$e_u$} & 1 &0.432701653	&0.35581177	&0.21585298	&0.475181482
		\\
		%\hline
		&2 &0.086811271	&0.065467727	&0.05410848	&0.098506169
		\\
		%\hline
		&3 &0.044886257	&0.048678703	&0.05040011	&0.045425112
		\\
		%\hline
		&4 &0.04361247	&0.048099417	&0.050278713	&0.043653213
		\\
		%\hline
		&5 &0.043535989	&0.048050545	&0.050272542	&0.043579891
		\\
		\hline	
	\end{tabular}
	%\label{tbl:table1}
\end{table}

\begin{table}
	\footnotesize
	\centering
	\caption{Relative errors for the stress and displacement in Test model 1 with different $osly$, $E_1/E_2$, $\nu_1,\nu_2$ for nearly incompressible materials  under Dirichlet boundary condition($H=1/16,Nbf=6$).}
	\label{nearly incompressible material error dirichlet}
	\begin{tabular}{ccccc}
		\hline
		&$osly$ & \thead{$E_1/E_2=10^9/10^6$\\($\nu_1=0.499$,\\$\nu_2=0.35$)}& \thead{$E_1/E_2=10^6/10^9$\\($\nu_1=0.35$,\\$\nu_2=0.499$)}& \thead{$E_1=E_2=10^9$\\($\nu_1=\nu_2=0.499$)} \\
		\hline
       \multirow{4}*{$e_\sigma$} & 1  &0.33708693 &0.28235139 &0.55820290\\
		&2  &0.11867062 &0.13830502 &0.27645047\\
&3 &0.04937808 &0.06514989 &0.08926279\\
&4 &0.03313717 &0.01109502 &0.02403516
		\\
		\hline
        \multirow{4}*{$e_u$} & 1  &0.22564673 &0.34472815 &0.31380541\\
		&2  &0.08429602 &0.04610617 &0.06690731\\
&3 &0.08129079 &0.04307457 &0.06382159\\
&4 &0.07937435 &0.03154018 &0.05819446
		\\
		\hline
	\end{tabular}
	%\label{tbl:table1}
\end{table}

\begin{figure}[tbph] 
		\centering 
		\includegraphics[height=4cm,width=11cm]{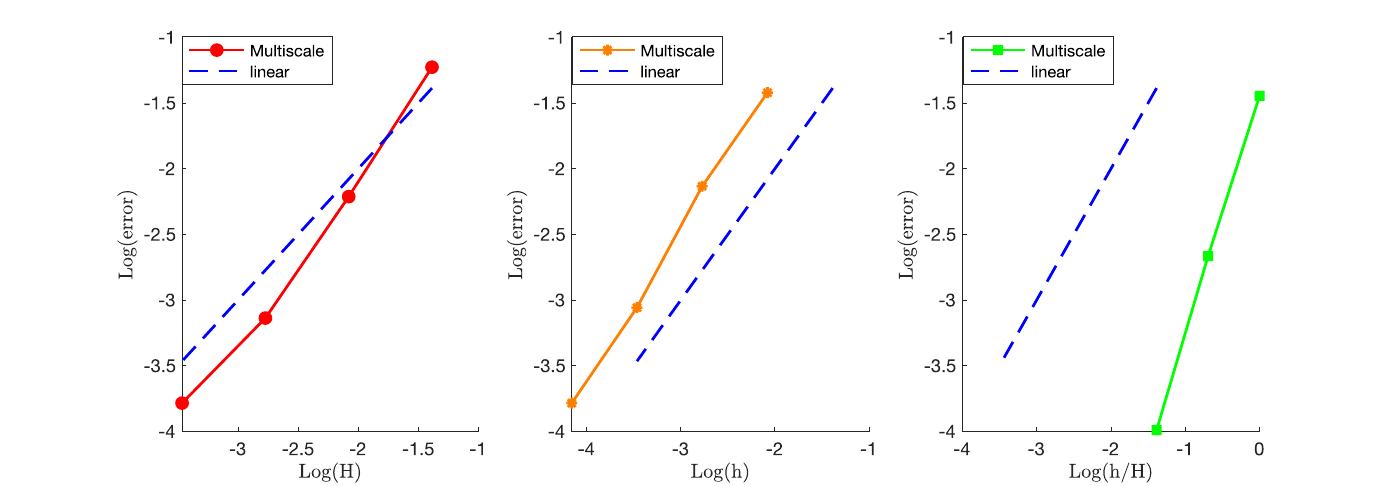} 
		\caption{Convergence rates for the stress approximation (i.e. $e_\sigma$) with respect to $H,h,h/H$ in Test model 1 under Dirichlet boundary condition}
		\label{dirichlet_stress_rate_test_model_1}
\end{figure}

\begin{figure}[tbph] 
		\centering 
		\includegraphics[height=4cm,width=11cm]{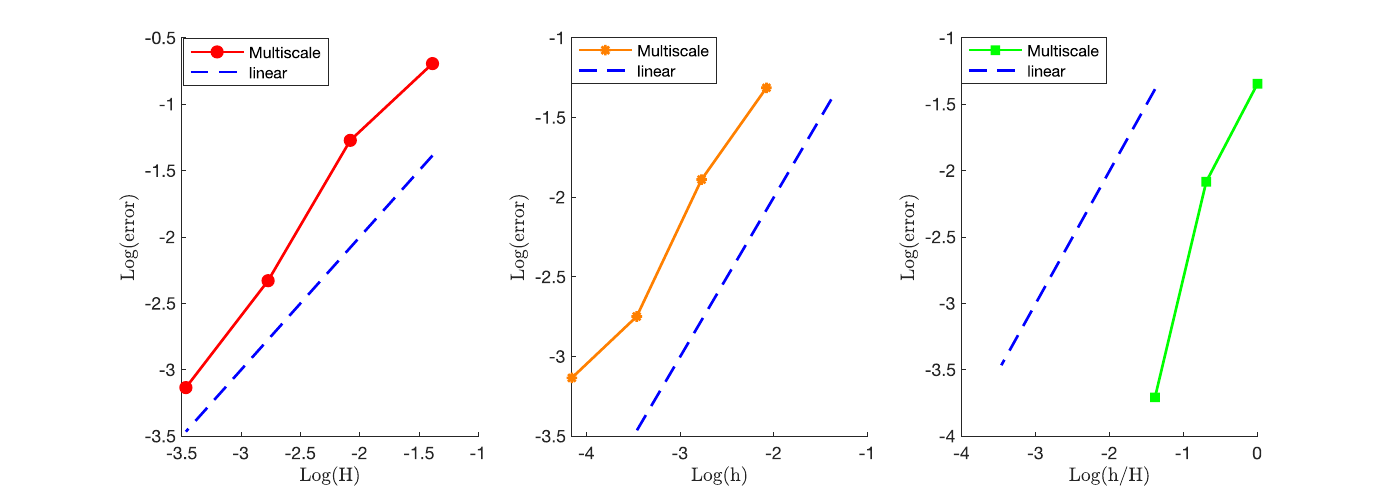} 
		\caption{Convergence rates for the displacement approximation (i.e. $e_u$) with respect to $H,h,h/H$ in Test model 1 under Dirichlet boundary condition}
	\label{dirichlet_displacement_rate_test_model_1}
\end{figure}

\begin{figure}[tbph] 
		\centering 
		\includegraphics[height=8cm,width=15cm]{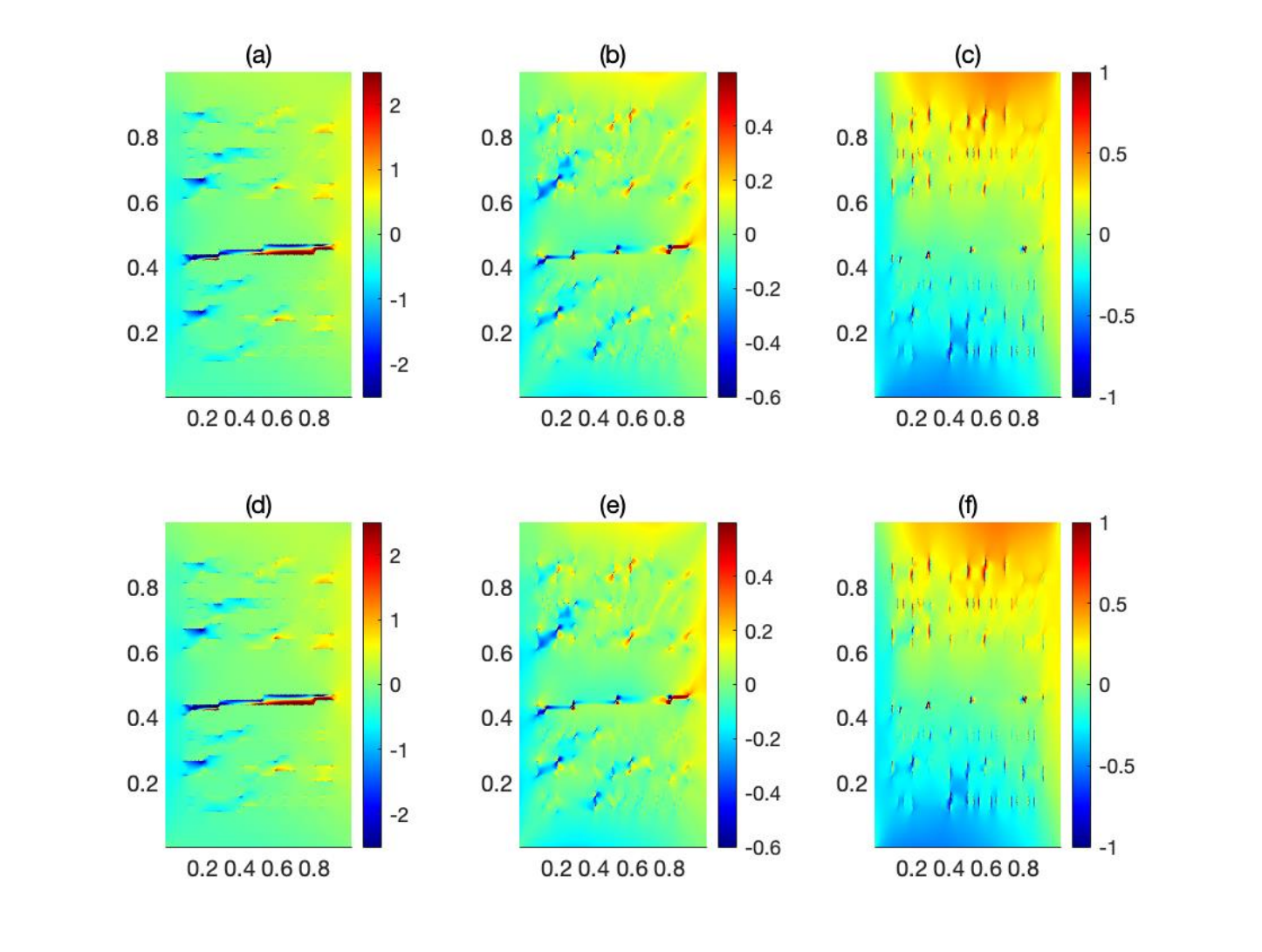} 
		\caption{(a)-(c): components of reference solution for the stress in Test model 1 (i.e. $(\underline{\bm{\sigma}}_h)_{11},(\underline{\bm{\sigma}}_h)_{12},(\underline{\bm{\sigma}}_h)_{22}$); (d)-(f): components of multiscale solution for the stress in Test model 1 (i.e. $(\underline{\bm{\sigma}}_{\text{ms}})_{11},(\underline{\bm{\sigma}}_{\text{ms}})_{12},(\underline{\bm{\sigma}}_{\text{ms}})_{22}$) with $H=1/16$, $osly=4$, $Nbf=6$ (under Dirichlet boundary condition)}
		\label{dirichlet_stress_test_model_1}
\end{figure}

\begin{figure}[tbph] 
		\centering 
		\includegraphics[height=8cm,width=10cm]{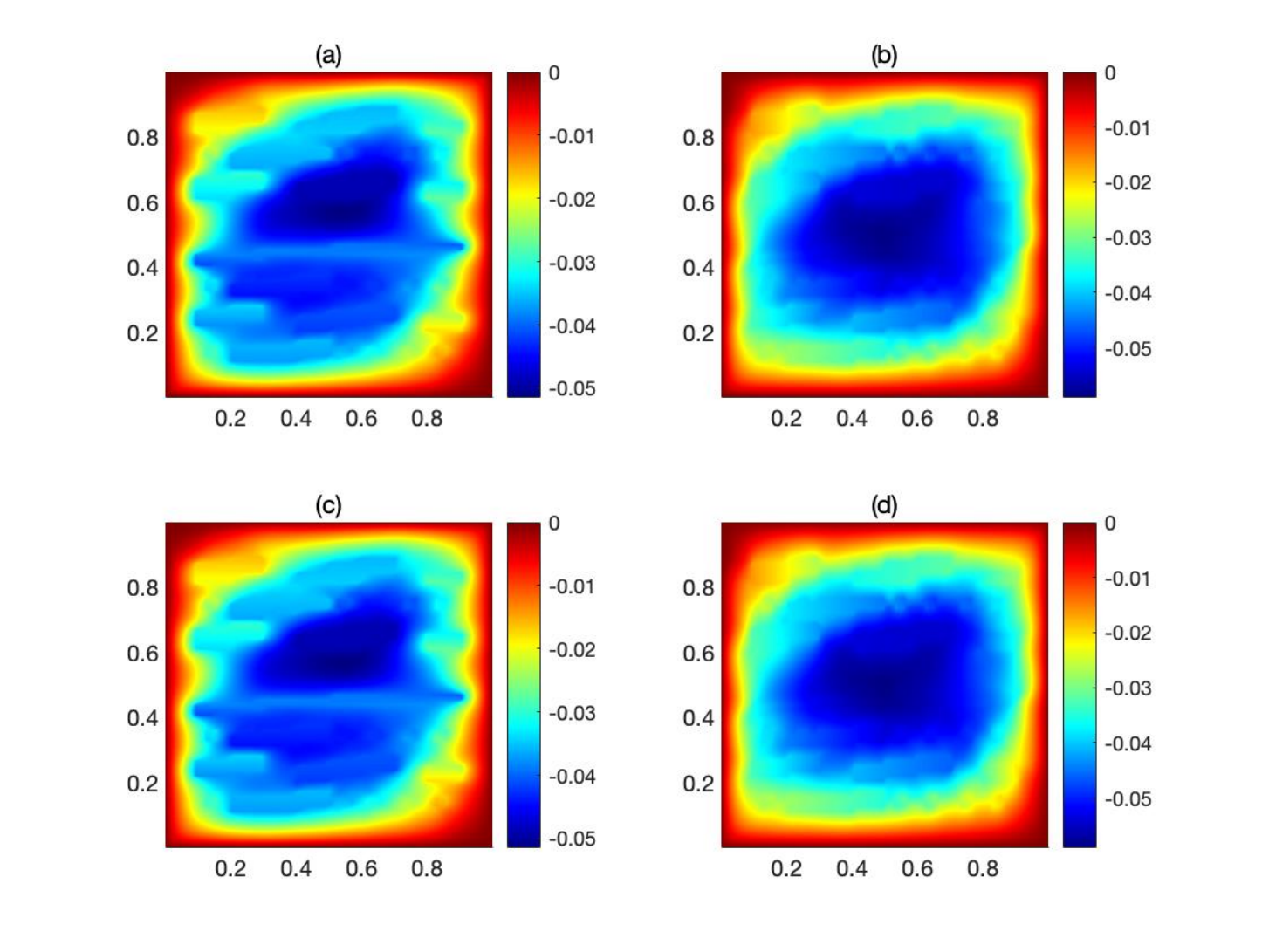} 
		\caption{(a)-(b): components of reference solution for the displacement in Test model 1 (i.e. $(\mathbf{u}_h)_1,(\mathbf{u}_h)_2$); (c)-(d): components of multiscale solution for the displacement in Test model 1 (i.e. $(\mathbf{u}_\text{ms})_1,(\mathbf{u}_\text{ms})_2$) with $H=1/16$, $osly=4$, $Nbf=6$ (under Dirichlet boundary condition)}
		\label{dirichlet_u_test_model_1}
\end{figure}

\begin{table}
	\footnotesize
	\centering
	\caption{Relative errors for the stress and displacement in Test model 1 with different contrast ratios and $osly$  under mixed boundary conditions ($H=1/32$, $Nbf=6$).}
	\label{errors_mixed_1}
	\begin{tabular}{cccccc}
		\hline
		&$osly$ & $E_1/E_2=10^3$ & $E_1/E_2=10^4$ & $E_1/E_2=10^5$ & $E_1/E_2=10^6$ \\
		\hline
		\multirow{5}*{$e_\sigma$} &1	&0.852949121	&0.895234867	&0.871283462	&0.822068717\\
&2	&0.224284775	&0.305484824	&0.368776604	&0.358221741\\
&3	&0.084602695	&0.113294977	&0.135120858	&0.131942142\\
&4	&0.031697465	&0.040939858	&0.047107308	&0.044554545\\
&5	&0.013066804	&0.016778756	&0.018355102	&0.017717853
\\
		\hline
		\multirow{5}*{$e_u$} & 1	&0.842137025	&0.731930609	&0.825051426	&0.618965915\\
&2	&0.064074511	&0.105684447	&0.154687533	&0.153606729\\
&3	&0.033456064	&0.03583056	&0.038294754	&0.037819102\\
&4	&0.032322308	&0.032802352	&0.033072001	&0.033032182\\
&5	&0.032279773	&0.032693031	&0.032899875	&0.03289858\\
		\hline	
	\end{tabular}
	%\label{tbl:table1}
\end{table}

\begin{table}
	\footnotesize
	\centering
	\caption{Relative errors for the stress and displacement in Test model 1 with different $osly$, $E_1/E_2$, $\nu_1,\nu_2$ for nearly incompressible materials  under mixed boundary conditions ($H=1/16,Nbf=6$).}
	\label{nearly incompressible material error mixed bdc}
	\begin{tabular}{ccccc}
		\hline
		&$osly$ & \thead{$E_1/E_2=10^9/10^6$\\($\nu_1=0.499$,\\$\nu_2=0.35$)}& \thead{$E_1/E_2=10^6/10^9$\\($\nu_1=0.35$,\\$\nu_2=0.499$)}& \thead{$E_1=E_2=10^9$\\($\nu_1=\nu_2=0.499$)} \\
		\hline
       \multirow{4}*{$e_\sigma$} & 1	&0.390441481	&0.467685769	&0.538097959\\
&2	&0.12785988	&0.165486743	&0.163951884\\
&3	&0.040211609	&0.018320269	&0.022125696\\
&4	&0.039605258	&0.01569492	&0.019164127\\
		\hline
        \multirow{4}*{$e_u$} &1	&0.255761615	&0.348563272	&0.42064666\\
&2	&0.033556266	&0.068500455	&0.044506193\\
&3	&0.032021729	&0.059232702	&0.038119042\\
&4	&0.031460695	&0.028048208	&0.018440514\\
		\hline
	\end{tabular}
	%\label{tbl:table1}
\end{table}

\begin{figure}[tbph] 
		\centering 
		\includegraphics[height=4cm,width=11cm]{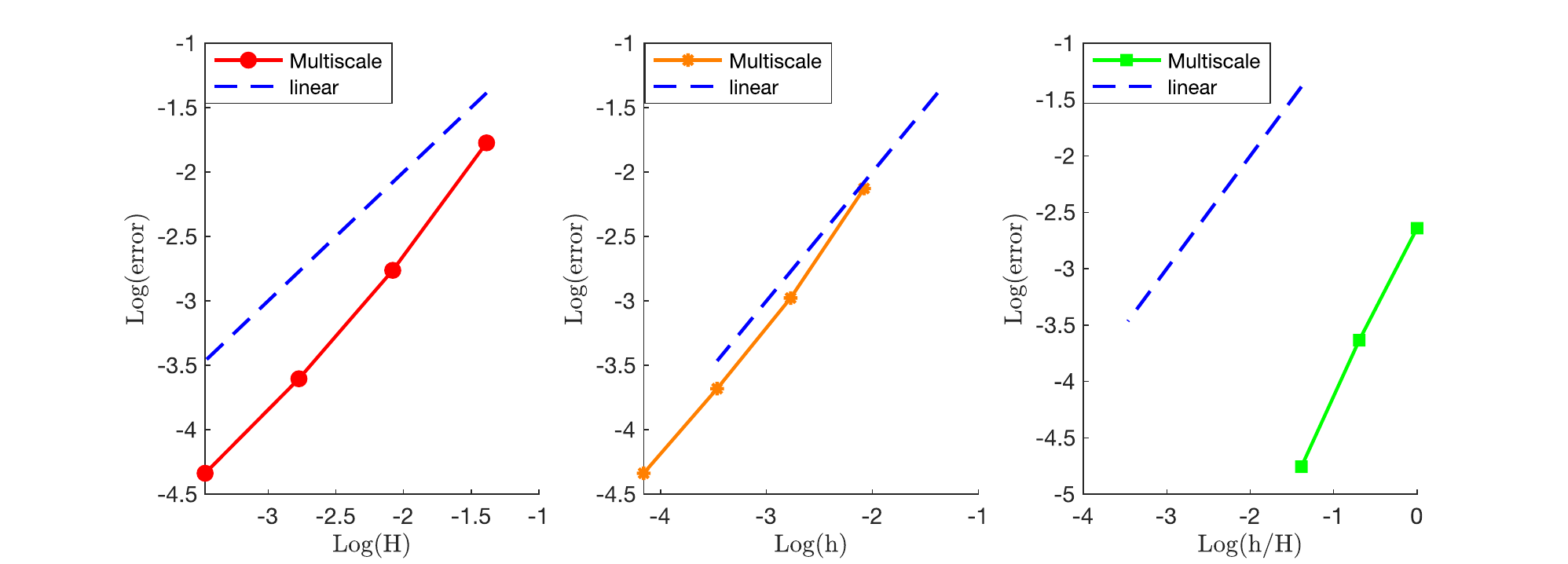} 
		\caption{Convergence rates for the stress approximation (i.e. $e_\sigma$) with respect to $H,h,h/H$ in Test model 1 under mixed boundary conditions}
		\label{mixed_stress_rate_test_model_1}
\end{figure}

\begin{figure}[tbph] 
		\centering 
		\includegraphics[height=4cm,width=11cm]{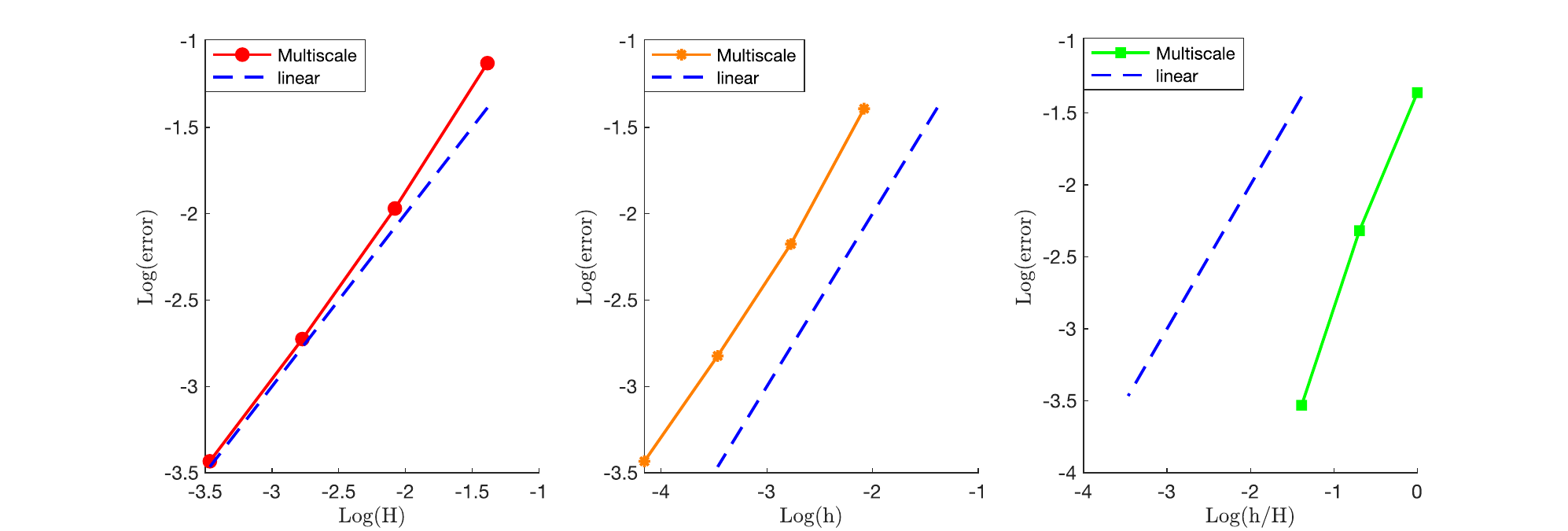} 
		\caption{Convergence rates for the displacement approximation (i.e. $e_u$) with respect to $H,h,h/H$ in Test model 1 under mixed boundary conditions}
	\label{mixed_displacement_rate_test_model_1}
\end{figure}

\begin{figure}[tbph] 
		\centering 
		\includegraphics[height=8cm,width=15cm]{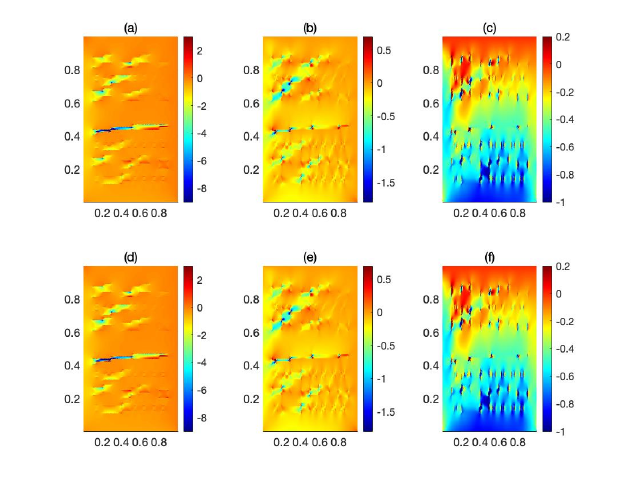} 
		\caption{(a)-(c): components of reference solution for the stress in Test model 1 (i.e. $(\underline{\bm{\sigma}}_h)_{11},(\underline{\bm{\sigma}}_h)_{12},(\underline{\bm{\sigma}}_h)_{22}$); (d)-(f): components of multiscale solution for the stress in Test model 1 (i.e. $(\underline{\bm{\sigma}}_{\text{ms}})_{11},(\underline{\bm{\sigma}}_{\text{ms}})_{12},(\underline{\bm{\sigma}}_{\text{ms}})_{22}$) with $H=1/16$, $osly=4$, $Nbf=6$ (under mixed boundary conditions)}
		\label{mixed_stress_test_model_1}
\end{figure}

\begin{figure}[tbph] 
		\centering 
		\includegraphics[height=8cm,width=10cm]{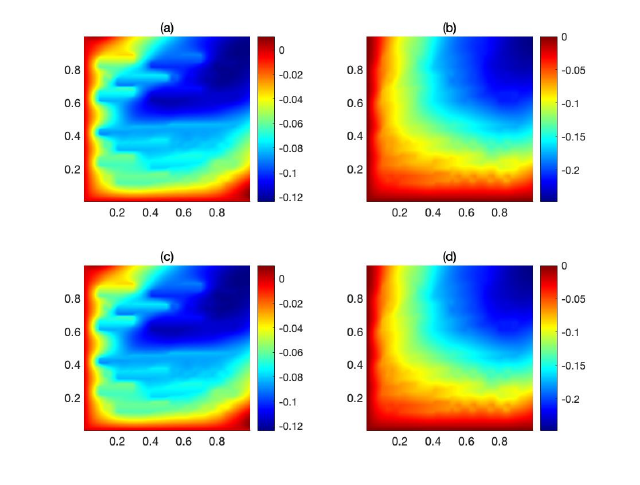} 
		\caption{(a)-(b): components of reference solution for the displacement in Test model 1 (i.e. $(\mathbf{u}_h)_1,(\mathbf{u}_h)_2$); (c)-(d): components of multiscale solution for the displacement in Test model 1 (i.e. $(\mathbf{u}_\text{ms})_1,(\mathbf{u}_\text{ms})_2$) with $H=1/16$, $osly=4$, $Nbf=6$ (under mixed boundary conditions)}
		\label{mixed_displacement_test_model_1}
\end{figure}

\begin{table}[ht]
 \footnotesize
    \centering
	\caption{Computational cost for the fine grid solution ($h=1/64$) and multiscale solutions with different $H$ ($E_1/E_2=10^3$).}
	\label{times for two methods}
	\begin{tabular}{ccccc}
		\hline
		 & $h=1/64$  & $H=1/32$ & $H=1/16$ & $H=1/8$ \\
		\hline
         DOFs &221952 &6144 &1536 & 384 \\
		times (s) &2.9741 & 0.8766 & 0.0616 & 0.0036\\
		\hline
	\end{tabular}
	%\label{tbl:table1}
\end{table}

\section{Conclusions}
\label{conclusions}
In this work, we have introduced a locking‐free and robust multiscale method for linear elasticity in the stress–displacement formulation with high‐contrast coefficients. 
The approach leverages a multiscale model‐reduction framework on a coarse mesh to mitigate computational cost. 
First, we solve local spectral problems on each coarse block to generate auxiliary multiscale basis functions for the displacement field. 
Next, we obtain the stress basis functions by enforcing an energy‐minimization constraint over enlarged oversampling regions. 
This mixed formulation yields direct stress approximations and remains locking‐free with respect to Poisson's ratio---an essential feature for nearly incompressible materials. 
We have established an $O(H)$ convergence rate that is independent of coefficient contrast, provided the oversampling regions grow appropriately. 
Finally, a suite of numerical experiments validates both the accuracy and robustness of the proposed method.

\section*{CRediT authorship contribution statement}

\textbf{Eric T. Chung:} Writing – review \& editing, Supervision, Resources, Methodology, Funding acquisition, Conceptualization. \textbf{Changqing Ye:} Writing – review \& editing, Methodology, Software, Conceptualization. \textbf{Xiang Zhong:} Writing – review \& editing, Writing – original draft, Visualization, Validation, Software, Resources, Methodology, Formal analysis, Data curation, Conceptualization.

\section*{Declaration of competing interest}

The authors declare that they have no known competing financial interests or personal relationships that could have appeared
to influence the work reported in this paper.

\section*{Declaration of Generative AI and AI-assisted technologies in the writing process}

During the preparation of this work the authors used ChatGPT in order to improve readability and language. After
using this tool, the authors reviewed and edited the content as needed and take full responsibility for the content of the
publication.

\section*{Acknowledgments}
Eric Chung’s research is partially supported by the Hong Kong RGC General Research Fund Projects 14305423 and 14305624.

\bibliographystyle{siamplain}
% Use this way to switch the bib file.
%\bibliography{references2}

\end{document}